\documentclass[11pt]{amsart}
\usepackage{pstricks,pst-plot}

\definecolor{Black}{cmyk}{0,0,0,1}
\definecolor{OrangeRed}{cmyk}{0,0.6,1,0}            
\definecolor{DarkBlue}{cmyk}{1,1,0,0.20}
\definecolor{myblue}{rgb}{0.66,0.78,1.00}
\definecolor{Violet}{cmyk}{0.79,0.88,0,0}
\definecolor{Lavender}{cmyk}{0,0.48,0,0}

\parskip=\smallskipamount

\newtheorem{theorem}{Theorem}[section]
\newtheorem{lemma}[theorem]{Lemma}
\newtheorem{corollary}[theorem]{Corollary}

\theoremstyle{definition}

\newtheorem{example}[theorem]{Example}

\newcommand{\z}{\zeta}

\newcommand{\bea}{\begin{eqnarray*}}
\newcommand{\eea}{\end{eqnarray*}}

\numberwithin{equation}{section}

\begin{document}

\title[SCV]{ Several complex variables\\}

%
%
\subjclass[2000]{}
\date{\today}

\subjclass[2000]{}
\date{\today}

\author{John Erik Forn\ae ss}
\address{Department for Mathematical Sciences\\
Norwegian University of Science and Technology\\
Trondheim, Norway}
\email{john.fornass@math.ntnu.no}


\maketitle

\centerline{Course for talented undergraduates} 
\centerline{Beijing University, Spring Semester 2014}

\tableofcontents

(cf. \cite{FS:DynamicsI})

\section{Introduction}

These lectures will give an introduction to several complex variables. We will generally follow the classical book by Hormander, {\bf An Introduction to complex analysis in several variables.} The notes will add some more details to the text of Hormander, especially after the introductory material. We generally follow the numbering of results as in Hormanders book, but results in Hormander might be broken up into smaller steps, for example, Lemma 4.1.3 in Hormanders book is broken up into 4.1.$3_a$ to 4.1.$3_j$
in these notes. 

\noindent We follow Chapter 1.1, 1.2 and 1.6\\
and then Chapter 2.1, 2.2, 2.5 and 2.6.\\

Afterwards we move to Chapter 4.1, 4.2. We deviate a little by considering $L^p$ spaces for general $p$ for a while, before restricting to $L^2$ spaces. We prove the existence of solutions to $\overline{\partial}$ on pseudoconvex domains in $\mathbb C^n$ in $L^2_{loc}$ and give the solution of the Levi problem. At the end we go through the recent proof by 
Bo-Yong Chen of the Ohsawa-Takegoshi Theorem. This requires first that we discuss a version of Hormanders solution of the $\overline{\partial}$ equation in $L^2$ as in Theorem 1.1.4 in the 1965 Acta paper of Hormander.\\

Some additional material that is not needed for the presentation are in the Appendices.
For example, the solution of $\overline{\partial}$ in a polydics of section 2.3 is in an appendix,
since it is not needed in the proof of the $L^2$ theorem of Hormander. There are also some remarks
on $L^p$ spaces there, such as Ohsawa-Takegoshi in $L^p$ and the strong openness conjecture in $L^p$. \\

The author thanks the Beijing International Center for Mathematical Research, BICMR,
for its hospitality during the Spring Semester of 2014.\\

The course can also be downloaded from\\

\noindent http://www.bicmr.org/news/2014/0221/1386.html \\
1.7may2014.pdf\\

Preliminaries for the course is some knowledge of one complex variable and some functional analysis.\\

\section{Hormander, Section 1.1-1.2}


Chapter I: Analytic functions of one complex variable.\\

1.1 Preliminaries: \\

Here holomorphic (=analytic) functions are introduced. These are $\mathcal C^1$ functions on domains in $\mathbb C$ which satisfy the Cauchy-Riemann equation

$$
\frac{\partial u}{\partial \overline{ z}}=0$$

The set of all such functions is denoted by $A(\Omega).$

Here

\bea
\frac{\partial u}{\partial  z} & = & \frac{1}{2}\left(\frac{\partial u}{\partial x}-i\frac{\partial u}{\partial y}\right)\\
\frac{\partial u}{\partial \overline{ z}} & = & \frac{1}{2}\left(\frac{\partial u}{\partial x}+i\frac{\partial u}{\partial y}\right)\\
\eea

Generally,
$$du=\frac{\partial u}{\partial x}dx+\frac{\partial u}{\partial y}dy=\frac{\partial u}{\partial  z}dz+\frac{\partial u}{\partial \overline{z}}d\overline{z}.$$

For analytic functions, $du=\frac{\partial u}{\partial  z}dz$, i.e. $du$ and $dz$ are paralell.
For analytic functions we write $\frac{\partial u}{\partial z}=u'.$

Examples of analytic functions are all polynomials $P(z)=\sum_i a_i z^i$ and the exponential function $e^z=e^x(\cos y+i\sin y)$.
We have that products, compositions and inverses of analytic functions are analytic.

1.2 Cauchy's integral formula and its applications.\\

Let $\omega$ be a bounded open set in $\mathbb C$ with boundary consisting of finitely many $\mathcal C^1$ Jordan curves. For ease of reference, we list the results using the same numbering as in Hormander. The proofs can be read easily in Hormanders book.

Cauchy integral formula for general functions:\\

{\bf Theorem 1.2.1} 
Let $u\in \mathcal C^1(\overline{\omega}).$ Then
for $\zeta\in \Omega,$
$$
u(\zeta)=\frac{1}{2\pi i}\left(\oint_{\partial \omega}\frac{u(z)}{z-\zeta}dz+
\int\int_\omega \frac{\frac{\partial u}{\partial \overline{z}}}{z-\zeta}dz\wedge d\overline{z}\right)
$$

{\bf Theorem 1.2.2 }
If $\mu$ is a measure with compact support in $\mathbb C$, the integral
$$
u(\zeta)=\int \frac{d\mu(z)}{z-\zeta}
$$
defines an analytic $\mathcal C^\infty$ function outside the closed support of $\mu$. In any open set
$\omega$ where $d\mu=\frac{\phi (z)dz\wedge d\overline{z}}{2\pi i}$ for some $\phi\in \mathcal C^k(\omega),k \geq 1,$ we have that $u\in \mathcal C^k(\omega)$ and $\frac{\partial u}{\partial \overline{z}}=\phi.$

{\bf Corollary 1.2.3}
Every $u\in A(\Omega)$ is in $\mathcal C^\infty$. Also $u'\in A(\Omega)$ if $u\in A(\Omega).$

{\bf Theorem 1.2.4}
For every compact set $K\subset \Omega$ and every open neighborhood $\omega\subset \Omega$ of $K$ there are constants $C_j,j=0,1,\dots$ such that 
$$
(1.2.4) \sup_{z\in K}|u^{(j)}(z)| \leq C_j \|u\|_{L^1(\omega)}$$
for all $u\in A(\Omega)$, where $u^{(j)}=\frac{\partial^j u}{\partial z^j}.$

{\bf Exercises}\\
Recall that a function $f(z)=u(x,y)+iv(x,y)$ is analytic if $f\in \mathcal C^1$ and $\frac{\partial f}{\partial \overline{z}}=0.$\\
1) Show that the equation $\frac{\partial f}{\partial \overline{z}}=0$ is equivalent to the classical Cauchy Riemann equations 
\bea
\frac{\partial u}{\partial x} & = & \frac{\partial v}{\partial y} \\
\frac{\partial u}{\partial y} & = & -\frac{\partial v}{\partial x} \\
\eea

2) Show that the function defined by $f(z)=e^{-(z^{-4})}$ for $z\neq 0$ and $f(0)=0$
satisfies the Cauchy Riemann equations at every point. 
Is $f\in \mathcal C^1?$\\

3) Show that 
\bea
\overline{\left(\frac{\partial f}{\partial z}\right)} & = & \frac{\partial \overline{f}}{\partial \overline{z}}\\
& \mbox{and} & \\
\overline{\left(\frac{\partial f}{\partial \overline{z}}\right)} & = & \frac{\partial \overline{f}}{\partial {z}}\\
\eea

\newpage

\section{Hormander, 1.2,1.6}

{\bf Corollary 1.2.5}
If $u_n\in A(\Omega)$ and $u_n\rightarrow u$ when $n\rightarrow \infty$, uniformly on compact subsets of $\Omega$, it follows that $u\in A(\Omega).$ 

\begin{proof}
Pick a point $\zeta\in \Omega$ and choose a disc $\overline{\Delta}(w,r)\subset \Omega$
with $\zeta\in \Delta(w,r).$
For each $n,$ we have by theorem 1.2.1. that 
$$
u_n(\zeta)=\frac{1}{2\pi i}\int_{|z-w|=r}\frac{u_n(z)}{z-\zeta}dz.
$$
Using the uniform convergence we then get $u$ is continuous and 
$$
u(\zeta)=\frac{1}{2\pi i}\int_{|z-w|=r}\frac{u(z)}{z-\zeta}dz.
$$
From this formula we see that $u$ is $\mathcal C^1$ and analytic on $\Delta(w,r).$

\end{proof}

{\bf Corollary 1.2.6}
If $u_n\in A(\Omega)$ and the sequence $|u_n|$ is uniformly bounded on every compact subset of $\Omega$, there is a subsequence $u_{n_j}$ converging uniformly on every compact subset of $\Omega$ to a limit $u\in A(\Omega).$

{\bf Corollary 1.2.7}
The sum of a power series $\sum_0^\infty a_nz^n$ is analytic in the interior of the circle of convergence.

{\bf Theorem 1.2.8}
If $u$ is analytic in $\Omega=\{z;|z|<r\}$, we have 
$$
u(z)=\sum_0^\infty u^{(n)}(0)z^n/n!
$$
with uniform convergence on every compact subset of $\Omega.$

Uniqueness of analytic continuation:\\

{\bf Corollary 1.2.9}
If $u\in A(\Omega)$ and there is some point $z\in \Omega$ where $u^{(k)}(0)=0$ for all $k\geq 0$,
it follows that $u=0$ in $\Omega$ if $\Omega$ is connected.

{\bf Corollary 1.2.10}
If $u$ is analytic in the disc $\Omega=\{z;|z|<r\}$ and if $u$ is not identically $0$, one can write $u$ in one and only onoe way in the form
$$
u(z)=z^nv(z)
$$
where $n\geq 0$ is an integer and $v\in A(\Omega)$, $v(0)\neq0$ (which means that $1/v$ is also analytic in a neighborhood of $0.$

{\bf Theorem 1.2.11}
If $u$ is analytic in $\{z;|z-z_0|<r\}=\Omega$ and if $|u(z)|\leq |u(z_0)|$ when $z\in \Omega$, then $u$ is constant in $\Omega.$

Maximum Principle:\\

{\bf Corollary 1.2.12}
Let $\Omega$ be bounded and let $u\in \mathcal C(\overline{\Omega})$ be analytic in $\Omega.$ Then the maximum of $|u|$ in $\overline{\Omega}$ is attained on the boundary.

1.6 Subharmonic Functions.\\

Definition: A $\mathcal C^2$ function is said to be harmonic if
$\Delta h= 4\frac{\partial^2h}{\partial z\partial \overline{z}}=0$. This is equivalent to 
the equation $\frac{\partial^2 h}{\partial x^2}+\frac{\partial^2 h}{\partial x^2}=0.$
If $h=\mbox{Re}(f),$ f analytic, then $h$ is harmonic:
\bea
\frac{\partial^2 h}{\partial z\partial \overline{z}} & = & \frac{\partial^2 (\frac{f+\overline{f}}{2})}{\partial z\partial \overline{z}}\\
& = & \frac{1}{2}\frac{\partial^2 f}{\partial z\partial \overline{z}}+\frac{1}{2}\frac{\partial^2 \overline{f}}{\partial z\partial \overline{z}}\\
& = & \frac{1}{2}\frac{\partial}{\partial z}\left(\frac{\partial f}{\partial \overline{z}}\right)+\frac{1}{2}\frac{\partial}{\partial \overline{z}}\left(\frac{\partial \overline{f}}{\partial {z}}\right)\\
& = & \frac{1}{2}\frac{\partial}{\partial \overline{z}}\overline{\left(\frac{\partial {f}}{\partial \overline{z}}\right)}\\
& = & 0\\
\eea
Conversely, suppose that $u$ is $\mathcal C^2$ and harmonic on a disc $D$.
Then for each Jordan curve $\gamma\in D$ bounding a domain $U,$ we get by Stokes theorem that $\oint_\gamma (-u_ydx+u_xdy)=\int_U u_{xx}+u_{yy}=0.$ Hence the function $v(q)=\int_{z_0}^q
-u_ydx+u_xdy$ defines a function $v$ on $D$. This function satisfies $v_x=-u_y$ and
$v_y=u_x$ so $u+iv$ is analytic. 
So on a disc, we see that $h$ is harmonic if and only if $h$ is the real part of an analytic function.\\

{\bf Definition 1.6.1}
A function $u$ defined in an open set $\Omega\subset \mathbb C$ with values in $[-\infty,\infty>$
is called subharmonic if \\
a) $u$ is upper semicontinuous, that is $\{z\in \Omega;u(z)<c\}$ is open for every real number $c$.\\
b) For every compact $K\subset \Omega$ and every continuous function $h$ on $K$ which is harmonic in the interior of $K$ and such that $u\leq h$ on $\partial K$, we have $u\leq h$ on $K.$ \\

Note: The function $u\equiv -\infty$ is called subharmonic in this text.\\
Note: A function is $u$ upper semicontinuous if and only if there exists a sequence of continuous functions $u_j$ such that $u_j\searrow u.$\\

\newpage

\section{Hormander 1.6}

Note: Harmonic functions are subharmonic: Let $u$ be harmonic on $\Omega$ and let $K$ be a compact subset of $\Omega.$ Let $h$ be a continuous function on $K$ which is harmonic on the interior of $K$. Also suppose that $u\leq h$ on $\partial K.$ Suppose there exists a point $p\in$ the interior $K$ such that $u(p)>h(p).$ Then if $\epsilon>0$ is small enough we have that the function
$v=u-h+\epsilon |z|^2$ satisfies $v(p)>\sup_{\partial K} v$. We can choose $p$ to be a point where
$v$ takes on a maximum value. Note however that $v_{xx}+v_{yy}=4\epsilon>0$. This contradicts that $p$ is a max point.

{\bf Theorem 1.6.2}
If $u$ is subharmonic and $0<c\in \mathbb R$, it follows that $cu$ is subharmonic.
If $u_\alpha, \alpha\in A,$ is a family of subharmonic functions, then $u=\sup_\alpha{u_\alpha}$
is subharmonic if $u<\infty$ and $u$ is upper semicontinuous, which is always the case if $A$ is finite. If $u_1,u_2,\dots$ is a decreasing sequence of subharmonic functions, then $u=\lim_{j\rightarrow \infty}u_j$ is subharmonic.

{\bf Theorem 1.6.3}
Let $u$ be defined with values in $[-\infty,\infty>$ and assume that $u$ is upper semicontinuous. Then each of the following conditions are necessary and sufficient for $u$ to be subharmonic:\\
(i) If $D$ is a closed disc in $\Omega$ and $f$ is an analytic polynomial such that $u\leq {\mbox{Re}}(f)$ on $\partial D$, then it follows that $u\leq {\mbox{Re}}(f)$ on $D.$\\
(ii) If $\Omega_\delta=\{z\in \Omega; d(z,\Omega^c)>\delta\}$, we have
$$
(1.6.1)\;\; u(z)2\pi \int d\mu(r)\leq \int_0^{2\pi}\int u(z+re^{i\theta})d\theta d\mu(r), z\in \Omega_\delta$$
for every positive finite measure $d\mu$ on the interval $[0,\delta].$\\
(iii) For every $\delta>0$ and every $z\in \Omega_\delta$ there exists some positive finite measure
$d\mu$ with support in $[0,\delta]$ such that $d\mu$ has some mass outside the origin and (1.6.1)
is valid.\\

As pointed out by one the students in class, statement (ii) is not the right one. The correct statement
should be

(ii) Let $z\in \Omega$ and assume that $\{z+re^{i \theta};\theta\in [0,2\pi], 0\leq r\leq s\}\subset \Omega.$ Then (1.6.1) holds where $\mu$ is supported on $[0,s].$\\

{\bf Corollary 1.6.4}
If $u_1,u_2$ are subharmonic, then $u_1+u_2$ is subharmonic

\begin{proof}
We use (1.6.1).
\end{proof}

{\bf Corollary 1.6.5}
A function $u$ defined in an open set $\Omega\subset \mathbb C$ is subharmonic if every point in $\Omega$ has a neighborhood on which $u$ is subharmonic.

We use property (iii) of Theorem 1.6.3.

{\bf Corollary 1.6.6}
If $f\in A(\Omega),$ then $\log |f|$ is subharmonic.

\begin{proof}
Use property (i) in Theorem 1.6.3.
So suppose that $\log |f|\leq \mbox{Re}(g)$ on the boundary of a disc, where $g$ is a holomorphic polynomial.
Then $|f|\leq e^{\mbox{Re}(g)}=|e^g|$ on the boundary of the disc. Hence
$|fe^{-g}|\leq 1$ on the boundary of the disc. Hence also on the inside.
Therefore $\log|f|\leq \mbox{Re}(g)$ on the whole disc.
\end{proof}

{\bf Theorem 1.6.7}
Let $\phi$ be a convex increasing function on $\mathbb R$ and set $\phi(-\infty)=\lim_{x\rightarrow -\infty} \phi(x).$ Then $\phi(u)$ is subharmonic if $u$ is subharmonic.

{\bf Definition}
A function $\phi(x)$ is convex if for every $a<b$, and every $t\in <0,1>$,
$\phi(ta+(1-t)b)\leq tf(a)+(1-t)f(b),$ i.e. the graph lies under the chord.

Observation: An immediate consequence is that out side $(a,b)$ the graph lies above the straight line continuing the chord.

{\bf Lemma 1.6.7a}
Suppose that $\phi(x)$ is convex and suppose that $x_0\in \mathbb R.$ Then there exists a constant $k\in \mathbb R$ so that
$\phi(x)\geq \phi(x_0)+k(x-x_0)$ for all $x.$ Also $\phi$ is continuous.

Proof of the Lemma. 
Take any sequence $a_n<x_0<b_n$ where both converge to $x_0.$ Let $k$ be any limit for the slopes of the chords from $a_n$ to $b_n$.\\

To prove continuity, suppose that $x_n\searrow x_0$. Fix $a<x_0<b.$ Considering the chord from
$a$ to $x_n$ shows that $\liminf \phi(x_n)\geq \phi(x_0).$ Considering the chord from $x_0$ to $b$ shows that the $\limsup \phi(x_n)\leq \phi(x_0).$ A similar argument applies for $x_n \nearrow x_0.$

Proof of theorem 1.6.7:\\
Let $x_0\in \mathbb R$ and let $k$ be as in Lemma 1.6.7a.
Let $x=u(z+re^{i\theta}).$
Then
$$
\phi(u(z+re^{i\theta})) \geq \phi(x_0)+ k(u(z+re^{i\theta})-x_0).
$$
Hence 
$$
\frac{1}{2\pi } \int_0^{2\pi}\phi(u(z+re^{i\theta}))d\theta \geq \phi(x_0)+ k(\frac{1}{2\pi}\int_0^{2\pi}u(z+re^{i\theta})d\theta-x_0).
$$
We want to show that $$
\frac{1}{2\pi } \int_0^{2\pi}\phi(u(z+re^{i\theta}))d\theta \geq \phi(u(z)).
$$
If $\frac{1}{2\pi}\int_0^{2\pi}u(z+re^{i\theta})d\theta =-\infty,$ this is clear.
So assume the integral is finite.

Let $x_0=\frac{1}{2\pi}\int_0^{2\pi}u(z+re^{i\theta})d\theta.$
Then 

$$
\frac{1}{2\pi } \int_0^{2\pi}\phi(u(z+re^{i\theta}))d\theta \geq \phi(\frac{1}{2\pi}\int_0^{2\pi}u(z+re^{i\theta})d\theta).
$$

Since $\frac{1}{2\pi}\int_0^{2\pi}u(z+re^{i\theta})d\theta\geq u(z)$ and since $\phi$ is increasing,
$\phi(\frac{1}{2\pi}\int_0^{2\pi}u(z+re^{i\theta})d\theta)\geq \phi(u(z)).$

Hence
$$
\frac{1}{2\pi } \int_0^{2\pi}\phi(u(z+re^{i\theta}))d\theta \geq \phi(u(z)).
$$

It follows from Theorem 1.6.3 (iii) that $\phi(u)$ is subharmonic.\\

{\bf Corollary 1.6.7b}
If $u$ is subharmonic, then $e^{u}$ is subharmonic. If $f$ is analytic, $|f|$ is subharmonic.

The first part follows from Theorem 1.6.7 since $e^x$ is a convex increasing function. Since $\log |f|$
is subharmonic by Corollary 1.6.6, it follows that $|f|=e^{\log |f|}$ also is subharmonic.

{\bf Corollary 1.6.8}
Let $u_1,u_2$ be nonnegative and assume that $\log u_j$ is subharmonic in $\Omega.$
Then $\log (u_1+u_2) $ is subharmonic.

{\bf Theorem 1.6.9}
Let $u$ be subharmonic in the open set $\Omega$ and not identically $-\infty$ in any connected component of $\Omega.$ Then $u$ is integrable on all compact subsets of $\Omega$ (we write $u\in L^1_{loc}(\Omega)$), which implies that $u>-\infty$ almost everywhere.
 
\begin{proof}
Suppose that $u(z)>-\infty.$ Pick a closed disc $D$ centered at $z$ contained in $\Omega.$
If we let $\mu=rdr$ we get from (1.6.1) that
$u(z)A\leq \int_D u dA.$ Note that $u$ is bounded above on $D.$ Hence it follows that $u$ is in $L^1$ on $D.$ It follows that $u>-\infty$ a.e. on $D$. Hence we can repeat the argument for points $z$ near the boundary of $D.$ It follows that the set $U$ of points $z\in \Omega$ where $u$ is integrable in some neighborhood, is open and closed. By hypotheses $U$ is nonempty. Hence $U=\Omega.$
\end{proof}

{\bf Exercises}\\

1) Show that the function $\sup_{0<\epsilon<1} \epsilon \log |z|$ fails to be subharmonic.\\

2) Suppose that $u$ is a $\mathcal C^2$ subharmonic function on $\mathbb C$. Let $f:\Omega \rightarrow \mathbb C$ be an analytic function defined on an open set $\Omega$ in $\mathbb C.$
Show that $u\circ f$ is subharmonic on $\Omega.$\\

3) Let $u$ a subharmonic function in $\{|z|<2\}.$ Suppose that $u(z)=0$ for all $z,1<|z|<2.$
Show that $u\equiv 0.$\\

\section{Hormander 1.6, 2.1}

{\bf Theorem 1.6.10}
If $u$ is subharmonic in $\Omega$ and not $-\infty$ identically in any component of $\Omega,$
then we have that $$(1.6.3) \;\;\int u\Delta vd\lambda\geq 0$$ for all $v\in \mathcal C^2_0(\Omega)$ with $v\geq 0.$
Here $\lambda$ denotes Lebesgue measure.

\begin{proof}
Let $0<r<d(\mbox{supp}(v),\Omega^c).$ Then by 1.6.1 we have for every $z\in \mbox{supp}(v)$
that
$$
2\pi u(z) \leq \int_0^{2\pi}u(z+re^{i\theta})d\theta.
$$

Since $v\geq 0,$ we get for every $z\in \mbox{supp}(v)$ that
$$
2\pi u(z)v(z) \leq \int_0^{2\pi}u(z+re^{i\theta})v(z)d\theta.
$$

We integrate with respect to $\lambda.$

\bea
\int\int (2\pi u(z)v(z) )d\lambda & \leq  & \int\int(\int_0^{2\pi}u(z+re^{i\theta})v(z)d\theta)d\lambda\\
& = & \int_0^{2\pi}(\int\int u(z+re^{i\theta})v(z)d\lambda)d\theta\\
& = & \int_0^{2\pi}(\int\int u(z)v(z-re^{i\theta})d\lambda)d\theta\\
& = & \int \int u(z)(\int_0^{2{\pi}}v(z-re^{i\theta}) d\theta)d\lambda\\
\eea

We can also rewrite the left side:

\bea
\int\int (2\pi u(z)v(z) )d\lambda & = & \int\int u(z)(\int_0^{2\pi}v(z) d\theta)d\lambda\\
\eea

Hence we see that

$$
\int \int u(z)(\int_0^{2{\pi}} (v(z-re^{i\theta})-v(z))d\theta)d\lambda \geq 0.
$$

We Taylor expand the intergrand $v(z-re^{i\theta})-v(z)$.

\bea
v(z-re^{i\theta})-v(z) & = &- v_x(z)r\cos \theta-v_y(z)r\sin \theta
+\frac{1}{2}v_{xx} r^2 \cos^2\theta\\
& + & \frac{1}{2}v_{yy} r^2\sin^2\theta
+v_{xy}r^2\cos\theta\sin\theta+o(r^2)\\
\eea

We hence get an expression

\bea
-\int \int u(z)(\int_0^{2{\pi}} v_x(z)r\cos \theta)d\theta)d\lambda
& - & 
\int\int u(z)(\int_0^{2\pi}v_y(z)r\sin \theta)d\theta)d\lambda\\
& + & 
\int\int u(z)(\int_0^{2\pi}\frac{1}{2}v_{xx} r^2 \cos^2\theta)d\theta)d\lambda\\
& + & 
\int\int u(z)(\int_0^{2{\pi}}\frac{1}{2}v_{yy} r^2\sin^2\theta)d\theta)d\lambda\\
& + & 
\int\int u(z)(\int_0^{2\pi}v_{xy}r^2\cos\theta\sin\theta)d\theta)d\lambda\\
& + & 
\int\int u(z)(\int_0^{2\pi}o(r^2))d\theta)d\lambda\\
& \geq &  0.
\eea

Hence after carrying out the inner integrals we see that

$$
\int\int u(z)(\frac{1}{2}v_{xx}\pi r^2+ \frac{1}{2} v_{yy}\pi r^2+ o(r^2))d\lambda \geq 0.
$$

If we divide by $\frac{\pi r^2}{2}$ and let $r\rightarrow 0,$ we see that
$$
\int \int u(z)\Delta v(z)\geq 0.$$
\end{proof}

{\bf Corollary 1.6.10a}
If $u$ is a $\mathcal C^2$ subharmonic function, then $\Delta u \geq 0.$\\

{\bf Proof.}  Let $v\geq 0$ be a compactly supported $\mathcal C^2$ function in the domain of $u.$
By (1.6.3), we have that $\int u\Delta v\geq 0.$ Integrating by parts twice, we see that
$\int v\Delta u\geq 0.$ Since this is valid for all compactly supported nonnegative $\mathcal C^2$ functions $v$, it follows that $\Delta u\geq 0.$

{\bf Theorem 1.6.11}
Let $u\in L^1_{loc}(\Omega)$ and assume that (1.6.3) holds. Then there is one and only one subharmonic function $U$ in $\Omega$ which is equal to $u$ almost everywhere. If $\phi$ is an integrable non-negative function of $|z|$ with compact support and $\int \phi=1$, we have for every $z\in \Omega$
$$
(1.6.4)\;\; U(z)=\lim_{\delta\rightarrow 0}\int u(z-\delta z')\phi(z')d\lambda(z').
$$

We will divide the proof into lemmas.

{\bf Lemma 1.6.11a}
Assume $u\in L^1_{loc}(\Omega).$ Let $\psi=\psi(|z|)$ be a nonnegative $\mathcal C^\infty$ function with compact support in the unit disc and $\int \psi=1.$ Then
the function 
$$
u_{\delta}(z):=\int u(z-\delta z') \psi(z')d\lambda(z')=\frac{1}{\delta^2}\int u(w)\phi(\frac{z-w}{\delta})d\lambda(w)
$$
is $\mathcal C^\infty$ in $\Omega_\delta$. If $V\subset\subset \Omega$, we have that
$\|u_\delta\|_{L^1(V_\delta)}\leq \|u\|_{L^1(V)}.$ Moreover $u_\delta\rightarrow u\in L^1$ on compact subsets.

\begin{proof}
The last equality is obtained by the change of variable, $w=z-\delta z'.$ The fact that $u\in \mathcal C^\infty$ follows from differentiation under the integral sign in the last integral.\\

The inequality in the $L^1$ norm follows by integration with respect to $z$ first.\\
To show the last statement, write $u=u_1+u_2$ where $u_1$ is continuous and $u_2$ has small $L^1$ norm. The convergence for $(u_1)_\delta$ is obvious. And the $L^1$ norm of $(u_2)_\delta$ is as small as we wish.
\end{proof}

{\bf Lemma 1.6.11b}
Suppose $U$ is subharmonic. Then (1.6.4) holds when we set $u=U$ in the integral on the right side.
In particular it follows that if two subharmonic functions are equal almost everywhere, they are identical.

\begin{proof}
It follows by (1.6.1) that for small $\delta$,
$$
U(z)\leq \int U(z-\delta z')\phi(z') d\lambda(z').
$$
By upper semicontinuity of $U$ it follows that the upper limit of the right side when $\delta\rightarrow 0$ is at $\leq U(z).$ Hence (1.6.4) holds with $u=U.$
\end{proof}

{\bf Lemma 1.6.11c}
Assume that $u\in \mathcal C^2(\Omega)$ and that $\Delta u\geq 0.$ Then $u$ is subharmonic.
Moreover $u_\delta\searrow u$.

\begin{proof}
Fix $z_0\in \Omega_\delta.$ Let $u_{z_0}(w)=\int_0^{2\pi }u(z_0+e^{i\theta}w)d\theta$ for $|w|<\delta$.
Then $u_{z_0}$ is $\mathcal C^2$ and $\Delta u_{z_0}\geq 0$. Moreover $u_{z_0}$ only depends on $|w|.$
We calculate the Laplacian of $u_{z_0}$ at points $w=x+iy,x\geq 0, y=0.$
We can write $u_{z_0}(x,y)=u_{z_0}(\sqrt {x^2+y^2},0)=g(\sqrt{x^2+y^2}).$
We get that $g''(x)+g'(x)/x\geq 0$ for $x>0$, so $xg''(x)+g'(x)\geq 0, x\geq 0.$ 
It follows that $xg'(x)$ increases. The value at $x=0$ is $0$, so $g'(x)\geq 0.$ 
So $g(x)$ is increasing. By Theorem 1.6.3, it follows that $u$ is subharmonic, and we also get that
$u_\delta$ decreases to $u_\delta$ when $\delta\rightarrow 0.$
\end{proof}

{\bf Lemma 1.6.11d}
Assume $u\in L^1_{loc}$ satisfies (1.6.3). Let $\psi=\psi(|z|)$ be a nonnegative $\mathcal C^\infty$ function with compact support in the unit disc and $\int \psi=1.$ Then
the function 
$$
u_{\delta}(z):=\int u(z-\delta z') \psi(z')d\lambda(z')
$$
is $\mathcal C^\infty$ and subharmonic in $\Omega_\delta$. Moreover $u_\delta\rightarrow u\in L^1$ on compact subsets.

\begin{proof}
Suppose that $u\in L^1_{loc}$ and that $\int u(z)\Delta v(z)\geq 0$ for all functions that are $\mathcal C^2$ with compact support and with $v\geq 0.$
Then it follows that also $u_{\delta}$ has this property. Then, by Lemma 1.6.11c it follows that 
$u_\delta $ is subharmonic. Also by Lemma 1.6.11a we have convergence in $L^1$ norm to $u.$
\end{proof}

{\bf Lemma 1.6.11e}
Let $u$ and $\psi$ be as in Lemma 1.6.11d.  Then if $\delta,\epsilon>0$,
\bea
& \mbox{Let}\; \epsilon \searrow 0:&\\
(u_\delta)_\epsilon & \searrow  & u_\delta\\
& \mbox{Let}\; \delta \searrow 0:&\\
(u_\epsilon)_\delta & \searrow & u_\epsilon\\
(u_\delta)_\epsilon & = & (u_\epsilon)_\delta \\
& \mbox{Let}\; \epsilon \searrow 0:&\\
u_\epsilon & \searrow & V\\
\eea

for some subharmonic function $V.$

\begin{proof}
We first show that $(u_\delta)_\epsilon\searrow u_\delta.$ 
By Lemma 1.6.11d we have that $u_\delta$ is $\mathcal C^\infty$ and subharmonic. Hence by Corollary 1.6.10.a it follows that $\Delta u_\delta\geq 0.$ Hence it follows by Lemma 1.6.11c that
$(u_{\delta})_{\epsilon}\searrow u_\delta.$ \\
The second limit holds for the same reason.\\

To prove the following equality, we see that
\bea
(u_\delta)_{\epsilon}(z) & = & \int u_\delta(z-\epsilon z') \psi(z')d\lambda(z')\\
& = & \int (\int u(z-\epsilon z'-\delta z'')\psi(z'')d\lambda(z''))\psi(z')d\lambda(z')\\
& = & (u_\epsilon)_\delta(z)\\
\eea

We show that $u_{\epsilon_1}(z)\geq u_{\epsilon_2}(z)$ if $\epsilon_1>\epsilon_2.$
We have shown that for each $\delta>0$,
$(u_{\delta})_{\epsilon_1}(z) \geq (u_{\delta})_{\epsilon_2}(z)$, hence
$(u_{\epsilon_1})_\delta(z)\geq (u_{\epsilon_2})_\delta(z)$. Hence by Lemma 1.6.11c,
$u_{\epsilon_1}(z)\geq u_{\epsilon_2}(z).$

Finally, by Theorem 1.6.2, the limit of $u_\epsilon$ is subharmonic.

\end{proof}

To finish the proof of Theorem 1.6.11, it suffices to note that by Lemma 1.6.11d the function $V$ obtained is equal to $u$ a.e. Also by Lemma 1.6.11b, we have that (1.6.4) holds for $V$ and so also for $u.$

Now we turn to Chapter 2, Elementary properties of functions of several complex variables.\\

2.1 Preliminaries.\\

We introduce coordinates in $\mathbb C^n$. Write $z=(z_1,z_2,\cdots z_n)$ where
each $z_j=x_j+iy_j.$ Let $u$ be a $\mathcal C^1$ function.
We introduce differential operators as in one variable.

\bea
\frac{\partial u}{\partial z_j} & = & \frac{1}{2}\frac{\partial u}{\partial x_j}-\frac{1}{2}\frac{\partial u}{\partial y_j}\\
\frac{\partial u}{\partial \overline{z}_j} & = & \frac{1}{2}\frac{\partial u}{\partial x_j}+\frac{1}{2}\frac{\partial u}{\partial y_j}\\
\eea

Then we get
$$
du= \sum_{j=1}^n \frac{\partial u}{\partial z_j}dz_j+\sum_{j=1}^n \frac{\partial u}{\partial \overline{z}_j}d\overline{z}_j.
$$

We write for short the first sum on the right as $\partial u$ and the last sum as $\overline{\partial}u$
We say that $\partial u$ is a form of type $(0,1)$ and $\overline{\partial}u$ a form of type $(0,1).$

{\bf Definition 2.1.1} A function $u\in \mathcal C^1(\Omega)$ is said to be analytic or holomorphic if
$du$ is of type $(1,0)$. Equivalently, $\overline\partial u=0$ and also equivalently, the function satisfies the Cauchy-Riemann equations in each variable separately which is again equivalent to saying that $u$ is analytic in each variable separately. The set is analytic functions on $\Omega$ is called $A(\Omega).$

If $I=(i_1,\dots,i_p)$ is a multiindex of integers between 1 and n we write
$dz^I=dz_{i_1}\wedge\cdots\wedge dz_{i_p}$, and we write $|I|=p.$ 

If $J=(j_1,\dots j_q)$ is a multiindex of integers between 1 and n we write
$d\overline{z}^J=d\overline{z}_{i_1}\wedge\cdots\wedge d\overline{z}_{i_q}$, and we write $|J|=q.$

The following expression is called a $(p,q)$ form
$f=\sum_{|I|=p,|J|=q} f_{I,J}dz^I\wedge d\overline{z}^J.$ It is OK to think of this as an expression without meaning.

\newpage

\section{Hormander 2.1, 2.2, 2.5}

Let $f=\sum_{I,J} f_{I,J}dz^I\wedge d\overline{z}^J$ be a $(p,q)$ form. We use the usual convention that
of two differentials, say  $dz_i$ and $dz_j$ are switched, the form changes sign. \\

We define the exterior differential $df$ as the form
$df=\sum_{I,J} df_{I,J}dz^I\wedge d\overline{z}^J$.
If we write $\partial f=\sum_{I,J} \partial f_{I,J}\wedge dz^I\wedge d\overline{z}^J$
and $\overline\partial f=\sum_{I,J} \overline\partial f_{I,J}\wedge dz^I\wedge d\overline{z}^J$
then $df=\partial f+\overline{\partial}f$. This writes $df$ as a sum of a $(p+1,q)$ form and a $(p,q+1)$ form.\\

We write $0=d^2f= \partial^2 f+(\partial\overline\partial+\overline{\partial}\partial)f+\overline{\partial}^2f$ and these have type (p+2,q),\\
(p+1,q+1) and (p,q+2) respectively so all three terms must vanish.

$$
\partial^2=0, \partial\overline\partial+\partial\overline{\partial}=0, \overline{\partial}^2=0.
$$

This implies that if $f$ is a (p,q+1) form and we want to solve the equation $\overline\partial u=f$ then is is necessary that $\overline{\partial} f=0.$\\

2.2 Applications of Cauchy's integral theorem in a polydisc.\\

Let $w=(w_1,\dots,w_n)$ be a point in $\mathbb C^n$ and let
 $r=(r_1,\dots,r_n)$ be positive numbers. We define the polydisc $D$ with center $w$ and polyradius $r$ to be the
set $D=D^n(w,r)=\{z\in \mathbb C^n; |z_j-w_j|<r_j, j=1,\dots, n\}.$
The set $\partial D_0:=\{|z_j-w_j|=r_j, j=1,\dots,n\}$ is called the distinguished boundary of $D.$

{\bf Theorem 2.2.1} Let $D$ be an open polydisc and let $u$ be in  $\mathcal C^1(\overline{D}).$
If $u$ is an analytic function of $z_j\in D^1(w_j,r_j)$ whenever the other variables are constant,
$|z_k-w_k|\leq r_k$, then
$$
(2.2.1)\;\;u(z)=\left(\frac{1}{2\pi i}\right)^n \int_{\partial_0D}\frac{u(\zeta_1,\dots,\zeta_n)d\zeta_1\dots d\zeta_n}{(\zeta_1-z_1)\cdots (\zeta_n-z_n)}
$$
for all $z\in D.$ Hence $u$ is $\mathcal C^\infty$ in $D.$

\begin{proof}
We use Cauchys integal formula inductively, i.e. Theorem 1.2.1.
We get then 
\bea
u(z) & = & \frac{1}{2\pi i} \int_{|\zeta_n-w_n|=r_n}\frac{u(z_1,\dots,z_{n-1},\zeta_n) d\zeta_n}{\zeta_n-z_n}\\
& = & \frac{1}{2\pi i} \int_{|\zeta_n-w_n|=r_n}\frac{(\frac{1}{2\pi i} \int_{|\zeta_{n-1}-w_{n-1}|=r_{n-1}}\frac{u(z_1,\dots,z_{n-2},\zeta_{n-1},\zeta_n) d\zeta_{n-1}}{\zeta_{n-1}-z_{n-1}}) d\zeta_n}{\zeta_n-z_n}\\
& = & \left(\frac{1}{2\pi i}\right)^2 \int_{|\zeta_{j}-w_{j}|=r_j, j=n-1,n}\frac{u(z_1,\dots,z_{n-2},\zeta_{n-1},\zeta_n)d\zeta_{n-1} d\zeta_n}{(\zeta_{n-1}-z_{n-1}) (\zeta_n-z_n)}\\
& = & \left(\frac{1}{2\pi i}\right)^3 \int_{|\zeta_{j}-w_{j}|=r_j, j=n-2,n-1,n}\frac{u(z_1,\dots,z_{n-3},\zeta_{n-2},\zeta_{n-1},\zeta_n)d\zeta_{n-2}d\zeta_{n-1} d\zeta_n}{(\zeta_{n-2}-z_{n-2}) \cdots(\zeta_n-z_n)}\\
& = & \cdots\\
& = & \left(\frac{1}{2\pi i}\right)^n \int_{\partial_0D}\frac{u(\zeta_1,\dots,\zeta_n)d\zeta_1\dots d\zeta_n}{(\zeta_1-z_1)\cdots (\zeta_n-z_n)}\\
\eea

The last statement follows by differentiation under the integral sign.

\end{proof}

{\bf Corollary 2.2.2}
If $u\in A(\Omega)$, then $u\in \mathcal C^\infty(\Omega)$ and all derivatives of $u$ are also analytic.

We use the following multiindex notation: 
We write $\alpha=(\alpha_1,\dots,\alpha_n)$ for nonnegative integers $\alpha_j.$ We call $\alpha$ a multiorder. Set $\alpha!:=\alpha_1!\cdots \alpha_n!.$
We define $\partial^\alpha=(\frac{\partial}{\partial z_1})^{\alpha_1}\cdots (\frac{\partial}{\partial z_n})^{\alpha_n},\overline\partial^\alpha=(\frac{\partial}{\partial \overline{z}_1})^{\alpha_1}\cdots (\frac{\partial}{\partial \overline{z}_n})^{\alpha_n}.$

{\bf Theorem 2.2.3}
 If $K$ is a compact subset of $\Omega$ and $K\subset \omega \subset\subset
\Omega$, then there exist contants $C_\alpha$ for all multiorders $\alpha$ so that
if $u\in A(\Omega)$ then
$$
(2.2.3)\;\;\sup_K |\partial^\alpha u|\leq C_\alpha \|u\|_{L^1(\omega)}.
$$

\begin{proof}
Assume  first that $K=\overline{D}^n(w,r), \omega=D^n(w,r')$ and $\Omega=D^n(w,r'')$
for multiradii $r=(s,\dots,s), r'=(s',\dots,s'), r''=(s'',\dots, s''), s<s'<s''.$
Let $C_j$ denote the constants in the one variable version of this theorem, Theorem 1.2.4.
Let $z\in K$. We get
\bea
|\partial^\alpha u(z)| & \leq & C_{\alpha_n} \int_{|\zeta_n-w_n|<s'}
|(\frac{\partial}{\partial z_1})^{\alpha_1}\cdots (\frac{\partial}{\partial z_{n-1}})^{\alpha_{n-1}}
u(z_1,\dots,z_{n-1},\zeta_n|)
d\lambda(\zeta_n)\\
& \leq & C_{\alpha_n} \int_{|\zeta_n-w_n|<s'}
|C_{\alpha_{n-1}}\\
& * & \int_{|\zeta_{n-1}-w_{n-1}|<s'}(\frac{\partial}{\partial z_1})^{\alpha_1}\cdots (\frac{\partial}{\partial z_{n-2}})^{\alpha_{n-2}}
u(z_1,\dots,z_{n-2},\zeta_{n-1},\zeta_n)|\\
& * & d\lambda(\zeta_{n-1})|
d\lambda(\zeta_n)\\
& \cdots &\\
& \leq & \Pi C_{\alpha_j} \int_\omega |u(\zeta)|d\lambda(\zeta).\\
\eea
 
To complete the proof we cover the compact set in the Theorem by finitely many such polydiscs.

\end{proof}

{\bf Corollary 2.2.4}
If $u_k\in A(\Omega)$ and $u_k\rightarrow u$ uniformly on compact subsets, then $u\in A(\Omega).$

\begin{proof}
This follows from the Cauchy integral formula:
$$
u_k(z)=\left(\frac{1}{2\pi i}\right)^n \int_{\partial_0D}\frac{u_k(\zeta_1,\dots,\zeta_n)d\zeta_1\dots d\zeta_n}{(\zeta_1-z_1)\cdots (\zeta_n-z_n)}
$$
We take limits on both sides. Then it follows that $u$ is analytic on $D.$ 
\end{proof}

{\bf Corollary 2.2.5}
If $u_k\in A(\Omega)$ and is a uniformly bounded sequence on any compact subset of $\Omega,$
then there is a subsequence $u_{k_j}$ which converges uniformly on compact subsets to a limit $u\in A(\Omega).$

\begin{proof}
By Theorem 2.2.3 the first derivatives of the $u_k$ are uniformly bounded on compact subsets. Hence, by Ascoli, there is a subsequence $u_{k_j}$ which converges uniformly on compact subsets to a limit $u$. By the previous corollary, $u\in A(\Omega).$
\end{proof}

Let $a_{\alpha}(z)$ be holomorphic functions in $\Omega.$ We say that $\sum_{\alpha}a_\alpha$ converges normally if $\sum_\alpha \sup_K |\alpha(z)|$ converges for each compact subset of $\Omega.$ In this case the sum $\sum_\alpha a_{\alpha}(z)$ is a holomorphic function on $\Omega.$

{\bf Theorem 2.2.6}
Suppose that $u$ is analytic in a polydisc $D(0,r), r=(r_1,\dots,r_n)$. Then
$$
u(z)=\sum_\alpha \frac{\partial^\alpha u(0)}{\alpha!} z^\alpha
$$ for every $z\in D$ and the convergence is normal.

\begin{proof} We use the formula for the sum of a geometric series for $\z\in \partial_0D$
and $z\in D.$
$$
\frac{1}{(\zeta_1-z_1)\cdots (\z_n-z_n)}=\sum_\alpha \frac{z^\alpha}{\z^\alpha \z_1\cdots \z_n}.
$$
The series converges normally in $D.$ If $u\in \mathcal C^1(\overline{D})$ then we get from 2.2.1 that

$$
u(z)=\left(\frac{1}{2\pi i}\right)^n\sum_\alpha \left(\int_{\z\in \partial_0D} \frac{u(\z)}{\z^\alpha \z_1\cdots \z_n}d\z_1\cdots d\z_n\right)z^\alpha.
$$

We also obtain by differentiation of (2.2.1) that
$$
(2.2.3)\;\;\partial^\alpha u(0)=\left(\frac{1}{2\pi i}\right)^n\alpha!\int_{\z\in \partial_0D} \frac{u(\z)}{\z^\alpha \z_1\cdots \z_n}d\z_1\cdots d\z_n
$$
Hence the theorem follows if $u\in\mathcal C^1(\overline{D}).$
To prove it in general, prove it for any strictly smaller polydisc.

\end{proof}

We obtain also from (2.2.3) the following

{\bf Theorem 2.2.7} (Cauchy's inequalities)
If $u$ is analytic on the polydisc $D(0,r)$ and if $|u|\leq M$, then
$$
|\partial^\alpha u(0)| \leq M \alpha! r^{-\alpha}.
$$

For the proof apply (2.2.3) to any smaller polydisc.

2.5. Domains of holomorphy\\

If $\Omega\subset \mathbb C$ and $p\in \partial \Omega, $ then the function $\frac{1}{z-p}$ cannot be extended analytically from $\Omega$ across $p.$ We express this fact by calling $\Omega$ a domain of holomorphy.

In $\mathbb C^n,n>1$ it might be sometimes possible to extend holomorphic functions past the boundary. For example, let $\Omega=\mathbb C^n\setminus \overline{D}(0,r), r=(s,\dots, s).$ Then all holomorphic functions on $\Omega$ extend to holomorphic functions on all of $\mathbb C^n.$
To see this, use the Cauchy integral formula in the last variable:

$g(z_1,\cdots,z_{n-1},z_n)=\frac{1}{2\pi i}\int_{|\z|= S}\frac{f(z_1,\dots, z_{n-1},\z)d\z}{\z-z_n}.$
Here we use any number $S>s.$ It is clear that if some $|z_j|>s$ for $j>n,$ the function $g= f$.
Also the integral defines an analytic function in $\mathbb C^n.$ We see then that this defines an anlytic extension to all of $\mathbb C^n.$ For this reason we will not call this domain a domain of holomorphy. The precise definition of domain of holomorphy is a little complicated. The reason is seen from the example $f=\sqrt{z}$ which is well defined in the complement of the set $\{x+i0,x\geq 0\}.$
This function can locally be extended across the boundary at any $x+iy,x>0.$ But these local extensions dont agree with the function as defined on the other side of the real axis.
We now give the precise definition of domain of holomorphy.

{\bf Definition 2.5.1}
An open set in $\mathbb C^n$ is called a domain of holomorphy if there are no open sets $\Omega_1$ and $\Omega_2$ in $\mathbb C^n$ with the following properties:\\
(a) $\emptyset \neq \Omega_1\subset \Omega_2\cap \Omega.$\\
(b) $\Omega_2$ is connected and not contained in $\Omega.$\\
(c) For every $u\in A(\Omega)$ there is a function $u_2\in A(\Omega_2)$ 
such that $u=u_2$ in $\Omega_1.$

{\bf Definition 2.5.2}
If $K$ is a compact subset of $\Omega$, we define the $A(\Omega)$ hull $\hat{K}_\Omega$ of $K$ by
$$
(2.5.1)\;\; \hat{K}_\Omega= \{z\in \Omega; |f(z)|\leq \sup_K|f| \;\mbox{if}\;f\in A(\Omega)\}.
$$

{\bf Exercises}\\

1) Let $\Omega$ be an open set in $\mathbb C.$ For $K$ compact in $\Omega$ show that
$\hat{K}_\Omega$ is compact and that the distance of $\hat{K}_\Omega$ to the boundary of $\Omega$ is the same as the distance of $K$ to the boundary.\\

2) Let $K\subset \Omega\subset \mathbb C$ be a compact subset. Let $U$ be a connected component of $\mathbb C\setminus K.$ Show that $U\subset \hat{K}_\Omega$ if and only if
$U$ is a bounded set and $U\subset \Omega.$\\

3) Show that a polydisc in $\mathbb C^n$ is a domain of holomorphy.\\

\newpage

\section{Hormander 2.5, 2.6}

{\bf Lemma 2.5.3}
Let $f\in A(\Omega).$ Let $K$ be a compact subset of $\Omega.$ Let $d_K$ denote the sup of the numbers $r$ so that $\Delta^n(z,(r,\cdots,r))\subset \Omega$ for  all $z\in K.$ Then, if $w\in \hat{K}_\Omega$
then the power series of $f$ around $w$ converges normally in $\Delta^n(w,(r,\cdots,r)).$

\begin{proof}
Fix any $0<r<d.$ Then $f$ is uniformly bounded by some constant $M$ on all
polydiscs $\Delta^n(z,(r,\dots,r))$ when $z\in K.$ It follows from Theorem 2.2.7
that for any $z\in K$ and any multiindex $\alpha,$
$|\partial^\alpha f(z)|\leq M\alpha!r^{-\alpha}.$ Then these inequalties must also hold at every point $w\in \hat{K}_\Omega.$

But then this implies that the power series expansion of $f$ around $w$ converges normally in
$\Delta(w,(r,\dots,r)).$ 
\end{proof}

{\bf Theorem 2.5.4}
Let $\Omega$ be a domain of holomorphy. 
Let $K$ be a compact subset. Then if $d(L)$ denotes the sup of all radii $r$ so that
$B(z,r)\subset \Omega$ for all $z\in L,$, then $d(K)=d(\hat{K}).$

\begin{proof}
If we instead measure boundary distance using polydiscs of multiradius $(r,\dots,r)$ then
the result holds by Lemma 2.5.3. By scaling in each variable, we see that
if $D$ is any polydisc such that $z+D\subset \Omega$ for all $z\in K,$ then
also $w+D\subset \Omega$ for all $w\in \hat{K}.$

Next we can choose any orthonormal basis for $\mathbb C^n$ and define polydiscs in these coordinates. Then we see that the result also holds for such polydiscs.
Next let $B(r)$ be a ball such that $z+B(r)\subset \Omega$ for all $z\in K,$ then since the ball is a union of polydiscs included rotated polydiscs, we see that $w+B(r)\subset \Omega$ for all $w\in \hat{K}.$
\end{proof}

{\bf Theorem 2.5.5}
If $\Omega$ is an open set in $\mathbb C^n,$ then the following conditions are equivalent:\\
(i) $\Omega$ is a domain of holomorphy.\\
(ii) If $K$ is a compact subset of $\Omega,$ then $\hat{K}$ is a compact subset of $\Omega$ and $d(\hat{K})=d(K).$ \\
(iii) If $K$ is a  compact subset of $\Omega,$ then $\hat{K}$ is also a compact subset of $\Omega.$\\
(iv) There exists a function $f\in A(\Omega)$ such that it is not possible to find $\Omega_1$ and $\Omega_2$ satisfying (a) and (b) in Definition 2.5.1 and $f_2\in A(\Omega_2)$ so that $f=f_2$ in $\Omega_1.$

\begin{proof}
Notice that if $K\subset D$ for some polydisc, then $\hat{K}$ is also contained in $D.$ Suppose that (i)
holds and that $K$ is a compact subset of $\Omega.$ Then by Theorem 2.5.4, $\hat{K}$ is a closed set in $\mathbb C^n.$ Since $\hat{K} $ is also bounded, it follows that $\hat{K}$ is a compact subset of $\Omega.$
Moreover, it follows from Theorem 2.5.4 that $d(\hat{K})=d(K).$ Hence (i) implies (ii). It is clear that (ii) implies (iii) and that (iv) implies (i).

Hence it only remains to show that (iii) implies (iv).

Let $D$ be a polydisc. For each $\z\in \Omega$, let $D_\z=\z+rD$ denote the largest polydisc contained in $\Omega.$ Let $M$ be a countable dense subset of $\Omega.$

We need a lemma:

{\bf Lemma 2.5.5a}
Suppose that $\Omega$ satisfies (iii). Then there exists a holomorphic function $f$ on $\Omega$
such that for each $\z\in M$, there is no holomorphic function $g$ defined on some neighborhood $U$ of
$\overline{D}_\z$ such that $f=g$ on $D_\z.$

Proof of Lemma 2.5.5a:\\
Let $\z_n$ be a list of the points in $M$ such that each point in $M$ is listed infinitely many times.
Let $K_1\subset \cdots \subset K_n\subset \cdots$ be compact subsets of $\Omega$ such that each compact set in $\Omega$ is contained in some $K_n.$ Since $\hat{K}_j$ is a compact subset of $\Omega$, there exists a point $z_j\in D_{\z_j} \setminus \hat{K}_j.$ Hence there is a function $f_j\in A(\Omega)$ so that $f_j(z_j)=1$ and $\sup_{K_j}|f_j|<2^{-j}.$ We can choose $f_j$ so that $f_j$ is not identically $1$ in any connected component of $\Omega.$

Let
$$f=\Pi_{j=1}^\infty (1-f_j)^j.
$$

Since the sum $\sum j2^{-j}$ is convergent, this infinite product converges to an analytic function which does not vanish identically on any connected component of $\Omega.$
All derivatives of $f$ up to order $j$ vanish at $z_j.$ Hence there can be no analytic function $g$ defined on any neighborhood $U$ of any $\overline{D}_{\z_j}$ agreeing with $f$ on $D_{\z_j}.$

This finishes the proof of Lemma 2.5.5a. We next continue with the proof that (iii) implies (iv).
Let $\Omega_1,\Omega_2$ be any two open sets satisfying (a) and (b) in Definition 2.5.1.
We assume that there is a holomorphic function $f_2$ on $\Omega_2$ such that
$f_2=f$ on $\Omega_1.$ We will show that this leads to a contradiction.

We can find a curve $\gamma(t)\in \Omega_2, 0\leq t\leq 1$ so that $\gamma(0)\in \Omega_1$ and
$\gamma(1)\in \partial \Omega, \gamma([0,1)\subset \Omega.$ By analytic continuation, $f=f_2$ on an open set containing $\gamma([0,1)).$ We then get a contradiction to the conclusion of Lemma 2.5.5a
by chosing a point $\z_j$ in this neighborhood very close to $\partial \Omega$ since $\Omega_2$ will contain a polydisc centered at $\gamma(1).$

\end{proof}

2.6 Pseudoconvexity and plurisubharmonicity.\\

{\bf Definition 2.6.1}
A function defined in an open set $\Omega\subset\mathbb C^n$ with values in $[-\infty,\infty)$
is called plurisubharmonic if\\
(a) $u$ is upper semicontinuous.\\
(b) For arbitrary $z$ and $w$ in $\mathbb C^n$, the function $\tau\rightarrow u(z+\tau w)$
is subharmonic in the part of $\mathbb C$ where it is defined.

{\bf Theorem 2.6.2}
A function $u\in \mathcal C^2(\Omega)$ is plurisubharmonic if and only if 
$$
(2.6.1)\;\; \sum_{j,k=1}^n \frac{\partial^2 u(z)}{\partial z_j\overline{z}_k}w_j\overline{w}_k\geq 0,
z\in \Omega, w\in \mathbb C^n.
$$

\begin{proof}
By Corollary 1.6.10a and Lemma 1.6.11c, a $\mathcal C^2$ function defined on an open set in $\mathbb C$ is
subharmonic if and only if $\Delta u\geq 0.$
We calculate:
\bea
\frac{\partial u(z+\tau w)}{\partial \overline{\tau}} &  = & \sum_{k=1}^n \frac{\partial u}{\partial \overline{\z}_k}
\overline{w}_k\\
\frac{\partial^2 u(z+\tau w)}{\partial \tau\partial \overline{\tau}} &  = &
 \sum_{j=1}^n\sum_{k=1}^n \frac{\partial^2 u}{\partial \z_j\partial \overline{\z}_k}w_j
\overline{w}_k.\\
\eea
The Theorem follows.
\end{proof}

{\bf Theorem 2.6.3}
Let $0\leq \phi\in \mathcal C^\infty_0(\mathbb C^n)$ be equal to $0$ when $|z|>1.$
Let $\phi$ depend only on $|z_1|,\dots,|z_n|,$ and assume that $\int \phi d\lambda=1$ where $d\lambda$ is the Lebesgue measure. If $u$ is plurisubharmonic in $\Omega,$ it follows that
$$
u_\epsilon(z) = \int u(z-\epsilon \z)\phi(\z)d\lambda(\z)
$$
is plurisubharmonic, that $u_\epsilon\in \mathcal C^\infty$ in $\Omega_\epsilon$, and that $u_\epsilon \searrow u.$ ( We assume that $u$ is not identically $ -\infty$ on any connected component of $\Omega.$)

\begin{proof}
That $u_\epsilon \searrow $ when $\epsilon \searrow 0$ was proved in Lemma 1.6.11e in the case $n=1.$
To show this when $n>1$, choose first $\epsilon'=(\epsilon_1,\dots,\epsilon_n)$ and define

$$
u_{\epsilon'}(z) = \int u(z_1-\epsilon_1 \z_1,\cdots,z_n-\epsilon_n\z_n)\phi(\z)d\lambda(\z)
$$
The one variable result implies that this expression decreases when we decrease only one of the $\epsilon_i.$ Hence repeating this process $n$ times show that $u_{\epsilon^1}\geq u_{\epsilon^2}$
if $\epsilon^1\geq \epsilon^2.$
From Theorem 1.6.3 it follows that
\bea
u_\epsilon(z) & = & \int u(z-\epsilon \z)\phi(\z)d\lambda(\z)\\
& = & \int u(z_1-\epsilon \z_1,\cdots,z_n-\epsilon\z_n)\phi(\z)d\lambda(\z_1)\dots d\lambda(\z_n)\\
& \geq & \int u(z_1-\epsilon \z_1,\cdots,z_{n-1}-\epsilon\z_{n-1},z_n)\phi(\z)d\lambda(\z_1)\dots d\lambda(\z_{n-1})\\
& \geq & \int uz_1-\epsilon \z_1,\cdots,z_{n-1}-\epsilon\z_{n-2},z_{n-1},z_n))\phi(\z)d\lambda(\z_1)\dots d\lambda(\z_{n-2})\\
& \dots & \\
& \geq & u(z)\\
\eea

By upper semicontinuity of $u$ it follows that $\limsup_{\epsilon\rightarrow 0}u_\epsilon \leq u.$
Hence $u_\epsilon \searrow u$ when $\epsilon \searrow 0.$
To show that $u_\epsilon$ is plurisubharmonic, we fix a complex line
$\tau\rightarrow z+\tau w$ and show that $u_\epsilon (z+\tau w)$ is subharmonic as a function of $\tau.$
By theorem 1.6.3 it suffices to find for each $\tau_0$ such that $z+\tau_0 w\in \Omega$ arbitrarily small $r>0$ so that $u_\epsilon(z+\tau_0w)\leq \frac{1}{2\pi}\int_0^{2\pi} u_\epsilon(z+(\tau_0+re^{i\theta})w)d\theta.$
We calculate:
\bea
& & \frac{1}{2\pi}\int_0^{2\pi} u_\epsilon(z+(\tau_0+re^{i\theta})w)d\theta\\
 & = &
\frac{1}{2\pi}\int_0^{2\pi} (\int\int u(z+(\tau_0+re^{i\theta})w-\epsilon \z)\phi(\z)d\lambda(\z))d\theta\\
& = &
 \int\int\phi(\z) (\frac{1}{2\pi}\int_0^{2\pi}u(z+(\tau_0+re^{i\theta})w-\epsilon \z))d\theta d\lambda(\z))\\
& \geq & \int\int\phi(\z) u(z+\tau_0w-\epsilon \z) d\lambda(\z))\\
& = & u_\epsilon (z+\tau_0 w)
\eea

\end{proof}

{\bf Definition 2.6.6}
If $K$ is a compact subset of $\Omega\subset \mathbb C^n,$ we define the $P(\Omega)$ hull $\hat{K}^P_\Omega$ of $K$ by
$$
\hat{K}^P_\Omega=\{z\in \Omega; u(z)\leq \sup_Ku\;\forall\;u\in P(\Omega)\}.
$$

Since $|f|\in P(\Omega)$ for all $f\in A(\Omega)$ we have that
$\hat{K}^P_\Omega \subset \hat{K}_\Omega.$

{\bf Theorem 2.6.5}
Any of the following two conditions imply that $-\log d(z,\Omega^c)$ is plurisubharmonic and continuous:\\
(a) $\Omega$ is a domain of holomorphy.\\
(b) $\hat{K}^P_\Omega\subset\subset\Omega$ whenever $K$ is a compact subset of $\Omega.$

\newpage

\section{Hormander 2.6, 4.1-Banach spaces}

\begin{proof} In the case (a) we will use Theorem 2.5.5 which implies that $\hat{K}_\Omega$ is a compact subset of $\Omega$ whenever $K$ is a compact subset of $\Omega.$\\

Pick a unit vector $\xi$. We define a distance in the $\xi$ direction:
For $z\in \Omega,$ we set $d_\xi(z):=\sup \{t; z+\tau \xi\in \Omega, \forall \; \tau\in \mathbb C,
|\tau|<t\}.$ We show that $-\log (d_\xi)$ is plurisubharmonic. Assume not. Then there is a complex line $L$ and a disc $\overline{D}\subset L$ so that $d_\xi$ has value at the center strictly larger than the average value on the boundary.
We can assume that $L$ is the $z_1$ axis and $D$ is the unit disc. We can choose a holomorphic polynomial $P(z)$ with $h=\Re (P(z))$ so that $h>-\log d_\xi$ on the boundary of the disc and
$h(0)<-\log d_\xi(0).$ Consider the complex discs $D_t=D_t(\z)$ for $t\in \mathbb C, |t|\leq 1$ and for
$\zeta\in \mathbb C, |\zeta|\leq 1,$
$D_t(\zeta)=(\zeta,0,\dots,0)+ t\xi e^{-P(\zeta)}.$
If $\zeta$ is on the boundary of the unit disc, we have that $-\log d_{\xi}(\zeta)<h(\zeta)$ and hence
$|t e^{-P(\zeta)}|=|t|e^{-h(\zeta)}< e^{\log d_\xi (\zeta)}=d_\xi(\zeta).$ It follows that boundaries of the discs $D_t$
are all in $\Omega.$ We will let $K$ be the compact union of all these boundaries. Then $\hat{K}_P^\Omega \subset \subset \Omega$ or $\hat{K}_\Omega$ is compact in $\Omega.$ For $t=0$ the interior of the disc is in $\Omega.$ 
 For those $t$ for which the whole disc is in $\Omega$,
we have then that the distance from the disc to the boundary of $\Omega$ has a uniform lower bound. Hence
no such disc gets too close to the boundary. So it follows that also for all $t,|t|\leq 1$ the disc $D_t$ is in $\Omega.$
Let $\zeta=0, |t|\leq 1.$ Then $D_t(0)=t\xi e^{-P(0)}\in \Omega.$ Hence 
$|e^{-P(0)}|< d_\xi(0),$ so $-\log d_\xi(0)< h(0)$, a contradiction.\\
To finish the proof note that $-\log d$ is the sup of all $-\log d_\xi$ and that this
is continuous.
\end{proof}

{\bf Theorem 2.6.7}
If $\Omega$ is a proper open subset of $\mathbb C^n,$ the following conditions are equivalent:\\
(i) $-\log d(z,\Omega^c)$ is plurisubharmonic in $\Omega.$\\
(ii) There exists a continuous plurisubharmonic function $u$ in $\Omega$ such that
$$
\Omega_c=\{z\in \Omega; u(z)\} \subset \subset \Omega
$$
for every $c\in \mathbb R.$\\
(iii) $\hat{K}^P_\Omega \subset \subset \Omega$ for all compact sets $K\subset \Omega.$

\begin{proof}
To prove that (i) implies (ii), we define $u(z)=\|z\|^2-\log d(z,\Omega^c).$ Then if $z_n$ is a sequence in $\Omega$ converging to $\partial \Omega,$ (including $\infty$ if $\Omega$ is unbounded,)
then $u(z_n)\rightarrow \infty.$\\
(ii) obviously implies (iii).\\
That (iii) implies (i) follows from Theorem 2.6.5.

\end{proof}

{\bf Definition 2.6.8}
The domain $\Omega$ is called pseudoconvex if $\Omega=\mathbb C^n$ or if the equivalent conditions in Theorem 2.6.7 are fulfilled.\\

We have shown that domains of holomorphy are pseudoconvex. The Levi problem asks if pseudoconvex domains are domains of holomorphy. This is solved in Chapter IV and is major result in this course.\\

{\bf Theorem 2.6.11}
Let $\Omega$ be a pseudoconvex open set in $\mathbb C^n$, let $K$ be a compact subset of $\Omega$, and $\omega$ an open neighborhood of $\hat{K}^P_\Omega.$ Then there exists a function $u\in \mathcal C^\infty(\Omega)$ such that \\
(a) $u$ is strictly plurisubharmonic, that is the hermitian form in (2.6.1) is strictly positive definite for every $z\in \Omega,$ i.e. $\geq c(z)\|w\|^2, c>0.$\\
(b) $u<0$ in $K$, but $u>0$ in $\Omega \setminus \omega.$\\
(c) $\{z\in \Omega; u(z)<c\} \subset \subset \Omega$ for all $c\in \mathbb R.$\\

{\bf Lemma 2.6.11a}
There exists a continuous plurisubharmonic function $v(z)$ satisfying (b) and (c).\\

We prove first the Lemma.
By theorem 2.6.7 (ii) there exists a continuous plurisubharmonic function $u_0$ on $\Omega$ satisfying (c). We can assume that $u_0<0$ on $K$ by subtracting a constant if necessary.
Set $K'=\{z\in \Omega; u_0(z)\leq 2\}$ and let $L=\{z\in \Omega\setminus \omega; u_0\leq 0\}.$
If $L$ is empty, we choose $v=u_0$. Then (b) is also satisfied. So we assume that $L$ is nonempty.
By continuity of $u_0$, the sets $K'$ and $L\subset K'$ are compact. Hence $K'\cup L\subset \Omega_\delta$ for all small enough $\delta.$\\

Let $p\in L.$ Then $p\in \omega^c$ and hence $p\notin \hat{K}^P_\Omega.$ Therefore there
exists a plurisubharmonic function $w_p$ on $\Omega$ such that $w_p(p)>0$ and $w_p<0$ on $K.$
Using $(w_p)_\epsilon$ as in Theorem 2.6.3, we get for small enough $\epsilon$ a smooth plurisubharmonic function in $\Omega_\epsilon$ so that $(w_p)_\epsilon<0$ on $K$ and
$(w_p)_\epsilon >0$ on some open neighborhood $U_{\epsilon,p}$ of $p$.
By compactness we can cover $L$ by finitely many such neighborhoods, $U_{\epsilon_j,p_j}.$
Let $w=\sup \{(w_{p_i})_{\epsilon_i}$ on $\Omega_{\max_j \epsilon_j}\}.$
Then $w$ is continuous and plurisubharmonic, $w>0$ on $L$ and $w<0$ on $K.$ In particular we can assume that $w$ is plurisubharmonic in a neighborhood of $K'.$
Let $C$ denote the maximum value of $w$ on $K'.$ Since $L\subset K'$, $C>0.$
Note then that if $1<u_0(z)<2$, then $w\leq C<Cu_0$, hence
$Cu_0=\max\{w,Cu_0\}$ on the set $1<u_0<2.$ Hence the function $v(z)=\max\{w,Cu_0\}$
on $\{u_0<2\}$ and $v(z)=Cu_0$ when $u_0>1$ is well defined on all of $\Omega$ and is locally plurisubharmonic, hence plurisubharmonic (see Corollary 1.6.5).
We then have a continuous plurisubharmonic function on $\Omega$ which satisfies (b) and (c).
This finishes the proof of Lemma 2.6.11.a

We now prove the Theorem.
\begin{proof}
Let $v$ be as in Lemma 2.6.11a.
For any $c\in \mathbb R,$ let $\Omega_c:=\{z\in \Omega; v(z)<c\}.$ We use the notation of Theorem 
2.6.3. Set
$$
v_j(z)=\int_{\Omega_{j+1}} v(\z)\phi((z-\z)/\epsilon_j) \epsilon_j^{-2n}d\lambda(\z)+\epsilon_j |z|^2,j=0,1,\dots.
$$

We choose $\epsilon_j$ small enough that $v_0,v_1<0$ on $K$ and $v_j<v+1$ in $\Omega_j$.
We also have that $v_j\in \mathcal C^\infty(\mathbb C^n)$ and we can arrange that $v$ is strictly plurisubharmonic and $>v$ in some neighborhood of $\overline{\Omega}_j$.

Let $\chi$ denote a $\mathcal C^\infty$ convex increasing function on $\mathbb R$ such that
$\chi(t)=0, t\leq 0$ and $\chi'(t)>0, t>0.$ Let $j\geq 1:$
On $\overline{\Omega}_j\setminus \Omega_{j-1},$ $v_j+1-j>v+1-j\geq (j-1)+1-j=0$, where the first inequality holds on $\overline{\Omega}_j$ and the second inequality holds in the complement of $\Omega_{j-1}.$
Hence $\chi(v_j+1-j)$ is strictly plurisubharmonic in a neighborhood of $\overline{\Omega}_j\setminus \Omega_{j-1}.$ 
Let $u_m=v_0+\sum_1^m a_j\chi(v_j+1-j).$
Then $u_0$ is strongly plurisubharmonic on a neighborhood of $\overline{\Omega}_0.$
Since $\chi(v_1+1-1)$ is plurisubharmonic on $\Omega_1$ and strongly plurisubharmonic
on a neighborhood of $\overline{\Omega}_1\setminus \Omega_0$, the function $u_1$ will
be strictly plurisubharmonic in a neighborhood of $\overline{\Omega}_1$ if we choose $a_1$ large enough. We also choose $a_1$ large enough so that $u_1>v$ on $\overline{\Omega}_1.$ Similarly, we can next choose $a_2$ large enough that $u_2$ is strongly plurisubharmonic in a neighborhood of $\overline{\Omega}_2.$ Inductively we see that $u_m$ can be chosen strongly plurisubharmonic in a neighborhood of $\overline{\Omega}_m$ and also larger than $v$ there. Notice that $u_0<0$ on $K.$
Also $v_1+1-1<0$ on $K$, so $u_1=v_0<0$ on $K.$ Finally
also observe that for $j\geq 2, v_j+1-j< v+1+1-j<(j-2)+2-j=0$ on $\Omega_{j-2}.$ Therefore all $u_m=v_0<0$ on $K$. Moreover on any $\Omega_j, $ if $m>j+2,$ $u_m=u_j$, so the infinite sum is locally finite, so the limit exists and is strongly plurisubharmonic on $\Omega.$

\end{proof}

\newpage

We next move on to Chapter 4 in Hormander: $L^2$ estimates and existence theorems for the $\overline{\partial}$ operator.\\

We start with some functional analysis.
Let $G_1,G_2$ denote two complex Banach spaces with norms $\|\cdot\|_1, \|\cdot\|_2$ respectively.
Let $E$ denote a complex subspace of $G_1$, not necessarily closed. 
We will consider linear maps $T:E\rightarrow G_2$.  Let $G_3=G_1\times G_2$ denote the product space
with Banach norm $\|(x,y)\|_3^2=\|x\|_1^2+\|y\|_2^2.$
We say that $T$ is closed if the graph of $T$, $G_T=\{(x,Tx); x\in E\}$ is a closed subspace of $G_3.$

{\bf Example 4.0.1}
$G_1=G_2=L^2(0,1)$ and $Tx=x'.$ Here the derivative is in the sense of distributions.
So $Tx(\phi)=-\int x \phi'$. Here $E$ consists of those $L^2$ functions $x$ for which $Tx$ is an $L^2$ function. Then $T$ is a closed operator: If $(x_n,Tx_n)$ converge to $(x,y)$ then
for any test function $\phi$ we have $\int Tx_n \phi=-\int x_n \phi'$ so taking limit one gets
$\int y\phi=-\int x\phi'.$ Therefore $y$ equals $x'$ in the sense of distributions.\\

Let $G_i'$ denote the dual Banach space of $G_i$.  So an element $y\in G_i'$ is a continuous linear function (also called functional) from $G_i$ to $\mathbb C, x\rightarrow y(x).$ Moreover there is a constant $C_y$ so that
$|y(x)|\leq C_y\|x\|_i$. The smallest such constant $C_y$ is denoted by $\|y\|_i.$ \\

An important theorem is the Hahn-Banach Theorem. 

{\bf Theorem 4.0.2, Hahn-Banach}
Let $L\subset G$ be a linear subspace of a Banach space. Suppose that $\phi:L\rightarrow \mathbb C$
be a linear function with bounded norm, i.e. $|\phi(x)|\leq C\|x\|.$ Then $\phi$ extends
to a linear function $\tilde{\phi}:G\rightarrow \mathbb C$ such that $\tilde{\phi}(x)=\phi(x)$ for all $x\in L$ and $|\tilde{\phi}(x)|\leq C\|x\|$ for all $x\in G.$ In particular the extension belongs to
$G'.$\\

\section{Hormander 4.1-Banach spaces}

We will next add the hypothesis that the operator $T$ is densely defined, i.e. the subspace $E$ is dense. We will define the adjoint operator $T^*:D_{T^*}\rightarrow G_1'$ where $D_{T^*}\subset G_2'$
is a linear subspace. We say that $y$ belongs to $D_{T^*}$ if there exists a constant $C$ so that  $|y(Tx)|\leq C\|x\|_1$ for all $x\in E.$ Hence $y\rightarrow y(Tx)$ extends to a continuous linear functional $z$ on $\overline{E}=G_1$ such that $\|z\|_{G_1'}\leq C.$
We set $T^*(y)=z.$

If $(y_1,y_2)\in G_1'\times G_2'$ then this defines a continuous linear functional on $G_1\times G_2$
by $(y_1,y_2)(x_1,x_2)=y_1(x_1)-y_2(x_2).$ 

Let $G_T^\perp\subset G_1'\times G_2'$ denote those $(y_1,y_2)$ for which
$(y_1,y_2)(x_1,x_2)=0$ for all $(x_1,x_2)\in G_T.$
Clearly $G_T^\perp$ is a closed subspace of $G'_1\times G'_2.$

{\bf Lemma 4.0.3}
$G_T^\perp = G_{T^*}.$\\

\begin{proof}
Suppose first that $(y_1,y_2)\in G_T^\perp.$ Then if $x\in D_T$ we have that
$y_1(x)-y_2(Tx)=0.$ This implies that $|y_2(Tx)|=|y_1(x)|\leq \|y_1\| \|x\|.$ Hence $y_2$ satisfies the requirement to be in $D_{T^*}.$ Moreover $y_1$ satisfies the equirement to equal $T^*(y_2).$ Hence $(y_1,y_2)\in G_{T^*}.$ 

Suppose next that $(y_1,y_2)\in G_{T^*}.$ Hence $y_1=T^*(y_2).$ This implies that for any $x\in D_T$, we have that $y_1(x)=y_2(Tx).$ Hence $(y_1,y_2)\in G_T^\perp.$
\end{proof}

{\bf Corollary 4.0.4}
The graph of $T^*$ is closed.\\

If $G$ is a complex Banach space, we define $G''$ to be the dual of $G'.$ There is a natural isometric embedding $\phi$ of $G$ into $G''$: For $x\in G,$ define $\phi(x)(y)=y(x)$ for $y\in G_1'$. Then $\phi(x)\in G_1''.$ Also $\|\phi(x)\|\leq \|x\|$ and by the Hahn Banach theorem you have equality: Given $x\neq 0$, choose a linear function $\tilde{y}$ on $\mathbb C x$ by $\tilde{y}(x)=\|x\|$ and extend to $G$ by Hahn-Banach. Call the extension $y.$ Then $y\in G'$  and $\|y\|=1.$ Now $\phi(x)(y)=y(x)=\|x\|=\|x\|\|y\|$ so $\|\phi(x)\|\geq \|x\|.$ Hence $\|\phi(x)\|=\|x\|.$
We say that $G''=G$ if this map is surjective. The Banach space is called reflexive in this case.
In this case, we also have that $G'=G'''.$ 
Note that if $G_1,G_2$ are Banach spaces, then the dual of $G_1\times G_2$ equals $G_1'\times G_2'$.
Namely, if $\phi$ is a continuous linear function on $G_1\times G_2$, then
$\phi(x_1,x_2)=\phi(x_1,0)+\phi(0,x_2)=y_1(x_1)+y_2(x_2).$\\

{\bf Lemma 4.0.5}
Let  $G$ be a reflexive  Banach space,  i.e. $\phi(G)=G'',$ and let $H$ be a closed subspace of $G.$ Then $(H^\perp)^\perp=\phi(H).$ \\

\begin{proof}
Since $\phi$ is an isometry and $H$ is closed, $\phi(H)$ is also closed. This follows because
Banach spaces are complete.
Suppose that $x\in H$ and $y\in H^\perp.$ Then $y(x)=0$ and hence $\phi(x)(y)=y(x)=0.$
Hence $\phi(x)\in H^{\perp\perp}.$ Hence $\phi(H)\subset H^{\perp\perp}.$ Suppose next that
$z\in G''\setminus\phi(H).$ 
Hence there exists an $\eta\in G'''$ so that $\eta(z)\neq 0$ while $\eta$ vanishes on $\phi(H).$
By reflexivity, we have that there exists a $y\in G'$ so that $z(y)=\eta(z)\neq 0$ while $w(y)=\eta(w)=0$ for all $w\in \phi(H).$ Hence if $x\in H$ and $w=\phi(x)\in \phi(H),$ then $y(x)=(\phi(x))(y) =w(y)=0.$ So, $y\in H^\perp.$ Since $z(y)\neq 0,$ it follows that $z$ cannot be in $H^{\perp\perp},$ so $H^{\perp\perp}\subset \phi(H).$
\end{proof}

{\bf Corollary 4.0.6}
If $G_1,G_2$ are reflexive and $T:G_1\rightarrow G_2$ is a densely defined closed linear operator,
then $T^{**}=T.$\\

\begin{proof}
Set $G_j''=\phi_j(G_j).$ Then
$$
G_{T^{**}}=(G_{T^*})^\perp=G_T^{\perp\perp}=\{(\phi_1(x),\phi_2(Tx));(x,Tx)\in G_T\}.
$$
We write this imprecisely as $G_{T^{**}}=G_T,$ or $T^{**}=T.$
\end{proof}

{\bf Lemma 4.0.7}
Assume that $G_1,G_2$ are reflexive Banach spaces. 
Then the  operator $T^*$ is densely defined.\\

\begin{proof}
Suppose that there exists a $y_0\in G_2'\setminus \overline{D}_{T^*}.$ 
Then there exists a $z_0\in G_2''$ so that $z_0(y_0) \neq 0$ while
$z_0(y)=0$ for all $y\in D_{T^*}.$ So $z_0\neq 0.$ Hence for the point $(0,z_0)\in G_1''\times G_2''$
we have that $0(T^*y)-z_0(y)=0$ for all $y\in D_{T^*}.$ 
It follows that $(0,z_0)\in G_{T^*}^\perp=G_T^{\perp\perp}.$ We write $(0,z_0)=(\phi_1(0),\phi_2(x))$ for some
$(0,x)\in G_1\times G_2.$ Then $(0,x)\in G_T.$ But $T(0)=0$ and $x$ cannot be zero since $\phi_2(x)=z_0\neq 0$, a contradiction.
\end{proof}

We recall the uniform boundedness principle (Banach-Steinhaus theorem).

{\bf Theorem 4.0.8}
Let $\mathcal F$ denote a family of continuous linear functionals on a Banach space $B$.
 Suppose that for every $x\in B$ there exists a constant $c_x$ so that $|F(x)|\leq 
c_x\|x\|_B$ for all $F\in \mathcal F.$ Then there exists a constant $C$ so that $\|F\|\leq C$ for all $F\in \mathcal F.$

The proof uses Baire category. For any number $A$ the set of $
x$ where $c_x\leq A$ 
is closed. (Closedness follows from continuity of the functionals.) Hence for some $A$ the set has interior.

Notation: Let $y\in G'$ and let $H\subset G$ be a linear subspace.
We denote by $\|y_{|H}\|_{G'}$ the norm of the linear functional $y$ restricted to $H.$
So $\|y_{|H}\|_{G'}$ is the smallest $c$ so that $|y(x)|\leq c\|x\|$ for all $x\in H.$

We next give a more general version of Theorem 4.1.1 in Hormander. 

{\bf Theorem 4.1.1'}
Let $G_1,G_2$ be reflexive Banach spaces and let $T:G_1\rightarrow G_2$ be a densely defined
closed linear operator.  Let $F\subset G_2$ be a closed subspace containing the range of $T, R_T.$
Then $F=R_T$ if and only if there exists a constant $C>0$ such that
$$
(4.1.1)' \;\;\;\|y_{|F}\|_{G_2'} \leq C\|T^*y\|_{G'_1}\; \forall \; y\in  D_{T^*}.
$$
If any of the two equivalent conditions are satisfied, then there exists for every $z\in F$
an $x\in D_T$ with $Tx=z$ and $\|x\|_{G_1}\leq C\|Tx\|_{G_2}=C\|z\|_{G_2}$ for the same constant $C.$

\begin{proof}
Suppose that $R_T=F.$ We will apply the Banach-Steinhaus theorem to a family of linear
functionals on the Banach space $F.$ Namely, let $\mathcal G$ denote the family of $y\in
D_{T^*}$ for which $\|T^*y\|_{G_1'} \leq 1.$
Define for each such $y$ a linear function $L_y\in\mathcal F$ on $F$ given by $L_y(x)=y(x).$ This is a continuous linear functional defined on $F.$ For a given $x\in F$, pick some $z\in D_T$ for which $x=Tz.$
Then we have that for any $L_y\in \mathcal F,$ that
$$
|L_y(x)|=|y(x)|= |y(Tz)|= |(T^*(y))(z)|\leq \|T^*(y)\|_{G_1'} \|z\|_{G_1}\leq \|z\|_{G_1}.
$$
Hence the family $\mathcal F$ is bounded uniformly on any given $x\in F.$
Hence by the Banach-Steinhaus Theorem, there is a constant $C$ so that
$\|(L_y)_{|F}\|_{G_2'}\leq C$ for any $y\in D_{T^*}$ with $\|T^*y\|_{G'_1}\leq 1.$ Then it follows that
$\|(L_y)_{|F}\|_{G_2'} \leq C\|T^*y\|_{G'_1}\; \forall \; y\in  D_{T^*}$. Since $(L_y)_{|F}=y(x),x\in F$ we get that $\|y_{|F}\|_{G_2'}\leq C\|T^*y\|_{G_1'}$ for all $y\in D_{T^*}.$

We next suppose that (4.1.1)' is satisfied. Fix a $z\in F.$ If $y\in D_{T^*}, w=T^*(y)\in R_{T^*},$ set
$\phi(w)=y(z).$ Note that if $w=T^*(y_1)=T^*(y_2),$ then $T^*(y_1-y_2)=0.$ Hence
by (4.1.1)', $(y_1-y_2)(z)=0$, so $y_1(z)=y_2(z)$ and therefore $\phi(w)$ is well defined.  Also, by (4.1.1)', 
\bea
|\phi(w)| & = & |y(z)|\\
& \leq & \|y_{|F}\|_{G_2'}\|\|z\|_{G_2}\\
& \leq & C\|T^*y\|_{G_1'}\|z\|_{G_2}\\
& = & C\|w\|_{G_1'}\|z\|_{G_2},
\eea
hence $\phi$ is a bounded linear functional with norm at most $C\|z\|_{G_2}.$
 This is then a bounded linear function on the Range of $T^*$ in $G_1'$ with norm $\leq C\|z\|_{G_2}.$
We extend $\phi$ to all of $G_1'$ using the Hahn-Banach theorem. Then $\phi\in G_1''$ with norm
$\|\phi\|_{G_1''}\leq C\|z\|_{G_2}.$

Recall the definition of $*$ as it applies in this situation. We say that $\phi\in G_1''$ belongs to
$D_{(T^*)^*}\subset G_1''$ if there exists a constant $c$ so that $|\phi(T^*y)|\leq c\|y\|_{G_2'}$
for all $y\in D_{T^*}.$ Since $|\phi(T^*y)|=|y(z)|\leq \|z\|_{G_2}\|y\|_{G_2'}$ and our $z$ is fixed,
it follows that $\phi\in D_{T^{**}}\subset G_1''.$
By
reflexivity there is an $x\in G_1$ with norm $\leq C\|z\|$ such that $u(x)=\phi(u)$ for all $u\in G_1'.$ Moreover $x\in D_T.$ So whenever $y\in D_{T^*},$ 
$$
y(z)=\phi(T^*y)=(T^*(y))(x)=y(Tx).
$$
So we have shown that for any fixed $z\in F,$ there is an $x\in D_T, \|x\|_{G_2}\leq C \|z\|_{G_1}$ so that
$y(z)=y(Tx)$ for all $y\in D_{T^*}.$
Since $D_{T^*}$ is dense in $G_2',$ it follows that $y(z-Tx)=0$ for all $y\in G_2'.$ 
By the Hahn-Banach theorem this implies that $z-Tx=0.$
\end{proof}

Let $N_T$ denote the nullspace of $T.$ Clearly $N_T$ is contained in $D_T$ and $N_T$ is a closed subspace.

Next we study the case of three reflexive Banach spaces, $G_1,G_2,G_3.$ Also we consider closed, densely defined linear operators, $T:G_1\rightarrow G_2$ and $S:G_2\rightarrow G_3$
satisfying the condition that $T(G_1)\subset N_S.$ Hence $S\circ T$ is defined on $D_T$ and
$S\circ T \equiv 0.$
We call this for short an $S,T$ system.\\

{\bf Definition 4.1.1a}
An $S,T$ system satisfies the Basic Estimate if there exists a constant $C$
 so that for every $y\in D_{T^*}$ and every $u\in D_S$ we have that
$$
(4.1.5)'\;\:|y(u)|\leq C(\|T^*(y)\|_{G_1'}\|u\|_{G_2}+\|y\|_{G_2'}\|S(u)\|_{G_3}).
$$

{\bf Theorem 4.1.1b}
If we have an $S,T$ system satisfying the basic estimate, then $T(G_1)=N_S.$\\

\begin{proof}
The space $F=N_S$ satisifes the condition of Theorem 4.1.1', namely, $F$ is a closed subspace of $G_2$ and it contains $R_T.$ We apply the Basic estimate to $y\in D_{T^*}$ and $u\in F.$ Then
$|y(u)|\leq C\|T^*(y)\|_{G_1'}\|u\|_{G_2}$ since $S(u)=0.$ This is the estimate (4.1.1)' in Lemma 4.1.1'. The theorem follows.
\end{proof}

\section{Hormander 4.1-$L^p$ spaces}

We introduce some Banach spaces. Let $\Omega$ be an open subset of $\mathbb C^n.$
Let $1<r,s<\infty, \frac{1}{r}+\frac{1}{s}=1.$
If $\phi$ is a continuous real function on $\Omega,$ we define 
$$
L^r(\Omega,\phi)=\{f:\Omega\rightarrow \mathbb C, \int_\Omega |f|^re^{-\phi}d\lambda
=:\|f\|^r_{L^r(\Omega,\phi)}<\infty\}.
$$
Here $d\lambda$ is Lebesgue measure and $f$ is assumed to be measurable and locally in $L^r,$
$f\in L^r_{loc}(\Omega).$ We define similarly $L^s(\Omega,\phi).$ We know that
the dual of $L^r(\Omega,\phi)$ is $L^s(\Omega,\phi)$ and vice versa. In particular these spaces are reflexive. We have for $f\in L^r(\Omega,\phi)$ and $g\in L^s(\Omega,\phi)$ that
$g(f)=\int_\Omega f\overline{g}e^{-\phi}.$ We write $g(f)=<f,g>_{\phi}$
and get $<f,g>_{\phi}=\overline{<g,f>}_{\phi}.$ Also we have $|g(f)|\leq \|f\|_{L^r(\Omega,\phi)}
\|g\|_{L^s(\Omega,\phi)}.$

We can do the same for $(p,q)$ forms.

Let $L^r_{(p,q)}(\Omega,\phi)$ denote the space of
forms of type $(p,q)$ with coefficients in $L^r(\Omega, \phi).$
$$
f=\sum'_{|I|=p}\sum'_{|J|=q} f_{I,J}dz^I\wedge d\overline{z}^J.
$$
where $\sum'$ refers to summing over strictly increasing multiindices.
We set $|f|^r=\sum'_{I,J}|f_{I,J}|^r$ and  $\|f\|^r_\phi=\int |f|^re^{-\phi}.$
Then the dual of $L^r_{(p,q)}(\Omega,\phi)$ is $L^s_{(p,q)}(\Omega,\phi)$ and
we have $<f,g>_\phi= \int \sum'_{I,J} \int f_{I,J}\overline{g}_{IJ}e^{-\phi}.$
Set $f\cdot g=\sum_{I,J} f_{I,J}\overline{g}_{I,J}$ for the pointwise product.

Similarly we define $L^r_{(p,q)}(\Omega)_{loc}$. Let $D(\Omega)$ denote the space of $\mathcal C^\infty$ functions on $\Omega$ with compact support in $\Omega.$ Similarly we define
$D_{(p,q)}(\Omega).$ We observe that $D_{(p,q)}(\Omega)$ is dense in $L^r_{(p,q)}(\Omega,\phi).$

Let $1<r',s'<\infty$ with $\frac{1}{r'}+\frac{1}{s'}=1.$
Let $\phi_1,\phi_2$ be continuous functions on $\Omega.$
Consider the operator $\overline{\partial}$. This gives rise to a linear densely defined
closed operator
$$
T:L^r_{(p,q)}(\Omega,\phi_1)\rightarrow L^{r'}_{(p,q+1)}(\Omega,\phi_2).
$$
An element $u\in L^r_{(p,q)}(\Omega,\phi_1)$ is in $D_T$ if $\overline{\partial}u$, defined in the sense of distributions belongs to $L^{r'}_{(p,q+1)}(\Omega,\phi_2)$ and then we set $Tu=\overline{\partial }u.$
The operator is densely defined since it is defined on $D_{(p,q)}(\Omega).$ The closedness is as in Example 2.1. Our goal is to show that the range of $T$ consists of all those $f$ for which $\overline{\partial} f=0$ for some choices of $\phi_j $  and $r,r'$.

Consider the case of 3 Banach spaces, 
\bea
G_1 &  = & L^{r_1}_{(p,q)}(\Omega,\phi_1),\\
G_2 & = & L^{r_2}_{(p,q+1)}(\Omega,\phi_2),\\
 G_3 & = & L^{r_3}_{(p,q+2)}(\Omega,\phi_3)\\
\eea
 with $\overline{\partial}$ operators $T:G_1\rightarrow G_2$ and $S:G_2\rightarrow G_3.$ This is then an $(S,T)$ system and our
goal is to prove the Basic Estimate under suitable conditions.

We will assume that $1<r_3\leq r_2\leq r_1<\infty.$

{\bf Lemma 4.1.$3_a$}
Let $\eta_\nu$ be a sequence of $\mathcal C^\infty$ functions with compact support in $\Omega.$
Suppose that $0\leq \eta_\nu\leq 1$ and that on any given compact subset of $\Omega$ we have $\eta_\nu=1$ for all large $\nu.$ Suppose that  
$$
(4.1.6)'\;\; e^{-\phi_3}\sum_{k=1}^n |\partial \eta_\nu/\partial \overline{z}_k|^{r_3}\leq e^{-r_3\phi_2/r_2}.
$$
If $r_3<r_2$ we add the extra condition that $\Omega$ has a finite volume. Then for every $f\in D_S$ the sequence $\eta_\nu f\rightarrow f$ in $G_2$. Moreover $\eta_\nu f\in D_S$
and $S(\eta_\nu f) \rightarrow S(f)$ in $G_3.$ \\

\begin{proof}
The sequence $|\eta_\nu f|\leq |f|$ and $\eta_\nu f$ converges pointwise to $f$,  so $\int |\eta_\nu f-f|^{r_2}e^{-\phi_2}d\lambda \rightarrow 0$ by the Lebesgue dominated convergence theorem.

We have that $S(\eta_\nu f)=\eta_\nu S(f)+\overline\partial \eta_\nu\wedge f$ in the sense of distributions and $S(f)\in L^{r_3}(\phi_3)$ so to show that $\eta_\nu f\in D_S$ we need to show that
$\overline{\partial} \eta_\nu\wedge f\in L^{r_3}(\phi_3).$
We have the pointwise estimate that $|\overline\partial \eta_\nu\wedge f|^{r_3} e^{-\phi_3}\leq
|f|^{r_3}e^{-r_3\phi_2/r_2}.$ We show that the right side is an $L^1$ function on $\Omega.$
If so, $\overline{\partial} \eta_\nu\wedge f\in L^{r_3}(\phi_3)$ and the Lebesgue dominated convergence theorem implies that the integral converges to $0.$ 
If $r_2=r_3$ the function is in $L^1$ by the hypothesis that $f\in G_2.$ Suppose that $r_2>r_3.$
We then get if $\frac{r_3}{r_2}+ \frac{1}{t}=1,$
\bea
\int_\Omega |f|^{r_3}e^{-r_3\phi_2/r_2} & = &  \left(\int_\Omega (|f|^{r_3}e^{-r_3\phi_2/r_2})^{r_2/r_3}\right)^{r_3/r_2} (\int_\Omega d\lambda)^{1/t}\\
& = &  \left(\int_\Omega |f|^{r_2}e^{-\phi_2}\right)^{r_3/r_2} (\int_\Omega d\lambda)^{1/t}\\
& < & \infty\\
\eea

Then the lemma follows since $\eta_\nu S(f)\rightarrow S(f)$ in $G_3.$
\end{proof}

{\bf Lemma 4.1.$3_b$}
 Let $\eta_\nu$ be a sequence of $\mathcal C^\infty$ functions with compact support in $\Omega.$
Suppose that $0\leq \eta_\nu\leq 1$ and that on any given compact subset of $\Omega$ we have $\eta_\nu=1$ for all large $\nu.$  Suppose that  
$$
(4.1.6)''\;\;e^{-\phi_2}\sum_{k=1}^n |\partial \eta_\nu/\partial \overline{z}_k|^{r_2}\leq e^{-r_2\phi_1/r_1}.
$$
If $r_2<r_1$ we add the extra condition that $\Omega$ has a finite volume. 
Then for every $f\in D_T^*$ the sequence $\eta_\nu f\rightarrow f$ in $G_2'$. Moreover $\eta_\nu f\in D_{T^*}$
and $T^*(\eta_\nu f) \rightarrow T^*(f)$ in $G_1'.$ \\

\begin{proof}
Suppose that $f\in D_{T^*}.$ By the Lebesgue dominated convergence theorem, $\eta_\nu f\rightarrow f$ in $G_2'.$ We show that $\eta f\in D_{T^*}$ if $\eta$ is smooth with compact support. Let $u\in D_T.$ Then $Tu\in G_2$ and
\bea
(\eta f)(Tu) & = & <Tu,{\eta f}>_{\phi_2}\\
& = & <\eta Tu,{f}>_{\phi_2}\\
& = & <T(\eta u)-\overline{\partial}\eta \wedge u,{f}>_{\phi_2}\\
& = & <T(\eta u),{f}>_{\phi_2}-<\overline{\partial}\eta \wedge u,{f}>_{\phi_2}\\
& = & <\eta u,{T^*f}>_{\phi_1}-<\overline\partial \eta \wedge u,{f}>_{\phi_2}\\
& = & <u,{\eta T^*f}>_{\phi_1}-<\overline\partial \eta \wedge u,{f}>_{\phi_2}\\
\eea

So 
\bea
|(\eta f)(Tu)| & \leq & \|\eta T^*f\|_{G_1'}\|u\|_{G_1}+\|f\|_{G_2'}\left( \int |\overline{\partial} \eta \wedge u|^{r_2} e^{-\phi_2}\right)^{1/r_2}\\
& \leq & \| T^*f\|_{G_1'}\|u\|_{G_1}+\|f\|_{G_2'}\left( \int |\overline{\partial} \eta \wedge u|^{r_2} e^{-\phi_2}\right)^{1/r_2}\\
& \leq & \| T^*f\|_{G_1'}\|u\|_{G_1}+\|f\|_{G_2'}\left( \int |u|^{r_2} (\sum_j \frac{{\partial} \eta}{\partial \overline{z}_j}|^{r_2} )  e^{-\phi_2}\right)^{1/r_2}\\
& \leq & \| T^*f\|_{G_1'}\|u\|_{G_1}+C_\eta\|f\|_{G_2'}\left( \int |u|^{r_2} e^{-r_2\phi_1/r_1}\right)^{1/r_2}\\
& & \mbox{where}\;C_\eta=1\; \mbox{for}\;\eta=\eta_\nu\\
\eea

If $r_1=r_2$ we see that

\bea
|(\eta f)(Tu)|&  \leq  & \| T^*f\|_{G_1'}\|u\|_{G_1}+C_\eta\|f\|_{G_2'}\left( \int |u|^{r_1} e^{-\phi_1}\right)^{1/r_1}\\
&  =  & \left(\| T^*f\|_{G_1'}+C_\eta\|f\|_{G_2'}\right) \|u\|_{G_1}\\
\eea

So then $\eta f\in D_{T^*}.$ Next, consider the case $r_1>r_2.$

\bea
|(\eta f)(Tu)|&  \leq  & \| T^*f\|_{G_1'}\|u\|_{G_1}+C_\eta\|f\|_{G_2'}\left( \int |u|^{r_2} e^{-r_2\phi_1/r_1}\right)^{1/r_2}\\
& = & \| T^*f\|_{G_1'}\|u\|_{G_1}\\
& + & C_\eta\|f\|_{G_2'}\left(\left( \int \left(|u|^{r_2} e^{-r_2\phi_1/r_1}\right)^{r_1/r_2}\right)^{r_2/r_1}|\Omega|^{1/t'}\right)^{1/r_2}\\
& \leq & \| T^*f\|_{G_1'}\|u\|_{G_1}+C_\eta\|f\|_{G_2'}|\Omega|^{1/(r_2t')}\left( \int |u|^{r_1} e^{-\phi_1}\right)^{1/r_1}\\
&  =  & \left(\| T^*f\|_{G_1'}+C_\eta\|f\|_{G_2'}|\Omega|^{1/(r_2t')}\right) \|u\|_{G_1}\\
\eea

This shows that $\eta f\in D_{T^*}$ also in the case when $r_1>r_2.$ It remains to show that $T^*(\eta_\nu f)\rightarrow T^* f.$
It suffices to show that $T^*(\eta_\nu f)-\eta_\nu T^*f\rightarrow 0.$ Let $u\in D_T.$
Then
\bea
(T^*(\eta_\nu f)-\eta_\nu T^*f)u & = & <u,{T^*(\eta_\nu f)-\eta_\nu T^*f}>_{\phi_1}\\
& = & <Tu,{\eta_\nu f}>_{\phi_2}-<u,{\eta_\nu T^*f}>_{\phi_1}\\
& = & <\eta_\nu Tu,{ f}>_{\phi_2}-<u,{\eta_\nu T^*f}>_{\phi_1}\\
& = & < T(\eta_\nu u)-\overline{\partial} \eta_\nu\wedge u,{f}>_{\phi_2}-<u,{\eta_\nu T^*f}>_{\phi_1}\\
& = & <\eta_\nu u,{ T^*f}>_{\phi_1}-<\overline{\partial} \eta_\nu\wedge u,{f}>_{\phi_2}-<u,{\eta_\nu T^*f}>_{\phi_1}\\
& = & -<\overline{\partial} \eta_\nu\wedge u,{f}>_{\phi_2}\\
\eea

Hence
$$
|(T^*(\eta_\nu f)-\eta_\nu T^*f)u|\leq \int |f||\overline\partial \eta_\nu \wedge u|e^{-\phi_2}.
$$

So in particular we have for any $u\in D_{p,q}$ that
$$
|\int (T^*(\eta_\nu f)-\eta_\nu T^*f)\cdot u e^{-\phi_1}|\leq \int |f||\overline\partial \eta_\nu| | u|e^{-\phi_2}.
$$

This implies the pointwise a.e. estimate
$$
|(T^*(\eta_\nu f)-\eta_\nu T^*f) e^{-\phi_1}|\leq  |f||\overline\partial \eta_\nu| e^{-\phi_2}.
$$

So
\bea
|(T^*(\eta_\nu f)-\eta_\nu T^*f)| & \leq  & |f||\overline\partial \eta_\nu| e^{\phi_1-\phi_2}\\
|(T^*(\eta_\nu f)-\eta_\nu T^*f)|^{s_1}e^{-\phi_1} & \leq  & |f|^{s_1}|\overline\partial \eta_\nu|^{s_1} e^{s_1(\phi_1-\phi_2)}e^{-\phi_1}\\
& = & |f|^{s_1}\left(\sum_k |\frac{\partial \eta_{\nu}}{\partial \overline{z}_k}|\right)^{s_1} e^{s_1(\phi_1-\phi_2)}e^{-\phi_1}\\
& \leq & |f|^{s_1}\left(n\sum_k |\frac{\partial \eta_{\nu}}{\partial \overline{z}_k}|^{r_2}\right)^{s_1/r_2} e^{s_1(\phi_1-\phi_2)}e^{-\phi_1}\\
& \leq & |f|^{s_1}\left(ne^{-r_2\phi_1/r_1}e^{\phi_2}\right)^{s_1/r_2} e^{s_1(\phi_1-\phi_2)}e^{-\phi_1}\\
& = & |f|^{s_1}n^{s_1/r_2}e^{\phi_1(-s_1/r_1+s_1-1)}e^{\phi_2(s_1/r_2-s_1)}\\
& = & |f|^{s_1}n^{s_1/r_2}e^{\phi_1(-s_1(1-1/s_1)+s_1-1)}e^{\phi_2(s_1(1-1/s_2)-s_1)}\\
& = & |f|^{s_1}n^{s_1/r_2}e^{-s_1\phi_2/s_2}\\
\eea

Since the functions 
$|(T^*(\eta_\nu f)-\eta_\nu T^*f)|$ converge pointwise to zero, it suffices by the dominated convergence theorem  to show that $|f|^{s_1}n^{s_1/r_2}e^{-s_1\phi_2/s_2}$ is an $L^1$ function.
If $s_1=s_2$ this follows since $f\in L^{s_2}_{(p,q+1)}(\Omega,\phi_2).$
If $s_1<s_2$ it follows also because

\bea
\int |f|^{s_1}e^{-s_1\phi_2/s_2} & \leq & \left(\int \left(|f|^{s_1}e^{-s_1\phi_2/s_2} \right)^{s_2/s_1}\right)^{s_1/s_2}|\Omega|^{1/t''}\\
& < & \infty\\
\eea

\end{proof}

Next  we will study smoothing.\\

We will use Minkowski's integral inequality. See Stein, Elias (1970). Singular integrals and differentiability properties of functions. Princeton University Press.

{\bf Theorem 4.1.$4_0$}
Let $F(x,y)\geq 0$ be a measurable function on the product of two measure spaces
$S_1,S_2$ with positive measures $d\mu_1(x),d\mu_2(y)$ respectively.
Let $1\leq r<\infty.$
Then
$$
\left(\int_{S_1} \left(\int_{S_2} F(x,y) d\mu_2(y) \right)^rd\mu_1(x)\right)^{1/r}
\leq \int_{S_2} \left(\int_{S_1} F^r(x,y) d\mu_1(x) \right)^{1/r}d\mu_2(y)
$$

The following is the smoothing theorem.

{\bf Lemma 4.1.$4_a$}
Let $\chi$ be a smooth function with compact support in $\mathbb R^N,$ with $\int \chi(x)dx=1$
and set $\chi_\epsilon(x)=\frac{1}{\epsilon^N}\chi(\frac{x}{\epsilon}).$ Let $1\leq r<\infty.$ 
If $g\in L^r(\mathbb R^N)$ then the convolution $g*\chi_\epsilon $ satisfies
\bea
(g*\chi_\epsilon)(x) & := & \int_{\mathbb R^N} g(y)\chi_\epsilon(x-y)dy\\
& = & \int g(x-y)\chi_\epsilon (y) dy\\
& = & \int g(x-\epsilon y)\chi(y)dy\\
\eea
and is a $\mathcal C^\infty$ function such that $\|g*\chi_\epsilon -g\|_{L^r}\rightarrow 0$
when $\epsilon\rightarrow 0.$
The support of $g*\chi_\epsilon$ has no points at distance $>\epsilon$ from the support of $g$
if the support of $\chi$ lies in the unit ball.\\

\begin{proof}
The equalities for $g*\chi_\epsilon(x)$ are obvious. The first integral shows that $g*\chi_\epsilon$
is $\mathcal C^\infty$ since $g$ is in $L^1_{loc}$ and since we can differentiate under the integral sign.
We apply Minkowski's integral inequality to the second integral.

\bea
\left(\int |g*\chi_\epsilon|^rdx\right)^{1/r} & = & \left( \int |\int g(x-y)\chi_\epsilon(y)dy|^rdx\right)^{1/r}\\
& \leq & \left( \int \left(\int |g(x-y)||\chi_\epsilon(y)|dy\right)^rdx\right)^{1/r}\\
& \leq &  \int \left(\int \left[|g(x-y)||\chi_\epsilon(y)|\right]^rdx\right)^{1/r}dy\\
& = & \int |\chi_\epsilon(y)| \|g\|_{L^r}\\
& = & C\|g\|_{L^r}, C:=\int|\chi|\\
\eea

This shows that $g*\chi_\epsilon\in L^r$ and that
$$
\|g*\chi_\epsilon\|_{L^r}\leq C\|g\|_{L^r}.
$$

Next pick a $\delta>0$ and choose a continuous function $h$ with compact support so that
$\|g-h\|_{L^r}<\delta.$
We then get

\bea
\|g*\chi_\epsilon-g\|_{L^r} & \leq & \|g*\chi_\epsilon-h*\chi_\epsilon \|_{L^r}+
\|h*\chi_\epsilon-h\|_{L^r}+\|h-g\|_{L^r}\\
& \leq & \|(g-h)*\chi_\epsilon\|_{L^r}+
\|h*\chi_\epsilon-h\|_{L^r}+\delta\\
& \leq & (C+1)\delta+
\|h*\chi_\epsilon-h\|_{L^r}\\
\eea

We have that 
$$
(h*\chi_\epsilon-h)(x) = \int (h(x-\epsilon y)-h(x))\chi(y)dy.
$$

Since $h$ is continuous with compact support and $\chi$ has compact support, it follows that
$h*\chi_\epsilon-h$ is supported in a ball $\|x\|\leq R$ for $\epsilon<1$ and converges uniformly to $0$ when $\epsilon\rightarrow 0.$ It follows that $g*\chi_\epsilon \rightarrow g$ in $L^r.$

The last assertion follows from the last of the integrals in the expression for $g*\chi_\epsilon.$
\end{proof}

{\bf Exercises}\\

1) We know that the following holds for functions. Show that it also holds for (p,q) forms. 
$$
|<f,g>_\phi|\leq |f|_r |g|_s.
$$

2) Let $T$ be the operator $\overline{\partial}: L^2_{p,q}(\Omega)\rightarrow L^2_{p,q+1}(\Omega).$ Show that $T$ is a closed and densely defined operator.\\

3) Consider 3 Banach spaces, 
\bea
G_1 &  = & L^{r_1}_{(p,q)}(\Omega,\phi_1),\\
G_2 & = & L^{r_2}_{(p,q+1)}(\Omega,\phi_2),\\
 G_3 & = & L^{r_3}_{(p,q+2)}(\Omega,\phi_3)\\
\eea
 with $\overline{\partial}$ operators $T:G_1\rightarrow G_2$ and $S:G_2\rightarrow G_3.$
where the $\phi_i$ are continuous functions and $1<r_i<\infty.$ Show that the range ot $T$ is in the domain of $S$. Show that $S\circ T=0.$

\section{Hormander 4.1-$L^p,L^2$ spaces}

{\bf Lemma 4.1.$4_b$}
Let $f_1,\dots,f_N\in L^1_{loc}(\mathbb R^N)$. Also suppose that the distribution
$\sum_{j=1}^N \frac{\partial f_j}{\partial x_j}\in L^1_{loc}.$
Then
$$
\sum_{j=1}^N \frac{\partial (f_j*\chi_\epsilon)}{\partial x_j}=\left(\sum_{j=1}^N \frac{\partial f_j}{\partial x_j}\right)*\chi_\epsilon.$$

\begin{proof}
Both sides are $\mathcal C^\infty$ functions. To show that they are equal, we show that for any $\phi\in \mathcal C^\infty_0$ that
$$\int \sum_{j=1}^N \frac{\partial (f_j*\chi_\epsilon)}{\partial x_j}(x) \phi(x)dx=
\int \left(\sum_{j=1}^N \frac{\partial f_j}{\partial x_j}\right)*\chi_\epsilon(x)\phi(x)dx.
$$

\bea
\int \left(\sum_{j=1}^N \frac{\partial f_j}{\partial x_j}\right)*\chi_\epsilon(x)\phi(x)dx & = & 
\int \left(\int \left(\sum_{j=1}^N \frac{\partial f_j}{\partial x_j}\right)(x-y)\chi_\epsilon(y)dy\right)\phi(x)dx\\
& =  & \int \chi_{\epsilon}(y) \left(\int \left(\sum_{j=1}^N \frac{\partial f_j}{\partial x_j}\right)(x-y)\phi(x)dx\right) dy\\
& =  & -\int \chi_{\epsilon}(y) \left(\int \sum_{j=1}^N f_j(x-y)\frac{\partial \phi}{\partial x_j}(x)dx\right) dy\\
& =  & -\int \left(\int \sum_{j=1}^N \frac{\partial \phi}{\partial x_j}(x)f_j(x-y)\chi_\epsilon(y)dy\right) dx\\
& =  & -\int \sum_{j=1}^N \frac{\partial \phi}{\partial x_j}(x)(f_j*\chi_\epsilon)(x)dx\\
& = & \int \sum_{j=1}^N \frac{\partial (f_j*\chi_\epsilon)}{\partial x_j}(x) \phi(x)dx\\
\eea

\end{proof}

{\bf Lemma 4.1.$3_c$}
Let $f\in D_S$ have compact support in $\Omega.$ Then $f*\chi_\epsilon\rightarrow f$ in $G_2$,
$f*\chi_\epsilon \in D_S$ and $S(f*\chi_\epsilon)\rightarrow Sf$ in $G_3.$

\begin{proof}
By the smoothing theorem, $f*\chi_\epsilon\rightarrow f$ in $L^{r_2}$ and since $f$ has compact support, $f*\chi_\epsilon\rightarrow f$ in $\|\cdot\|_{G_2}.$ Since $f*\chi_\epsilon$ is smooth with compact support, $f*\chi_\epsilon\in D_S.$ Furthermore, $Sf$ is a form such that each coefficient can be written as an expression $\sum_j \frac{\partial f_j}{\partial x_j}$ where each $f_j$ is a
finite linear combination of the coefficients of $f.$ [Recall that $\frac{\partial f_{I,J}}{\partial \overline{z}_j}=\frac{1}{2}\frac{\partial f_{I,J}}{\partial x_j}+\frac{i}{2}\frac{\partial f_{I,J}}{\partial y_j}$.] Hence, by Lemma 4.1.$4_b$,  $(Sf)*\chi_\epsilon=S(f*\chi_\epsilon).$ By the
smoothing theorem $(Sf)*\chi_\epsilon\rightarrow Sf$ in $L^{r_3}_{loc}.$ Since $f$ has compact support, $\phi_3$ is bounded so $(Sf)*\chi_\epsilon \rightarrow Sf$ in $G_3$. Therefore
$S(f*\chi_\epsilon)\rightarrow Sf$ in $G_3.$

\end{proof}

Next we prove a similar lemma for $T^*.$ We first do some preparations. In this lemma $\phi_1,\phi_2$ are continuous functions on $\Omega.$

{\bf Lemma 4.1.$3_d$}
Let $f=\sum'_{|I|=p}\sum'_{|J|=q+1}f_{I,J}dz^I\wedge d\overline{z}^J
\in L^{s_2}_{(p,q+1)}(\Omega,\phi_2)$ and suppose that $f\in D_{T^*}.$
Let $T^*(f)=\sum'_{I,K}g_{I,K}dz^I\wedge d\overline{z}^K\in L^{s_1}_{(p,q)}(\Omega,\phi_1).$
Then 
$$(4.1.9)\;\;g_{I,K}=(-1)^{p-1}e^{\phi_1}\sum_{j=1}^n \frac{\partial(e^{-\phi_2} f_{I,jK})}{\partial z_j}.$$
In particular, the distribution $\sum_{j=1}^n \frac{\partial(e^{-\phi_2} f_{I,jK})}{\partial z_j}\in L^{s_1}_{loc}.$\\

\begin{proof}
Let $u=\sum'_{|I|=p}\sum'_{|K|=q} u_{I,K}dz^I\wedge d\overline{z}^K \in D_{p,q}(\Omega).$
So in particular, $u\in D_T.$
Then $(T^*f)(u)=f(Tu).$ Here
\bea
Tu & = & \sum'_{|I|=p}\sum'_{|K|=q}\frac{\partial u_{I,K}}{\partial \overline{z}_j} d\overline{z}_j\wedge
dz^I\wedge d\overline{z}^K\\
& = & (-1)^p \sum'_{|I|=p}\sum'_{|K|=q}\frac{\partial u_{I,K}}{\partial \overline{z}_j}
dz^I\wedge
d\overline{z}_j\wedge d\overline{z}^K.\\
\eea
If $I_1$ is a permutation of $I$ and $J_1$ is a permutation of $J$ we write
$f_{I_1,J_1}=\epsilon_{I,J}^{I_1,J_1}f_{I,J}$ where $\epsilon$ is the signature of the permutation.
So, for example,  the signature is $-1$ if only two indices are interchanged.
In particular a term $f_{I,jK}=0$ if $j\in K.$
We get
$$
(T^*f)(u)=\int \sum_{I,K} u_{I,K} \overline{g}_{I,K}e^{-\phi_1}= f(Tu)=
\int (-1)^p \sum_{I,K,j}{\frac{\partial u_{I,K}}{\partial \overline{z}_j}}\overline{f_{I,jK}}e^{-\phi_2}.$$

Hence for all smooth functions $\psi$ with compact support, we have for each $I,K$ that

$$
\int \overline{g_{I,K}} e^{-\phi_1} \psi=
(-1)^{p} \int \sum_{j}\overline{f_{I,jK}}e^{-\phi_2}\frac{\partial \psi}{\partial {z}_j}.$$

Hence

$$
\int g_{I,K} e^{-\phi_1} \overline{\psi}=
(-1)^{p} \int \sum_{j}f_{I,jK}e^{-\phi_2}\frac{\partial \overline{\psi}}{\partial {\overline{z}}_j}.$$

Therefore
$$g_{I,K}e^{-\phi_1}=(-1)^{p-1}\sum_{j=1}^n \frac{\partial ( e^{-\phi_2} f_{I,jK})}{\partial z_j} \in L^{s_1}_{loc}.$$
\end{proof}

{\bf Corollary 4.1.$3_e$}
If $\phi_2\in \mathcal C^\infty(\Omega)$, then
$$
g_{I,K}=(-1)^{p} e^{\phi_1-\phi_2}\sum_j\frac{\partial \phi_2}{\partial z_j}f_{I,jK}+
(-1)^{p-1}e^{\phi_1-\phi_2}\sum_j \frac{\partial f_{I,jK}}{\partial z_j}
$$ where the distribution  $\sum_j \frac{\partial f_{I,jK}}{\partial z_j}$  is in $L^{s_1}_{loc}\supset L^{s_2}_{loc}.$\\

\begin{proof}
The distribution $\sum_j\frac{\partial (f_{I,jK}e^{-\phi_2})}{\partial z_j}\in L^{s_1}_{loc}.$
Let $\phi$ be a test function.
\bea
\int \left(\sum_j\frac{\partial (f_{I,jK}e^{-\phi_2})}{\partial z_j}\right)\phi & = &
- \sum_j \int f_{I,jK} e^{-\phi_2} \frac{\partial \phi}{\partial z_j}\\
& = &- \sum_j \int f_{I,jK}  \frac{\partial (e^{-\phi_2}\phi)}{\partial z_j}- \sum_j \int f_{I,jK} e^{-\phi_2} \frac{\partial \phi_2}{\partial z_j}\phi\\
&  = & \left( \sum_j  \frac{\partial f_{I,jK}}{\partial z_j}\right)(e^{-\phi_2}\phi)-\left(\sum_j  f_{I,jK} e^{-\phi_2} \frac{\partial \phi_2}{\partial z_j}\right)(\phi)\\
&  = & \left(e^{-\phi_2}\sum_j   \frac{\partial f_{I,jK}}{\partial z_j}\right)(\phi)-\left(e^{-\phi_2}\sum_j  f_{I,jK}  \frac{\partial \phi_2}{\partial z_j}\right)(\phi)\\
\eea

The expression on the left is in $L^{s_1}_{loc}$ and the second expression on the right is in $L^{s_2}_{loc}$. Hence the first expression on the right is in $L^{s_1}_{loc}.$

\end{proof}

{\bf Lemma 4.1.$3_f$}
Suppose that $f=\sum'_{|I|=p,|J|=q+1}f_{I,J}dz^I\wedge d\overline{z}^J\in D_{p,q+1}, \phi_2\in \mathcal C^{\infty}.$ Then
$f\in D_{T^*}.$\\

\begin{proof}
Let $g=\sum'_{I,K}g_{I,K}dz^I \wedge d\overline{z}^J$ where
$g_{I,K}=(-1)^{p-1}e^{\phi_1}\sum_{I,K,j}\frac{\partial (f_{I,jK}e^{-\phi_2})}{\partial z_j}.$
To show that $f\in D_{T^*},$ we prove that for any $u\in D_T, f(Tu)=g(u).$
We write $u=\sum'_{|I|=p}\sum'_{|K|=q}u_{I,K}dz^I\wedge d\overline{z}^K\in G_1,$ and
$$Tu= (-1)^p\sum'_{|I|=p}\sum'_{|K|=q} \frac{\partial u_{I,K}}{\partial \overline{z}_j} dz^I \wedge d\overline{z}_j\wedge d\overline{z}^K.$$ Then
\bea
f(Tu) & = & <Tu,f>_{\phi_2}\\
& = & \int (-1)^p\sum'_I\sum'_J\sum_{jK=J}\epsilon^{jK}_{J}\overline{f}_{I,J}\frac{\partial u_{I,K}}{\partial \overline{z}_j} e^{-\phi_2}\\
& = & (-1)^{p-1}\sum'_I \sum'_J \int \sum_{jK=J}\epsilon^{jK}_{J}\overline{\frac{\partial (f_{I,J}e^{-\phi_2})}{\partial z_j}}{u}_{I,K}\\
& = & (-1)^{p-1}\sum'_I \sum_{jK} \int \overline{\frac{\partial (f_{I,jK}e^{-\phi_2})}{\partial z_j}}{u}_{I,K}\\
& = & (-1)^{p-1}\sum'_I \sum'_{K} \int \left(\overline{e^{\phi_1}\sum_j\frac{\partial (f_{I,jK}e^{-\phi_2})}{\partial z_j}}\right){u}_{I,K}e^{-\phi_1}\\
& = & <u,g>\\
& = & g(u)\\
\eea

\end{proof}

{\bf Lemma 4.1.$3_g$}
Suppose that  $\phi_2\in \mathcal C^\infty.$ Let $f\in D_{T^*}$ have compact support. Then $f*\chi_\epsilon\rightarrow f$ in $G_2'$. Moreover, $f*\chi_\epsilon \in D_{T^*}$ and
$T^*(f*\chi_{\epsilon})\rightarrow T^*f$ in $G_1'.$\\

\begin{proof}
Since $f*\chi_\epsilon\in D_{p,q+1}$ for small $\epsilon,$ we have by Lemma 4.1.$3_f$ that $f*\chi_\epsilon\in D_{T^*}.$
Also by the smoothing theorem, $f*\chi_\epsilon \rightarrow f$ in $ G_2'. $
By Corollary 4.1.$3_e$ we can write $T^*(f*\chi_\epsilon)=\sum'_{I,K} g_{I,K}dz^I\wedge d\overline{z}^K$ where
$$
g_{I,K}  =  (-1)^{p}e^{\phi_1-\phi_2} \sum_j \frac{\partial \phi_2}{\partial z_j} f_{I,jK}*\chi_\epsilon\\
+(-1)^{p-1} e^{\phi_1-\phi_2}\sum_j \frac{\partial (f_{I,jK}*\chi_\epsilon) }{\partial z_j}.$$
The first term on the right converges to $(-1)^{p}e^{\phi_1-\phi_2} \sum_j \frac{\partial \phi_2}{\partial z_j} f_{I,jK}$ in $L^{s_2}$ by the smoothing Theorem. Hence it also converges in $L^{s_1}$. The second part can be written, using Lemma 4.1.$4_b$, as
$(-1)^{p-1} e^{\phi_1-\phi_2}\left(\sum_j \frac{\partial f_{I,jK}}{\partial z_j}\right) *\chi_\epsilon$
and converges to $$(-1)^{p-1} e^{\phi_1-\phi_2}\sum_j \frac{\partial f_{I,jK}}{\partial z_j} $$  in $L^{s_1}.$
\end{proof}

{\bf Theorem 4.1.$3_h$}  Suppose that $\phi_1,\phi_3$ are continuous and $\phi_2$ is $\mathcal C^\infty.$ Suppose that $\{\eta_\nu\}, 0\leq \eta_\nu\leq 1$ is a sequence of compactly supported $\mathcal C^\infty$ functions such that on any compact subset of $\Omega$ all $\eta_\nu=1$ except finitely many, as in (4.1.6)' and (4.1.6)''. Suppose also that if $r_1\neq r_3$ then $\Omega$ has bounded volume.
 Suppose that $f\in D_{T^*}$ and $g\in D_S.$ Then there exist sequences $\{f_n\},\{g_n\}\subset D_{(  p,q+1)}$ so that $f_n\in D_{T^*}, g_n\in D_S,$ $ f_n\rightarrow f$ in $G_2'$, $T^*f_n\rightarrow T^*f$ in $G_1'$, $g_n\rightarrow g$ in $G_2$ and $Sg_n\rightarrow S$ in $G_3.$

\begin{proof}
Let $\delta>0.$ Using Lemma 4.1.$3_a$ for $S$ and Lemma 4.1.$3_b$ for $T^*$, we can let $\nu_0$ be large enough that
$$
\|\eta_{\nu_0}f-f\|,\|T^*(\eta_{\nu_0}f)-T^*f\|,\|\eta_{\nu_0}g-g\|,\|S(\eta_{\nu_0}g)-Sg\|<\delta/2
$$
and $\eta_{\nu_0}f\in D_{T^*}, \eta_{\nu_0}g\in D_S.$
Then for $\epsilon>0$ small enough, $\hat{f}=(\eta_{\nu_0}f)*\chi_\epsilon$ and $\hat{g}=(\eta_{\nu_0}g)*\chi_\epsilon$ are in $D_{T^*}$ and $D_S$ respectively and 
$$
\|\hat{f}-f\|,\|T^*\hat{f}-T^*f\|,\|\hat{g}-g\|,\|S\hat{g}-Sg\| <\delta.
$$

\end{proof}

{\bf Corollary 4.1.$3_i$}
The $S,T$ system in  Theorem 4.1.$3_h$ satisfies the Basic Estimate if there is a constant $C$ so that for every
$y,u\in D_{(p,q+1)}$ thought of as elements of $G_2',G_2$ respectively, we have that
$$
|y(u)|\leq C\left( \|T^*(y)\|_{G_1'}\|u\|_{G_2}+\|y\|_{G_2'}\|S(u)\|_{G_3}\right).
$$

We now discuss the case when all $r_i,s_i=2.$ This is the Hilbert space case. Recall first a few facts about complex Hilbert spaces, $H$. We have an inner product $<x,y>$ for $x,y\in H.$ The inner product satisfies $$<ax,by>=a\overline{b}<x,y>.$$ The norm is $\|x\|^2=<x,x>.$ There is a natural identification between $H$ and the dual $H'.$ If $x\in H$, then $\lambda(x)$ defined by
$\lambda(x)(y)=<y,x>$ defines a continuous linear functional on $H.$ The map $\lambda:H\rightarrow H'$ is norm preserving and antilinear: $\lambda(cx)=\overline{c}\lambda(x).$ Conversely, if $g\in H',g\neq 0$, let $x\in N_g$, $\|x\|=1.$ Set $c=\overline{g(x)}.$ We show that $g=\lambda(cx).$
Any $y\in H$ can be written uniquely as $y=ax+z, z\in N_g.$
We get
\bea
g(y) & = & g(ax+z)\\
& = & ag(x)+g(z)\\
& = & a\overline{c}\\
& = & \overline{c}<ax+z,x>\\
& = & <y,cx>\\
& = & \lambda(cx)(y)\\
& \Rightarrow & \\
g & = & \lambda(cx)\\
\eea

In this case we can write the basic estimate as follows:

{\bf Theorem 4.1.$3_j$}
Suppose that $r=s=2$. Assume the conditions of Theorem 4.1.$3_h$. If for every $f\in G_2$ with $\lambda(f)\in D_{T^*}$ and $f\in D_S$ we have that
$$
(1)\;\;\|f\|\leq C(\|T^*f\|+\|Sf\|)
$$
for a fixed constant $C,$ independent of $f$, then we have that for all $y\in D_{T^*}, u\in D_S$ that
$$(2)\;\;\;|y(u)| \leq C(\|T^*y\|\|u\|+\|Su\|\|y\|).$$
We also have $(2)\Rightarrow (1).$\\

We note that $N_{T^*}=R_T^\perp.$ Namely, if $y\in D_{T^*}, T^*(y)=0$ and $x\in D_T,$
we have $(T^*(y))(x)=y(Tx)=0$. Conversely, if $y\in R_T^\perp$, then $y(Tx)=0$ for all $x\in D_T$, so
$y\in N_{T^*}.$ Note that we identify $R_T^\perp\subset H'$, the dual space with the orthogonal complement of $R_T\subset H,$ i.e. the vectors in $H$ perpendicular to the vectors in $R_T.$ If $y\in N_{T^*}$ and $x\in \overline{R_T}$ then $y(x)=0.$
On the other hand, if $y$ is in the subspace in $H'$ identified with $\overline{R_T}$ and $x$ is in the subspace of $H$ identified with $N_{T^*},$ we also have $y(x)=0.$

\begin{proof}
Assume (1). It suffices to prove (2) for all $y,u\in D_{p,q+1}$. We write $y=y_1+y_2$ where $\lambda^{-1}(y_1)\in \overline{R}_T$ and $y_2\in N_{T^*}.$ Similarly, we write $u=u_1+u_2$ where $u_1\in \overline{R_T}$ and $\lambda(u_2)\in N_{T^*}.$ Then $y_2(u_1)=0$ and $y_1(u_2)=0.$  Note that $y_1\in D_{T^*}$ since both $y$ and $y_2$ are. Similarly, $u_2\in D_S.$

It follows that
\bea
|y(u)| & = & |y_1(u_1)+y_2(u_1)+y_1(u_2)+y_2(u_2)|\\
& \leq  & |y_1||u_1|+|y_2||u_2|\\
& \leq & C(\|T^*y_1\|+\|Sy_1\|)\|u_1\|+\|y_2\|(\|T^*u_2\|+\|Su_2\|)\\
& = & C\|T^*y_1\|\|u_1\|+\|y_2\|\|Su_2\|\\
& =  & C\|T^*y\|\|u_1\|+\|y_2\|\|Su\|\\
& \leq  & C\|T^*y\|\|u\|+\|y\|\|Su\|\\
\eea
The reverse implication follows by applying (2) to the case $y=u\in D_{p,q}.$

\end{proof}

\section{Hormander 4.2}

We are trying to prove the basic estimate 
$$
\|f\| \leq C(\|T^*f\|+\|Sf\|).
$$
We need formulas for $\|Sf\|,\|T^*f\|.$

{\bf Definition 4.2.$1_a$}
Let $\phi\in \mathcal C^\infty(\Omega)$ be a real valued function. If $w\in \mathcal C^\infty(\Omega)$ we let
$\delta_j(w):= e^\phi \frac{\partial (we^{-\phi})}{\partial z_j}=\frac{\partial w}{\partial z_j}-w\frac{\partial \phi}{\partial z_j}.$\\

{\bf Lemma 4.2.$1_b$}
Let $w_1,w_2$ be $\mathcal C^\infty$ functions with compact support in $\Omega.$
Then $<w_1,\frac{\partial w_2}{\partial \overline{z}_k}>_{\phi}:=
\int w_1 \overline{\left(\frac{\partial w_2}{\partial \overline{z}_k}\right)}e^{-\phi}d\lambda=
\int (-\delta_kw_1)\overline{w}_2 e^{-\phi}d\lambda =:<-\delta_k w_1,w_2>_{\phi}.$\\

\begin{proof}
\bea
\int w_1 \overline{\frac{\partial w_2}{\partial \overline{z}_k}}e^{-\phi}d\lambda & = &
\int w_1  \frac{\partial \overline{w}_2}{\partial {z}_k}e^{-\phi}d\lambda\\
& = & -\int  \frac{\partial ({w}_1e^{-\phi})}{\partial {z}_k}\overline{w}_2 d\lambda\\
& = & -\int \frac{\partial {w}_1}{\partial {z}_k}\overline{w}_2 e^{-\phi}d\lambda+\int w_1\frac{\partial \phi}{\partial {z}_k}\overline{w}_2 e^{-\phi}d\lambda\\
& = & -\int \delta_kw_1\overline{w}_2 e^{-\phi} d\lambda\\
\eea

\end{proof} 

We next prove a commutation relation between $\delta_j$ and $\frac{\partial}{\partial \overline{z}_k}.$

{\bf Lemma 4.2.$1_c$}
Let $\psi$ be a smooth function. Then
$$
\delta_j (\frac{\partial \psi}{\partial \overline{z}_k})-
\frac{\partial}{\partial \overline{z}_k}(\delta_j(\psi))= \psi \frac{\partial ^2\phi}{\partial \overline{z}_k\partial z_j}.
$$

\begin{proof}

\bea
\delta_j(\psi_{\overline{z}_k})-(\delta_j\psi)_{\overline{z}_k} & = &
\psi_{\overline{z}_k,z_j}- \psi_{\overline{z}_k}\frac{\partial \phi}{\partial z_j}-
(\psi_{z_j}-\psi\frac{\partial \phi}{\partial z_j})_{\overline{z}_k} \\
& = & \psi_{\overline{z}_k,z_j}- \psi_{\overline{z}_k} \phi_{z_j}-
\psi_{z_j,\overline{z}_k}+\psi_{\overline{z}_k}\phi_{z_j} +\psi \phi_{z_j,\overline{z}_k}\\
& = & \psi \phi_{z_j,\overline{z}_k}\\
\eea

\end{proof}

{\bf Lemma 4.2.$1_d$}
Let $f,g$ be $\mathcal C^\infty$ functions with compact support in $\Omega.$
Then
\bea
\int \delta_j f\overline{\delta_k g}e^{-\phi} & = & -\int \frac{\partial (\delta_j f)}{\partial \overline{z}_k}\overline{g}e^{-\phi}\\
\eea

\begin{proof}
\bea
\int \delta_j f\overline{\delta_k g}e^{-\phi} & = & \overline{\int \delta_k g \overline{\delta_j f}  e^{-\phi}}\\
& = & -\overline{\int g\overline{  \left(\frac{\partial (\delta_j f)}{\partial \overline{z}_k }\right)}e^{-\phi}}\\
& = & -\int \overline{g}\frac{\partial (\delta_j f)}{\partial \overline{z}_k}e^{-\phi}\\
\eea

\end{proof}

The following lemma is immediate from Lemma 4.2.$1_b$:\\

{\bf Lemma 4.2.$1_e$}
Let $f,g$ be $\mathcal C^\infty$ functions with compact support. Then
\bea
\int \frac{\partial f}{\partial \overline{z}_k}\overline{\frac{\partial g}{\partial \overline{z}_j}   }e^{-\phi}
& = & -\int \delta_j(\frac{\partial f}{\partial \overline{z}_k})\overline{g}e^{-\phi}\\
\eea

{\bf Corollary 4.2.$1_f$}
Let $f,g$ be $\mathcal C^\infty$ functions with compact support. Then
\bea
\int \delta_j f\overline{\delta_k g}e^{-\phi} -\int \frac{\partial f}{\partial \overline{z}_k}
\overline{\frac{\partial g}{\partial \overline{z}_j}}e^{-\phi}
& = & \int f\overline{g}\frac{\partial^2 \phi}{\partial z_j\partial\overline{z}_k}e^{-\phi}\\
\eea

\begin{proof}
We combine Lemmas 4.2.$1_{c-e}$:
\bea
\int \delta_j f\overline{\delta_k g}e^{-\phi} -\int \frac{\partial f}{\partial \overline{z}_k}
\overline{\frac{\partial g}{\partial \overline{z}_j}}e^{-\phi}
& = & -\int \frac{\partial (\delta_j f)}{\partial \overline{z}_k}\overline{g}e^{-\phi}+
\int \delta_j(\frac{\partial f}{\partial \overline{z}_k})\overline{g}e^{-\phi}\\
& = & \int f\frac{\partial^2 \phi}{\partial z_j\partial\overline{z}_k}\overline{g}e^{-\phi}\\
\eea

\end{proof}

Let $f$ be in $D_{p,q+1}$. We calculate $Tf$.

{\bf Lemma 4.2.$1_g$.}
 If $f=\sum'_{I}\sum'_{J}f_{I,J}dz^I\wedge d\overline{z}^J$, then\\
$\overline{\partial} f = \sum'_I\sum'_J\sum_{j=1}^n \frac{\partial f_{I,J}}{\partial \overline{z}_j}
d\overline{z}_j\wedge dz^I\wedge d\overline{z}^J$ and\\
$
|\overline{\partial}f |^2= \sum'_{I,J}\sum_j |\frac{\partial f_{I,J}}{\partial \overline{z}_j}|^2
-\sum'_{I,K}\sum_{j,k}\frac{\partial f_{I,jK}}{\partial \overline{z}_k}
\overline{\frac{\partial f_{I,kK}}{\partial \overline{z}_j}}
$\\

\begin{proof}
The formula for $\overline\partial f$ is clear. We prove the second formula.
We deal first with the case $q+1=n$ which is degenerate. In this case $\overline{\partial} f=0$
for type reasons.
In the second term on the right side of the formula, we sum over all $K$ which are multiindices
of length $q$ so each $K$ misses one index, $j$. Recall that $f_{I,jK}=0$ if $j\in K.$ Hence in the last term you only sum over the case $j=k$. So this term becomes
$-\sum'_{I,K}|\frac{\partial f_{I,jK}}{\partial \overline{z}_j}|^2$ where $j$ is the missing index in $K.$
There is only one $J$ in this case and $f_{I,jK}=\epsilon^{jK}_{J}f_{I,J}.$
Since the term is squared, we can write the last sum as
$-\sum'_{I,J}\sum_j |\frac{\partial f_{I,J}}{\partial \overline{z}_j}|^2$ which is the same as the first term except for sign.

We continue by assuming that $q+1<n.$ For $M$ an increasing multiindex, $|M|=q+2\leq n$ and $j\in M,$
write $M^j$ to be the increasing multiindex with $j$ removed from $M.$  
We can then write 
\bea
\overline{\partial} f & = & \sum'_{|I|=p}\sum'_{|M|=q+2}\sum_{j\in M}\frac{\partial f_{I,M^j}}{\partial \overline{z}_j}d\overline{z}_j\wedge dz^I\wedge d\overline{z}^{M^j}\\
& = & (-1)^p\sum'_{|I|=p}\sum'_{|M|=q+2}\sum_{j\in M}\frac{\partial f_{I,M^j}}{\partial \overline{z}_j} \epsilon^{jM^j}_Mdz^I\wedge d\overline{z}^{M}\\
\eea

Hence we obtain that 

\bea
|\overline{\partial} f|^2  & = & \sum'_{|I|=p}\sum'_{|M|=q+2}\sum_{j,\ell\in M}\frac{\partial f_{I,M^j}}{\partial \overline{z}_j} \epsilon^{jM^j}_M\overline{\frac{\partial f_{I,M^\ell}}{\partial \overline{z}_\ell}} \epsilon^M_{\ell M^\ell}\\
& = & \sum'_{|I|=p}\sum'_{|M|=q+2}\sum_{j\in M}|\frac{\partial f_{I,M^j}}{\partial \overline{z}_j}|^2 \\
& + & \sum'_{|I|=p}\sum'_{|M|=q+2}\sum_{j,\ell\in M,j\neq \ell}\frac{\partial f_{I,M^j}}{\partial \overline{z}_j} \overline{\frac{\partial f_{I,M^\ell}}{\partial \overline{z}_\ell}} \epsilon^{jM^j}_M\epsilon^M_{\ell M^\ell} \\
& = & \sum'_{|I|=p}\sum'_{|J|=q+1}\sum_{j\notin J}|\frac{\partial f_{I,J}}{\partial \overline{z}_j}|^2 \\
& + & \sum'_{|I|=p}\sum'_{|K|=q}\sum_{j,\ell\notin K,j\neq \ell}\frac{\partial f_{I,lK}}{\partial \overline{z}_j} \epsilon^{lK}_{M^j}\overline{\frac{\partial f_{I,jK}}{\partial \overline{z}_\ell}}\epsilon^{ M^\ell}_{jK} \epsilon^{jM^j}_M\epsilon^M_{\ell M^\ell} \\
& = & \sum'_{|I|=p}\sum'_{|J|=q+1}\sum_{j\notin J}|\frac{\partial f_{I,J}}{\partial \overline{z}_j}|^2 \\
& + & \sum'_{|I|=p}\sum'_{|K|=q}\sum_{j,\ell\notin K,j\neq \ell}\frac{\partial f_{I,lK}}{\partial \overline{z}_j} \epsilon^{jlK}_{jM^j}\overline{\frac{\partial f_{I,jK}}{\partial \overline{z}_\ell}}\epsilon^{ \ell M^\ell}_{\ell jK} \epsilon^{jM^j}_M\epsilon^M_{\ell M^\ell} \\
& = & \sum'_{|I|=p}\sum'_{|J|=q+1}\sum_{j\notin J}|\frac{\partial f_{I,J}}{\partial \overline{z}_j}|^2 \\
& + & \sum'_{|I|=p}\sum'_{|K|=q}\sum_{j,\ell\notin K,j\neq \ell}\frac{\partial f_{I,lK}}{\partial \overline{z}_j} \overline{\frac{\partial f_{I,jK}}{\partial \overline{z}_\ell}}\epsilon^{ j\ell K}_{\ell jK} \\
& = & \sum'_{|I|=p}\sum'_{|J|=q+1}\sum_{j\notin J}|\frac{\partial f_{I,J}}{\partial \overline{z}_j}|^2 \\
& - & \sum'_{|I|=p}\sum'_{|K|=q}\sum_{j,\ell\notin K,j\neq \ell}\frac{\partial f_{I,lK}}{\partial \overline{z}_j} \overline{\frac{\partial f_{I,jK}}{\partial \overline{z}_\ell}} \\
\eea

For any fixed $I,$

\bea
\sum'_{|J|=q+1}\sum_{j\in J}|\frac{\partial f_{I,J}}{\partial \overline{z}_j}|^2 & = &
\sum'_{|K|=q, j\notin K} |\frac{\partial f_{I,jK}}{\partial \overline{z}_j}|^2\\
\eea

We add the left side to the first sum and the right side to the other sum.

Hence
\bea
|\overline{\partial} f|^2 & = & 
 \sum'_{|I|=p}\sum'_{|J|=q+1}\sum_{j}|\frac{\partial f_{I,J}}{\partial \overline{z}_j}|^2 \\
& - & \sum'_{|I|=p}\sum'_{|K|=q}\sum_{j,\ell\notin K}\frac{\partial f_{I,\ell K}}{\partial \overline{z}_j} \overline{\frac{\partial f_{I,jK}}{\partial \overline{z}_\ell}} \\
\eea

Since $f_{I,\ell K}=0$ if $\ell \in K$, we can drop the condition that $\ell\notin K$ and the same for $j\notin K.$ Hence 

\bea
|\overline{\partial} f|^2 
& = & \sum'_{|I|=p}\sum'_{|J|=q+1}\sum_{j}|\frac{\partial f_{I,J}}{\partial \overline{z}_j}|^2 \\
& - & \sum'_{|I|=p}\sum'_{|K|=q}\sum_{j,\ell}\frac{\partial f_{I,lK}}{\partial \overline{z}_j} \overline{\frac{\partial f_{I,jK}}{\partial \overline{z}_\ell}}\\
\eea

\end{proof}

We will next be more specific about our $L^2$ spaces. First pick some smooth function $\psi$ as in Theorem 4.1.$3_h$ for the $r=s=2$ case. We use weights $e^{-\phi_j}$ using smooth functions $\phi$ and $\psi$ as follows:

$$
\phi_1=\phi-2\psi,\phi_2=\phi-\psi,\phi_3=\phi.
$$
With these weights Theorem 4.1.$3_h$ applies to show that smooth compactly supported forms are dense
in the graph norms of $S$ and $T^*.$

Let $f\in D_{(p,q+1})$. Recall from Lemma 4.1.$3_d$ that if $$f=\sum'_{|I|=p}\sum'_{|J|=q+1}f_{I,J}dz^I\wedge d\overline{z}^J$$ then
\bea
T^*f & = & \sum'_{|I|=p}\sum'_{|K|=q}g_{I,K}dz^I\wedge d\overline{z}^K\\
&\mbox{where} &\\
g_{I,K} & = & (-1)^{p-1}e^{\phi_1}\sum_{j=1}^n \frac{\partial (e^{-\phi_2} f_{I,jK})}{\partial z_j}\\
\eea

{\bf Lemma 4.2.$1_h$}
\bea
e^{\psi}T^* f & = & (-1)^{p-1}\sum'_{I,K}\sum_j \delta_jf_{I,jK}dz^I\wedge d\overline{z}^K\\
& + & (-1)^{p-1}\sum'_{I,K}\sum_j f_{I,jK}\frac{\partial \psi}{\partial z_j}dz^I\wedge d\overline{z}^K\\
\eea

\begin{proof}
\bea
e^{\psi}T^* f & = & \sum'_{|I|=p}\sum'_{|K|=q}(e^\psi g_{I,K})dz^I\wedge d\overline{z}^K\\
e^\psi g_{I,K} & = & (-1)^{p-1}e^{\psi}e^{\phi_1}\sum_{j=1}^n \frac{\partial (e^{-\phi_2} f_{I,jK})}{\partial z_j}\\
& = & (-1)^{p-1}e^{\psi+\phi_1}\sum_{j=1}^n \frac{\partial (e^{-\phi+\psi} f_{I,jK})}{\partial z_j}\\
& = & (-1)^{p-1}e^{\psi+\phi_1}\sum_{j=1}^n e^\psi\frac{\partial (e^{-\phi} f_{I,jK})}{\partial z_j}\\
& + & (-1)^{p-1}e^{\psi+\phi_1}\sum_{j=1}^n e^{-\phi} f_{I,jK}e^{\psi}\frac{\partial \psi}{\partial z_j}\\
& = & (-1)^{p-1}\sum_{j=1}^n e^\phi\frac{\partial (e^{-\phi} f_{I,jK})}{\partial z_j} +  (-1)^{p-1}\sum_{j=1}^n  f_{I,jK}\frac{\partial \psi}{\partial z_j}\\
& = & (-1)^{p-1}\sum_{j=1}^n \delta_j(f_{I,jK}) +  (-1)^{p-1}\sum_{j=1}^n  f_{I,jK}\frac{\partial \psi}{\partial z_j}\\
e^{\psi}T^* f & = & \sum'_{|I|=p}\sum'_{|K|=q}(-1)^{p-1}\sum_{j=1}^n \delta_j(f_{I,jK})dz^I\wedge d\overline{z}^K\\
& + & \sum'_{|I|=p}\sum'_{|K|=q}(-1)^{p-1}\sum_{j=1}^n f_{I,jK}\frac{\partial \psi}{\partial z_j}
dz^I\wedge d\overline{z}^K\\
\eea
\end{proof}

We prove the large constant, small constant lemma:

{\bf Lemma 4.2.$1_i$.}
If $a,b$ are complex numbers and $c>0,$ then
$$2|ab|\leq c|a|^2+\frac{1}{c} |b|^2.$$
Moreover,
$$|a+b|^2 \leq (1+c)|a|^2+(1+\frac{1}{c})|b|^2.$$

\begin{proof}
The first inequality follows from
\bea
c|a|^2+\frac{1}{c}|b|^2-2|a||b| & = & (\sqrt{c}|a|-\sqrt{1/c}|b|)^2\\
& \geq & 0\\
\eea
The second inequality follows then from
\bea
|a+b|^2 & = & |a|^2+a\overline{b}+\overline{a}b+|b|^2\\
& \leq & |a|^2+2|a||b|+|b|^2\\
& \leq & (1+c)|a|^2+(1+\frac{1}{c})|b|^2\\
\eea
\end{proof}

We introduce the notation $2'=1+c,2''=1+\frac{1}{c}, c>0.$
Then statements involving $2',2''$ below are valid for any choice of $c.$

{\bf Lemma 4.2.$1_j.$}
$$
\sum'_{I,K}\sum_{j,k=1}^n \delta_j f_{I,jK}\overline{\delta_k f_{I,kK}}\leq
2'e^{2\psi} |T^*f|^2+2''|f|^2|\partial \psi|^2.$$

\begin{proof}
\bea
\sum'_{I,K}\sum_{j,k=1}^n \delta_j f_{I,jK}\overline{\delta_k f_{I,kK}} & = &
\sum'_{I,K}\left(\sum_{j=1}^n \delta_j f_{I,jK}\right)\left(\overline{\sum_{k=1}^n\delta_k f_{I,kK}}\right)\\
& = & |(-1)^{p-1}\sum'_{I,K}\sum_j \delta_jf_{I,jK}dz^I\wedge d\overline{z}^K|^2\\
& = & |e^\psi T^*f-(-1)^{p-1}\sum'_{I,K}\sum_j f_{I,jK}\frac{\partial \psi}{\partial z_j}dz^I\wedge d\overline{z}^K|^2\\
& = & \sum'_{I,K} |e^\psi g_{I,K}-(-1)^{p-1}\sum_j f_{I,jK}\frac{\partial \psi}{\partial z_j}|^2\\
& \leq & 2'\sum'_{I,K} |e^\psi g_{I,K}|^2+2''\sum'_{I,K}|\sum_j f_{I,jK}\frac{\partial \psi}{\partial z_j}|^2\\
& \leq & 2'e^{2\psi}|T^*f|^2+2''\sum'_{I,K}\sum_j |f_{I,jK}|^2
|\frac{\partial \psi}{\partial z_j}|^2\\
& = & 2'e^{2\psi}|T^*f|^2+2''\sum'_{I,K}\sum_{j\notin K} |f_{I,jK}|^2
|\frac{\partial \psi}{\partial z_j}|^2\\
& = & 2'e^{2\psi}|T^*f|^2+2''\sum'_{|I|=p,|J|=q+1}\sum_{j \in J} |f_{I,jJ^j}|^2
|\frac{\partial \psi}{\partial z_j}|^2\\
& = & 2'e^{2\psi}|T^*f|^2+2''\sum'_{I,J}\sum_{j \in J} |f_{I,J}|^2
|\frac{\partial \psi}{\partial z_j}|^2\\
& \leq & 2'e^{2\psi}|T^*f|^2+2''\sum'_{I,J} |f_{I,J}|^2
|\partial \psi|^2\\
\eea
\end{proof}

{\bf Lemma 4.2.$1_k$.}
\bea
\sum'_{I,K}\sum_{j,k=1}^n (\delta_j f_{I,jK}\overline{\delta_k f_{I,kK}}
& - & 
\frac{\partial f_{I,j K}}{\partial \overline{z}_k }
\overline{
\frac{\partial f_{I,kK}}{\partial \overline{z}_j}
}
)e^{-\phi}\\
& \leq  & 2'|T^*f|^2e^{-\phi_1}+|Sf|^2e^{-\phi_3}+2''|f|^2|\partial \psi|^2 e^{-\phi}\\
\eea

\begin{proof}
\bea
\sum'_{I,K}\sum_{j,k=1}^n (\delta_j f_{I,jK}\overline{\delta_k f_{I,kK}}
& - & \frac{\partial f_{I,j K}}{\partial \overline{z}_k }\overline{\frac{\partial f_{I,kK}}{\partial \overline{z}_j}})e^{-\phi}\\
& \leq & (2'e^{2\psi}|T^*f|^2+2''|f|^2|\partial \psi|^2)e^{-\phi}\\
& + & \left(|\overline{\partial}f|^2-\sum'_{I,J}\sum_j |\frac{\partial f_{I,J}}{\partial \overline{z}_j}|^2\right)e^{-\phi}\\
& = & 2'e^{2\psi-\phi}|T^*f|^2+2''|f|^2|\partial \psi|^2e^{-\phi}
 +  |Sf|^2e^{-\phi}\\
& = & 2'|T^*f|^2e^{-\phi_1}+2''|f|^2|\partial \psi|^2e^{-\phi}
 +  |Sf|^2e^{-\phi_3}\\
\eea
\end{proof}

{\bf Theorem 4.2.$1_\ell.$}
Let $\Omega$ be a pseudoconvex domain in $\mathbb C^n$. Let $0\leq \eta_\nu\leq 1$ be a sequence of $\mathcal C^\infty$ functions with compact support such that for any given compact subset of $\Omega$ only finitely many are not identically $1$. Let $\psi$ be a $\mathcal C^\infty$
function with $\sum_{k=1}^n|\frac{\partial \eta_\nu}{\partial \overline{z}_k}|^2\leq e^\psi$
in $\Omega$ for all $\nu=1,2,\dots.$
Let $\phi\in \mathcal C^\infty(\Omega)$ and set $\phi_1=\phi-2\psi, \phi_2=\phi-\psi,\phi_3=\phi.$
Let $T$ denote the $\overline{\partial}$ operator from $L^2_{(p,q)}(\Omega,\phi_1)$
to $L^2_{(p,q+1)}(\Omega,\phi_2)$ and let  $S$ denote the $\overline{\partial}$ operator from $L^2_{(p,q+1)}(\Omega,\phi_2)$
to $L^2_{(p,q+2)}(\Omega,\phi_3).$ Here $0\leq p\leq n, 0\leq q\leq n-1.$
Let $f\in D_{(p,q+1)}.$ Then
$$
\int \left(\sum'_{|I|=p,|K|=q}\sum_{j,k=1}^n f_{I,jK}\overline{f_{I,kK}}\frac{\partial \phi^2}{\partial z_j\partial \overline{z}_k}-2'' |\partial \psi|^2|f|^2\right)e^{-\phi} \leq 2'\|T^*f\|_1^2+\|Sf\|_3^2.
$$

\begin{proof}
Integrating both sides of Lemma 4.2.$1_k$, we get
\bea=
\sum'_{I,K}\sum_{j,k=1}^n \int(\delta_j f_{I,jK}\overline{\delta_k f_{I,kK}}
& - & \frac{\partial f_{I,j K}}{\partial \overline{z}_k }\overline{\frac{\partial f_{I,kK}}{\partial \overline{z}_j}})e^{-\phi}\\
& \leq  & 2'\int |T^*f|^2e^{-\phi_1}+\int|Sf|^2e^{-\phi_3}\\
& + & 2''\int |f|^2|\partial \psi|^2 e^{-\phi}\\
\eea

We next apply Corollary 4.2.$1_f$ to each integral on the left side, setting $f=f_{I,jK}, g=f_{I,kK}$
in the corollary.
Then we get 
\bea
\int \sum'_{|I|=p,|K|=q}\sum_{j,k=1}^n f_{I,jK}\overline{f_{I,kK}}\frac{\partial \phi^2}{\partial z_j\partial \overline{z}_k}e^{-\phi} &  \leq & 2'\|T^*f\|_1^2+\|Sf\|_3^2\\
& +& 2''\int |\partial \psi|^2|f|^2e^{-\phi}\\
\eea

\end{proof}

\section{Hormander 4.2 $\overline{\partial}$ in $L^2_{loc}$, Levi problem}

{\bf Corollary 4.2.$1_m$.}
Assume the conditions of Theorem 4.2.$1_\ell$. Suppose in addition the condition that
$$\sum_{j,k=1}^n \frac{\partial^2 \phi}{\partial z_j\partial \overline{z}_k}w_j\overline{w}_k\geq
2(|{\partial \psi}|^2+e^\psi)\sum_1^n|w|^2,w\in \mathbb C^n.$$
Then we have that for every $f\in D_{T^*}\cap D_S$ that
$$
\|f\|^2_{\phi_2}\leq \|T^*f\|^2_{\phi_1}+\|Sf\|^2_{\phi_3}.
$$

We let $c:\Omega\rightarrow \mathbb R$ denote the largest function such that $$c(z)|w|^2\leq \sum_{j,k=1}^n \frac{\partial^2 \phi}{\partial z_j\partial \overline{z}_k}w_j\overline{w}_k$$
for all $z\in \Omega$ and all $w\in \mathbb C^n$. Then, in particular, $c\geq 2(|{\partial \psi}|^2+e^\psi)$ and $c(z)$ is continuous.

\begin{proof}
\bea
\sum'_{|I|=p,|K|=q}\sum_{j,k=1}^n f_{I,jK}\overline{f_{I,kK}}\frac{\partial \phi^2}{\partial z_j\partial \overline{z}_k} & \geq &
\sum'_{I,K}c\sum_{j =1}^n|f_{I,jK}|^2\\
& = &
c\sum'_{I,K}\sum_{j \notin K}|f_{I,jK}|^2\\
& = &
c\sum'_{I,|J|=q+1}\sum|f_{I,J}|^2\\
& = & c|f|^2\\
\eea
  
Hence we get for any $f\in D_{(p,q+1)}$ that
\bea
\|f\|^2_{\phi_2} & = & \int |f|^2 e^{-\phi_2}\\
& = & \int e^{\psi} |f|^2e^{-\phi}\\
& \leq & \int (\frac{c}{2}-|{\partial} \psi|^2)|f|^2e^{-\phi}\\
& \leq &  \frac{1}{2}\left(\sum'_{|I|=p,|K|=q}\sum_{j,k=1}^n f_{I,jK}\overline{f_{I,kK}}\frac{\partial \phi^2}{\partial z_j\partial \overline{z}_k}e^{-\phi}-2\int (|{\partial} \psi|^2)|f|^2e^{-\phi}\right)\\
& \leq & \frac{1}{2}(2\|T^*f\|^2+\|Sf\|^2)\\
&  \leq & \|T^*f\|^2+\|Sf\|^2\\
\eea
The corollary follows now for all $f\in D_{T^*}\cap D_S$ by density of $D_{(p,q+1)}$ in the graph norm.
\end{proof}

{\bf Theorem 4.2.$1_n$.}
Assume the conditions in Theorem 4.2.$1\ell$ and Corollary 4.2.$1_m$.
Then if $f\in L^2_{(p,q+1)}(\Omega,\phi_2)$ and $\overline{\partial} f=0$, then there exists
a $g\in L^2_{(p,q)}(\Omega,\phi_1)$ such that $\overline{\partial} g=f$  and
$$
\|g\|_{\phi_1} \leq \|f\|_{\phi_2}.
$$

\begin{proof}
By Theorem 4.1.$3_j$, we have, since $\|f\|^2\leq \|T^*f\|^2+\|Sf\|^2$ for all $f\in D_{T^*}\cap D_S$ that for any $y\in D_{T^*}$ and any $u\in D_S$ we get:
$$
|y(u)| \leq \|T^*y\|\|u\|+\|Su\|\|y\|.
$$
Now $N_S\subset D_S$ so if $y\in D_{T^*}$ and $u\in N_S$ then
$$
|y(u)| \leq \|T^*y\|\|u\|.
$$
Hence we have the norm of $y$ as a linear functional on $N_S$ that
$\|y_{|N_S}\|\leq \|T^*y\|$ for all $y\in D_{T^*}.$ We set $F=N_S.$ This is a closed subspace containing $R_T.$
It follows then from Theorem 4.1.$1_b$ and Theorem 4.1.1'  with constant $C=1$ that $R_T=N_S$ and that for every $u\in N_S$ there is a $v\in D_T$
with $Tv=u$ and $\|v\|\leq \|u\|.$
\end{proof}

{\bf Lemma 4.2.$1_o$.}
Let $\{a_j,b_j\}_ {j=2,3,\dots}$ be given strictly positive constants.  Then there is a smooth, positive increasing, convex function f$\lambda(x)$ or $x\geq 0$ such that  $\lambda'(x)>a_j$ on $[j,j+1], j\geq 2$ and
$\lambda(x)>b_j$ on $[j,j+1], j\geq 2.$

\begin{proof}
To find such a $\lambda,$ pick a sufficiently large smooth function $\sigma(x)>>1, x\geq 0$ and define
$\nu(x)=\int_0^x \sigma(t)dt$, $\lambda(x)=\int_0^x \nu(t)dt.$
\end{proof}

{\bf Lemma 4.2.$1_p.$}
Suppose that $\Omega\subset \mathbb C^n$ is pseudoconvex and assume that $\psi$ is as in Theorem 4.2.$1_\ell$. Also suppose that $f\in L^2_{(p,q+1),loc}(\Omega)$. Then there exists a smooth strongly plurisubharmonic function $\phi$ on $\Omega$ so that
$$
\sum_{i,k=1}^n\frac{\partial^2 \phi}{\partial z_i\partial \overline{z}_k}w_i\overline{w}_k
\geq 2(|\partial \psi|^2+e^\psi) \sum_{\ell=1}^n|w_\ell|^2
$$
and $f\in L^2_{(p,q+1)}(\Omega, \phi_2), \phi_2=\phi-\psi.$

\begin{proof}
Using Theorem 2.6.11, we can find a smooth strongly plurisubharmonic function $\rho$ on $\Omega$ such that
$\{\rho<c\}\subset \subset \Omega$ for any $c\in \mathbb R.$ We can assume that $\min \rho=2$
by adding a constant. Choose a smooth function $m(z)>0$ on $\Omega$ such that
$\sum_{i,k=1}^n\frac{\partial^2 \rho}{\partial z_i\partial \overline{z}_k}w_i\overline{w}_k\geq
m \|w\|^2.$

For $j=2,3,\dots$, let $L_j=\{z\in \Omega; j\leq \rho(z)\leq j+1\}.$ Then each $L_j$ is compact and
$\Omega=\cup L_j.$
Define
$$a_j=\sup_{L_j}\frac{2(|\partial \psi|^2+e^\psi)}{m(z)}, j\geq 2.$$ Since $L_j$ is compact,
$f\in L^2_{(p,q+1)}(L_j).$ Pick $b_j>0$ so that
$\int_{L_j}|f|^2e^{\psi-b_j}<\frac{1}{2^j}, j\geq 2.$ Let $\lambda$ be as in the  Lemma 4.2.$1_o$. 
We define $\phi=\lambda\circ \rho$ on $\Omega.$
Then on $L_j$,

\bea
\sum_{i,k=1}^n\frac{\partial^2 \phi}{\partial z_i\partial \overline{z}_k}w_i\overline{w}_k & \geq &
\lambda'(\rho(z))\sum_{i,k=1}^n\frac{\partial^2 \rho}{\partial z_i\partial \overline{z}_k}w_i\overline{w}_k\\
& \geq & a_j\sum_{i,k=1}^n\frac{\partial^2 \rho}{\partial z_i\partial \overline{z}_k}w_i\overline{w}_k\\
& 0 & \sup_{L_j}\frac{2(|\partial \psi|^2+e^\psi)}{m(z)}\sum_{i,k=1}^n\frac{\partial^2 \rho}{\partial z_i\partial \overline{z}_k}w_i\overline{w}_k\\
& \geq & \frac{2(|\partial \psi|^2+e^\psi)}{m(z)}m(z)|w|^2\\
& = & 2(|\partial \psi|^2+e^\psi)
|w|^2\\
\eea

Also 
\bea
\int_{\Omega}|f|^2 e^{\psi-\phi} & = & \sum_{j=2}^\infty \int_{L_j}|f|^2 e^{\psi-\lambda(\rho)}\\
 & \leq & \sum_{j=2}^\infty \int_{L_j}|f|^2 e^{\psi-b_j}\\
& < & \sum_{j=2}^\infty \frac{1}{2^j}\\
& < & \infty\\
\eea

\end{proof}

{\bf Corollary 4.2.$1_q$.}
If $f\in L^2_{(p,q+1),loc}(\Omega),$ $\Omega$ pseudoconvex in $\mathbb C^n$, and $\overline{\partial}f=0$, then there exists a $g\in L^2_{(p,q),loc}(\Omega)$ so that 
$\overline{\partial}g=f$ in $\Omega.$

\begin{proof}
let $f\in L^2_{(p,q+1),loc}(\Omega)$ and suppose $\overline{\partial} f=0.$ 
We pick $\phi$ as in Lemma 4.2.$1_p$. Then $f\in L^2_{(p,q+1)}(\Omega,\phi_2)$ and $\overline\partial f=0.$ Using Theorem 4.2.$1_n$, we find $g\in L^2_{(p,q)}(\Omega,\phi_1)$ with $\overline\partial g=f.$ This
$g\in L^2_{(p,q),loc}(\Omega).$
\end{proof}

{\bf Lemma 4.2.$1_r$:}
Suppose that $f\in L^2(\Omega)$ and that $\overline{\partial}f=0.$ Then there is a holomorphic function $u$ on $\Omega$ so that $u=f$ a.e.

\begin{proof}
Suppose that $f\in L^2(B(0,\delta)),$ the ball of radius $\delta$ in $\mathbb C^n$ centered at $0.$
Assume that $\overline{\partial} f=0$ in the sense of distributions. We apply the smoothing theorem.
Then if $0<\epsilon<\frac{\delta}{2},$ $f*\chi_\epsilon$ is $\mathcal C^\infty$ in $B(0,\delta/2)$
and $\|f*\chi_\epsilon-f\|_{L^2(B(0,\delta/2)}\rightarrow 0$ when $\epsilon \rightarrow 0.$
By Lemma 4.1.$4_b$, $\overline{\partial}(f*\chi_\epsilon)= (\overline\partial f)*\chi_\epsilon=0*\chi_{\epsilon}=0.$ Hence each $f*\chi_\epsilon$ is holomorphic. We can choose
$\epsilon_j\searrow 0$ so that 
$\|f*\chi_{\epsilon_{j+1}}-f*\chi_{\epsilon_j}\|_{L^2(B(0,\delta/2)}<\frac{1}{2^j}.$
Let $u_j:=f*\chi_{\epsilon_j}.$ We get
$$
\|u_{j+1}-u_j\|_{L^1(B(0,\delta/2)}\leq C\|u_{j+1}-u_j\|_{L^2(B(0,\delta/2)}\leq \frac{C}{2^j}.
$$
By Theorem 2.2.3 we then get pointwise estimates
$|u_{j+1}(z)-u_j(z)|\leq \frac{C'}{2^j}$ on $B(0,\delta/4).$ Hence $u_j$ converges uniformly 
to a function $u$ on $B(0,\delta/4).$ By Corollary 2.2.5, $u$ is holomorphic. But necessarily $u=f$ a.e.

\end{proof}

We prove an extension Lemma. We use the notation $z=(z_1,\dots,z_n)=(z',z_n)$
for points in $\mathbb C^n.$

{\bf Theorem 4.2.$8_a$.}
Let $\Omega\subset \mathbb C^n$ be pseudoconvex, $n>1.$ Let $\Omega':=\{z'\in \mathbb C^{n-1}; (z',0)\in \Omega\}.$ If $f$ is a holomorphic function in $\Omega',$ then there exists a holomorphic function $F$ in $\Omega$ so that $F(z',0)=f(z')$ for all $z'\in \Omega.$

\begin{proof}
Let $z'\in \Omega.$ Then there exists an $\epsilon_{z'}>0$ so that $B((z',0),\epsilon_{z'})\subset \Omega.$ Set $\omega:=\cup_{z'\in \Omega'}B((z',0),\epsilon_{z'}).$
Then $\omega$ is an open set in $\mathbb C^n$, $\Omega'*(0)\subset \omega\subset \Omega$
and for every $(z',z_n)\in \omega,$ $z'\in \Omega'.$
Define $G$ on $\omega$ by $G(z',z_n)=f(z').$ Then $G$ is holomorphic. The sets $\Omega'*(0)$
and $\partial \omega\cap \Omega$ are disjoint relatively closed sets in $\Omega.$
Hence there exists a function $\chi\in \mathcal C^\infty(\Omega)$ such that
$\chi=1$ in an open set containing $\Omega'*(0)$ and $\chi=0$ in an open set containing
$\partial \omega\cap \Omega.$ We define $H$ on $\Omega$ by letting
$H(z)=\chi G$ on $\omega$ and $H(z)=0$ on $\Omega\setminus \omega.$
Then $H\in \mathcal C^\infty(\Omega).$ Let $\sigma=\overline{\partial}H.$ Then $\sigma\equiv 0$
in an open neighborhood of $\Omega'*(0).$ Define $\tau$ on $\Omega$, by setting
$\tau=0$ on $\Omega'*(0)$ and $\tau=\sigma/z_n$ on the complement.
Then $\tau=0$ in a neighborhood of $\Omega'*(0)$. Hence $\overline\partial \tau=0$
in a neighborhood of $\Omega'*(0).$ Also in the complement of $\Omega'*(0),$ we have
$\overline\partial \tau=\frac{\overline\partial \sigma}{z_n}=0$. So $\overline\partial{\tau}=0$ on $\Omega.$
Since $\tau$ is smooth, $\tau\in L^2_{(0,1),loc}(\Omega).$ So by Corollary 4.2.$1_q$, there exists
$h\in L^2_{loc}(\Omega)$ so that $\overline{\partial} h=\tau.$
By Lemma 4.2.$1_r$, we can let $h$ be a $\mathcal C^\infty$ holomorphic function in a neighborhood
of $\Omega'*(0).$ 

Let $F=H-z_nh.$ Then $F$ is holomorphic in a neighborhood of $\Omega'*(0)$ and $F(z',0)=f(z')$
if $z'\in \Omega.$ Moreover $F\in L^2_{loc}(\Omega)$ and $\overline{\partial}F=
\overline\partial H-z_n\overline\partial h=\sigma-z_n \tau=0.$
Hence by Lemma 4.2.$1_r$ again, there is a holomorphic function $\hat{F}$ on $\Omega$
so that $\hat{F}=F$ a.e. But then, $\hat{F}\equiv F$ on a neighborhood of $\Omega'*(0).$
So $f(z')=\hat{F}(z',0).$
\end{proof}

We can now solve the Levi problem:

{\bf Theorem 4.2.$8_b$.}
A pseudoconvex domain in $\mathbb C^n$ is a domain of holomorphy.

\begin{proof}
We prove the theorem by induction in the dimension. The theorem is true in dimension 1 because all domains are domains of holomorphy. Next, assume that the theorem is true for domains in $\mathbb C^{n-1},n \geq 2.$ We prove the theorem in $\mathbb C^n$ by contradiction. So assume that there is a domain $\Omega\in \mathbb C^n$ which is pseudoconvex, but $\Omega$ is not a domain of holomorphy.

Then, by definition 2.5.1 there are two open sets $\Omega_1$ and $\Omega_2$ in $\mathbb C^n$
with the following properties:\\
(a) $\emptyset \neq \Omega_1\subset \Omega_2\cap \Omega.$\\
(b) $\Omega_2$ is connected and $\Omega_2\setminus \Omega \neq \emptyset.$\\
(c) For every $u\in A(\Omega)$ there is a function $u_2\in A(\Omega_2)$ such that $u=u_2$ in $\Omega_1.$\\

Pick a point $p\in \Omega_1$ and a point $q\in \Omega_2\setminus \Omega.$
Since $\Omega_2$ is connected, there is a curve $\gamma(t), 0\leq t\leq 1,$ $\gamma(0)=p$ and
$\gamma(1)=q.$ Let $t_0$ be the smallest $t$ so that $\gamma(t)\in \Omega_2\setminus \Omega.$
Then $\gamma([0,t_0>)\subset \Omega_2\cap \Omega.$ We replace $q$ by $\gamma(t_0)$. Then we can arrange that $\gamma(0)=p\in \Omega_1$, $\gamma([0,1>)\subset \Omega_2\cap \Omega$
and $\gamma(1)=q\in \Omega_2\setminus \Omega.$

Pick $\epsilon>0$ so that $B(q,\epsilon)\subset \Omega_2.$
Pick $0<t_1<1$ so that $\gamma(t_1)\in B(q,\epsilon/4).$ Then there exists a $0<\rho\leq \epsilon/4$ so that $B(\gamma(t_1),\rho)\subset \Omega_2\cap \Omega$ and there exists a $q'\in \partial B(\gamma(t_1),\rho)\setminus \Omega.$ 
We let $\Omega_1'=B(\gamma(t_1),\rho), \Omega_2'=B(\gamma(t_1),\epsilon/2)\subset \Omega_2.$
Then \\
(a') $\emptyset \neq \Omega_1'\subset \Omega_2'\cap\Omega.$\\
(b') $\Omega_2'$ is connected and $\Omega_2'\setminus \Omega \neq \emptyset.$\\
We show also \\ 
(c') For every  $u'\in A(\Omega),$ there is a function $u_2'\in A(\Omega_2')$ such 
that $u'=u_2'$ in $\Omega_1'.$\\

To show (c'), let $u'\in A(\Omega).$ Then by (c) there is a function $u_2\in A(\Omega_2)$ so that
$u'=u_2$ in $ \Omega_1.$ Let $$V=\{z\in \Omega_2\cap \Omega\; \mbox{ such that $u'=u_2$ in a neighborhood of $z.$}\}$$ Then $\Omega_1\subset V$. Also $\gamma([0,1>)\subset \Omega_2\cap \Omega.$ 
Hence by the identity theorem we must have that $\gamma([0,1>)\subset V.$ So in particular,
there is a small ball centered at $\gamma(t_1)$ on which $u'=u_2.$ 
Since $\Omega_1'$ is a ball centered at $\gamma(t_1)$ and $\Omega_1'\subset \Omega_2\cap \Omega,$
it follows that $u'=u_2$ in $\Omega_1'.$
Finally, we let $u_2'$ be the restriction of $u_2$ to $\Omega_2'.$ 
Since $\Omega_1'\subset \Omega_2'$, it still follows that $u'=u_2'$ in $\Omega_1'.$
This proves $(c').$
The proof so far proves actually a small useful result in order to characterize domains of holomorphy.
We write this as a lemma:

{\bf Lemma 4.2.$8_c$.}
A domain $\Omega\subset \mathbb C^n$ is a domain of holomorphy if and only if
there do not exist two concentric open balls $B_1\subset B_2$ such the following three
properties hold:\\
(a'') $\emptyset \neq B_1\subset \Omega\cap B_2.$\\
(b'') $B_2\setminus \Omega\neq \emptyset.$\\
(c'') For every $u\in A(\Omega)$ there exists a holomorphic function $u_2$ on $B_2$ such
that $u=u_2$ on $B_1.$

We continue with the proof of the Theorem.
We summarize: Starting with the hypothesis that the pseudoconvex domain $\Omega$ is not a domain
of holomorphy, we have found two concentric balls satisfying the properties (a''),(b'') and (c'').
There is no loss of generality to assume that $B_1$ is the unit ball centered at the origin in $\mathbb C^n$ and that $B_2$ is the larger concentric ball $B(0,r)$ for some $r>1.$
Also we can assume that the point $(1,0,0,\dots,0)$ is in the boundary of $\Omega.$
Next we let $\Omega'$, $B_1',$ $B_2'$ be the corresponding open sets of $z'\in \mathbb C^{n-1}$
for which $(z',0)$ is in the domains. 
Note that the sets $B_1'$, $B_2'$ satisfy condition  (a''), (b'') in the characterization of domains of holomorphy for the domain $\Omega'.$ We will show that (c'') is also satisfied.
Let $v(z')\in A(\Omega').$ Since $\Omega$ is pseudoconvex, we can apply the extension lemma to find
a holomorphic function $V\in A(\Omega)$ such that $v(z')=V(z',0)$ for all $z'\in \Omega'.$
Hence there exists a holomorphic function $V_2$ on $B_2$ so that $V_2=V$ on $B_1.$
Let $v_2 $ be the holomorphic function on $B_2'$ given by $v_2(z')=V_2(z',0)$.
Then if $z'\in B_1'$, we have that $(z',0)\in B_1.$ Hence $v_2(z')=v(z').$ This proves (c'')
for the domains $B_1',B_2', \Omega'.$ 
By the inductive hypothesis, the domain $\Omega'$ is not a domain of holomorphy and therefore
also cannot be pseudoconvex. However, it follows from Theorem 2.6.7 that
$\Omega'$ is pseudoconvex because $\Omega$ is pseudoconvex. We have reached a contradiction.

\end{proof}

\section{Hormander, Acta paper}

We will now start with the preparations for proving the Ohsawa-Takegoshi extension theorem.
This is a more precise version of the extension theorem, Theorem 4.2.$8_a$ which was the key ingredient in the solution of the Levi problem. The proof is based on a version of Hormander's theorem which was already in his paper: $L^2$ estimates and existence theorems for the $\overline{\partial}$ operator, Acta Math, 113, 89--152 (1965). This paper is also the basis for his book.

We first investigate  Hermitian matrices.
Let $A=\{a_{ij}\}_{i,j=1}^n$ be an n by n matrix of complex numbers where $a_{ij}$ denotes
the term on row i and column j.
We say that $A$ is Hermitian if $a_{ij}=\overline{a}_{ji}.$ This is equivalent to the
statement that $A=\overline{A}^T$ where $T$ is the transpose.

\begin{lemma}
An n by n matrix $A$ is Hermitian if and only if there is an orthonormal basis $b_j$ of eigenvectors
and the corresponding eigenvalues $\lambda_j$ are all real numbers.
In this case, $A$ is selfadjoint in the sense that for any pair of vectors $x,y$
we have that $<x,Ay>=<Ax,y>.$ We then write $A=A^*$ and $<x,Ay>=<A^*x,y>.$
\end{lemma}

\begin{proof}
We think of $x$ as column vectors or $n$ by $1$ matrices. So $<x,y>=
\sum_i x_i\overline{y}_i=\overline{y}^Tx$ as matrix product.
If $A$ is Hermitian, then
\bea
<Ax,y> & = & \overline{y}^TAx\\
& = & \overline{y}^T \overline{A}^Tx\\
& = & \overline{Ay}^Tx\\
& = & <x,Ay>\\
\eea

Suppose that $A$ is Hermitian and $b_i,b_j$ are eigenvectors with eigenvalues $\lambda_i,\lambda_j.$
Then 
\bea
\lambda_i<b_i,b_j> & = & <Ab_i,b_j>\\
& = & <b_i,Ab_j>\\
& = & <b_i,\lambda_jb_j>\\
& = & \overline{\lambda}_j<b_i,b_j>\\
\eea

If we apply this first to the case $i=j$ we see that all $\lambda_j$ are real numbers.
Next if we apply this to the case when $b_i$ and $b_j$ belong to different eigenspaces,
so $\lambda_i\neq \lambda_j,$ we see that $b_i$ and $b_j$ are perpendicular vectors.
Hence the eigenspaces, $E_j$  are perpendicular. 
Suppose next that $b$ is perpendicular to all the eigenspaces. Then for any eigenvector,
$<Ab,b_j>=<b,Ab_j>=\lambda_j<b,b_j>=0$. Hence $A$ maps the orthogonal complement $B$ to all the eigenspaces to itself. But then $A=0$ since otherwise $B$ would contain a nonzero eigenvector. If we finally replace the $b_j's$ in the same eigenspace with
an orthonormal basis for that eigenspace, we obtain an orthonormal basis of eigenvectors.\\

Suppose next that $\mathbb C^n$ has an orthonormal basis $b_j$ of eigenvectors of $A$ with real
eigenvalues $\lambda_j.$ We then get:

\bea
<b_i,Ab_j> & = & <b_i,\lambda_j b_j>\\
& = & \lambda_j<b_i,b_j>\\
 & \mbox{and} &\\
<b_i,\overline{A}^Tb_j>
& = & \overline{\overline{A}^Tb_j}^Tb_i\\
& = & (A^T\overline{b}_j)^Tb_i\\
& = & \overline{b}_j^T Ab_i\\
& = & \lambda_i<b_i,b_j>\\
\eea

It follows that $$
<b_i,Ab_j>=<b_i,\overline{A}^Tb_j>$$ for all $b_i,b_j.$ It follows that
$Ab_j=\overline{A}^Tb_j$ for all $b_j$ but then we must have that $A=\overline{A}^T.$\\

The last part follows from the observation that $<x,Ay>=<\overline{A}^Tx,y>$ for any n by n matrix.
In other words $A^*=\overline{A}^T$ is valid for all n by n matrices.

\end{proof}

Next suppose that $A$ is Hermitian and $A$ is positive semidefinite, in the sense that all $\lambda_j\geq 0.$ Then we define the matrix $\sqrt{A}$ to be the matrix $B$ with the same eigenspaces as $A$ and with corresponding eigenvalues $\sqrt{\lambda_j}.$ Obviously, $BB=A.$

The previous lemma then immediately gives:

\begin{corollary}
If $A$ is a positive semidefinite Hermitian matrix, then the matrix $\sqrt{A}$ is also
a positive semidefinite Hermitian matrix.
\end{corollary}

\begin{lemma}
Let $\Omega$ be an open set in $\mathbb R^N$ and let $A(x)$ be a continuously varying positive
semidefinite Hermitian n by n matrix. Then the map $x\rightarrow \sqrt{A}(x)$ is also continuous.
\end{lemma}

\begin{proof}
Suppose that $x_n\rightarrow x$ and $A(x_n)$ has an orthonormal basis $\{b^n_j\}_{j=1}^n$
with eigenvalues $\lambda_j^n.$ By taking a subsequence we can assume that
$b^n_j\rightarrow b_j$ and $\lambda^n_j\rightarrow \lambda_j.$ Then $b_j$ must be an orthonormal basis of eigenvectors for $A$ with eigenvalues $\lambda_j$. But then also
$\sqrt{\lambda^n_j}\rightarrow \sqrt{\lambda_j}$ so $\sqrt{A}(x)$ is continuous.

\end{proof}

Let $H$ denote the Hilbert space $L^2_{(0,1)}(\Omega,\gamma)$ for a continuous real function $\gamma(x)$ on an open set $\Omega\subset \mathbb C^n.$ Let $A(x)$ be a continuously varying family of positive semidefinite Hermitian matrices.
We identify a $(0,1)$ form $f=\sum_jf_jd\overline{z}_j$ with the column vector $f=(
f_1,f_2,\dots,f_n)^T.$

Let $f\in H.$ Then $Af(x)$ is in $L^2_{(0,1),loc}.$ If $Af\in H$, then we say that $f\in D_A$. 

\begin{lemma}
The operator $f\rightarrow Af$ is densely defined and has closed graph $\Gamma=
\{(f,Af);f\in D_A\}.$
\end{lemma}

\begin{proof}
All $f\in H$ with continuous coefficients with compact support are in $D_A$. Hence, $A$ is densely defined. Suppose that $(f,g)$ is in the closure of the graph. Then there exist
$(f_n,g_n)\in \Gamma$ so that $f_n\rightarrow f$ and $g_n\rightarrow g.$
Then, on any compact subset, $f_n\rightarrow f$ and $Af_n=g_n\rightarrow g$ in $L^2.$
Hence $Af=g$ on any compact subset. Hence $Af=g$ on $\Omega.$ Therefore $Af\in H$,
so $f\in D_A$ and $(f,Af)\in \Gamma.$
\end{proof}

We recall the situation in Theorem 4.2.$1_\ell.$:
We write down the conclusion in the case of $(0,1)$ forms $f=\sum_j f_j d\overline{z}_j.$

$$
\int (\sum_{j,k=1}^n f_j\overline{f}_k \frac{\partial ^2 \phi}{\partial z_j\overline{z}_k}-
2''|\overline{\partial}\psi|^2|f|^2)e^{-\phi}
\leq 2'\|T^*f\|_1^2+\|Sf\|^2_3
$$ valid for all $f\in D_{(0,1)}.$

We let $A_{\phi,\psi}(x)=A(x)$ be the matrix valued function on $\Omega$ given by
$A(x)=\{\frac{\partial ^2 \phi}{\partial z_j\overline{z}_k}\}-2''|\overline\partial\psi|^2 I.$\\

(*)\;\;\;We make the assumption that $A(x)=\{a_{jk}\}$ is positive semidefinite.\\

\bea
\sum_{jk}a_{jk}f_j\overline{f}_k & = & \sum_j f_j\sum_ka_{jk}\overline{f}_k\\
& = & \sum_j f_j(A\overline{f})_j\\
& = & <f,\overline{A}f>\\
\eea

Now, $\overline{A}=\overline{\overline{A}^T}=\overline{\overline{A}}^T$ so $\overline{A}$ is also Hermitian. In fact if $b_j$ are the eigenvectors of $A$, then $\overline{b}_j$ are eigenvectors of $\overline{A}$ and the eigenvalues are the same. Hence $\overline{A}$ is positive semidefinite.
We can write $\overline{A}=B^2$, where $B={\sqrt{\overline{A}}}$ and $B$ is positive semidefinite.
Set $C=Be^{-\psi/2}.$ Then $C(x)$ is still positive semidefinite.
We then get that 

\bea
\sum_{jk}a_{j,k}f_j\overline{f}_ke^{-\phi} & = & <f,\overline{A}f>e^{-\phi}\\
& = & <f,B^2f>e^{-\phi}\\
& = & <Bf,Bf>e^{-\phi_2-\psi}\\
& = & |Cf|^2e^{-\phi_2}\\
\eea

Hence we have for the operator $C$ that
$$
\int |Cf|^2e^{-\phi_2}
\leq 2'\|T^*f\|_1^2+\|Sf\|^2_3
$$ valid for all $f\in D_{(0,1)}.$

We next prove:

\begin{lemma}
Suppose that $f\in D_{T^*}\cap D_S$. Then $f\in D_C$ for the $L^2$ space with weight $\phi_2$ and
$$
\|Cf\|^2_{\phi_2} \leq 2'\|T^*f\|^2_{\phi_1}+\|Sf\|^2_{\phi_3}.
$$
\end{lemma}

\begin{proof}
Let $f\in D_{T^*}\cap D_S.$ Then there exist $\{f_n\}\subset D_{(0,1)}(\Omega)$ such that $f_n$ converges to $f, T^*f,S_f$ in the graph  norm. Hence $\{f_n\}$ is a Cauchy sequence and by the estimate, we see that also $\{Cf_n\}$ is a Cauchy sequence. Hence $Cf_n\rightarrow g$ for some $g$ in $L^2_{\phi_2}.$ On compact subsets of $\Omega$ we must have that $g=Cf$. Hence $Cf\in L^2_{\phi_2}$ so $f\in D_C$ and we have the estimate
$$
\|Cf\|^2_{\phi_2} \leq 2'\|T^*f\|^2_{\phi_1}+\|Sf\|^2_{\phi_3}.
$$
\end{proof}

Next we come to Theorem 1.1.4 in Hormanders 1965 Acta paper.

\begin{theorem}
Assume the conditions on $\phi,\psi$ in Theorem 4.2.$1_\ell.$ Also assume that $A_{\phi,\psi}$ is positive semidefinite.
Suppose that $g\in R_C$, $g=Ch.$ Assume also that $g\in N_S.$ Then there exists
a $u\in D_T$ so that $Tu=g$ and $\|u\|_1\leq \sqrt{2'}\|h\|_2.$
\end{theorem}

\begin{proof}
Let $g\in N_S$ be as in the theorem. Define the functional $\sigma$ by $\sigma(T^*f)=<f,g>_2$ for any $f\in D_{T^*}.$ We want to show that
$$
(**)\;\; |<f,g>_2| \leq \sqrt{2'}\|h\|_2 \|T^*f\|_1.
$$
In this case $\sigma$ extends by Hahn-Banach to a linear functional $\tau$ defined on $L^2_{(0,0)}(\Omega,\phi_1)$ such that $\tau(T^*f)= \sigma(T^*(f))$ for all $f\in D_{T^*}$ and
$\|\tau(z)\|\leq |\sqrt{2'}\|h\|_2 \|z\|_1$ for all $z.$ Since $\tau(T^*{f})= <f,g>_2$ for all $f\in D_{T^*},$ it follows that $\tau$ in $D_{T^{**}}$ and $T^{**}(\tau)= g$. By reflexivity, there is then a $u$ such that
$Tu=g.$ Moreover, $\|u\|_1=\|\tau\|_1\leq \sqrt{2'}\|h\|_2.$ 

To prove (**), observe first that if $f\perp N_S$ then $f\perp R_T$ so $f\in N_{T^*}\subset D_{T^*}$ and $T^*f=0.$
But then $<f,g>_2=0$ so (**) holds. It suffices therefore to consider $f\in D_{T^*}\cap N_S\subset D_{T^*}\cap D_S.$ By Lemma 13.5, we then also have that $f\in D_C.$ By Lemma 13.5 we see that
$$
\|Cf\|^2_2 \leq  2'\|T^*f\|^2_1
$$
since $Sf=0.$
Hence
\bea
|<g,f>_2| & = & |<Ch,f>_2|\\
& = & |<h,Cf>_2|\\
& \leq & \|h\|_2 \|Cf\|_2\\
& \leq & \sqrt{2'}\|h\|_2\|T^*f\|_1\\
\eea

\end{proof}

Suppose that $H$ is a positive definite Hermitian n by n matrix, i.e. all eigenvalues are strictly positive.
Let $f=(f_1,\dots,f_n)$ be a vector. We define $|f|^2_H= \sum_{jk}a^{jk}f_j\overline{f}_k$
where $\{a^{jk}\}$ is the inverse matrix of $H.$

\begin{theorem}
Assume the conditions on $\phi,\psi$ in Theorem 4.2.$1_\ell$ and $A_{\phi,\psi}$ is positive definite.  Assume that $g\in  L^2_{(0,1)}(\Omega,\phi_2), \overline{\partial}g=0$ and that
$\int e^\psi |g|^2_{A_{\phi,\psi}}e^{-\phi_2}<\infty.$ Then there exists a $u\in D_T$ so that
$\overline{\partial}u=g$ and
$$
\|u\|^2_{\phi_1} \leq  2'\int e^\psi |g|^2_{A_{\phi,\psi}}e^{-\phi_2}.
$$
\end{theorem}

\begin{proof}
Let $B=\sqrt{\overline{A_{\phi,\psi}}}, C=Be^{-\psi/2}.$
We show that $g\in R_C,$ $g=Ch$ for $\|h\|_{\phi_2}<\infty.$
In fact, let $h=e^{\psi/2}B^{-1}g=e^{\psi/2}(e^{\psi/2}C)^{-1}g=C^{-1}g.$ 
Then

\bea
\|h\|_2^2 & = & \int e^{\psi}<B^{-1}g,B^{-1}g>e^{-\phi_2}\\
& = & \int e^{\psi}<g,B^{-2}g>e^{-\phi_2}\\
& = & \int e^{\psi}<g,\overline{A}^{-1}g>e^{-\phi_2}\\
& = & \int e^{\psi} <g,\overline{(A^{-1}\overline{g})}>e^{-\phi_2}\\
& = & \int e^{\psi}\sum_{jk}a^{jk}g_j\overline{g}_ke^{-\phi_2}\\
& = & \int e^{\psi}|g|^2_{A_{\phi,\psi}}e^{-\phi_2}\\
& < & \infty.\\
\eea

Since $g=Ch,$ we see that $h\in D_C$ and $g=Ch\in R_C$.
Therefore, by Theorem 13.6, there exists $u\in D_T$ so that
$Tu=g$ and  $\|u\|^2_{\phi_1}\leq 2'\|h\|^2_2= 2'\int e^{\psi}|g|^2_A e^{-\phi_2}.$\\

\end{proof}

We next eliminate the function $\psi$ and replace $2'=1+c$ by it's limiting value $1$ as $c\rightarrow 0.$
We note that this will be possible even though $2''\rightarrow \infty.$

\begin{theorem}
Let $\Omega$ be a pseudoconvex domain in $\mathbb C^n.$ Suppose that $\phi\in \mathcal C^2(\Omega)$ is strictly plurisubharmonic with Hessian matrix $A$. Then if $f=\sum_{j=1}^n f_jd\overline{z}_j$ is $\overline\partial $ closed and $\int_\Omega |f|^2_Ae^{-\phi}<\infty$ then there exists a function $u$, so that $\overline{\partial} u=f$ and
$$
\int|u|^2 e^{-\phi} \leq \int |f|^2_Ae^{-\phi}.
$$
\end{theorem}

\section{Hormander Acta cont., Chen's proof of Ohsawa-Takegoshi}

Fix a smooth strictly plurisubharmonic function $\rho$ so that $\{\rho\leq c\}$ is compact in $\Omega$
for real number  $c.$ Set $K_j=\{\rho\leq j\},j\geq 2.$ We prove first:
\begin{lemma}
There exists for each $j\geq 2$ a solution $u_j$ for the equation $\overline{\partial}u=f$ in $\Omega$
such that 
$$
(1-\frac{1}{j})\int_{K_j}|u_j|^2 e^{-\phi}\leq (1+\frac{1}{j})\int_\Omega |f|^2_Ae^{-\phi}.
$$
\end{lemma}

\begin{proof}
Choose the sequence $\eta_\nu$ so that all $\eta_\nu=1$ on $K_{j+1}.$
The function $\psi$ satisfies the inequalities $\sum_{k=1}^n |\frac{\partial \eta_\nu}{\partial \overline{z}_k}|^2\leq e^\psi.$ We will choose such a $\psi$ so that $\psi\geq 0$ in $\Omega$ and $\psi=0$ on $K_{j+1}.$ Let $\Sigma$ be a smooth convex increasing function $\Sigma(t)$ which vanishes when $t\leq j$.We will make $\Sigma$ sufficiently convex for the following to hold:\\

Let $\Lambda$ be a smooth real function with $\Lambda=1$ on $K_{j+1}$ and $\Lambda=0$ outside $K_{j+2}.$ Choose $\delta>0$ small enough so that $\delta\|z\|^2<\frac{1}{2j}$ on $K_j.$
Then if $\epsilon$ is small enough and $\tilde{\phi}=\Lambda\cdot (\phi*\chi_\epsilon)+\delta \|z\|^2
+\Sigma \circ \rho$ then\\
(1) $\tilde{\phi}\geq \phi+2\psi\geq \phi,$\\
(2) $A_{\tilde{\phi},\psi}\geq A_{\phi,0}.$\\
(3) $\tilde{\phi}\in\mathcal C^\infty(\Omega).$\\

This implies that $|f|^2_{A_{\tilde{\phi},\psi}}\leq |f|^2_{\phi,0}=|f|^2_A$

We apply Theorem 13.7 using the functions $\tilde{\phi}$ and $\psi$ and $2'=1+\frac{1}{j}.$
Note that

\bea
\int_\Omega e^\psi |f|^2_{A_{\tilde{\phi},\psi}}e^{-\tilde{\phi}_2} & \leq & \int_\Omega e^{\psi}|f|^2_A
e^{-(\tilde\phi-\psi)}\\
& \leq & \int_\Omega e^{\psi}|f|^2_A
e^{-(\tilde\phi-2\psi)}\\
& \leq & \int_\Omega |f|^2e^{-\phi}.
\eea

Hence we can and find $u_j\in L^2(\Omega,\tilde{\phi}_1)$ so that
$\overline{\partial}u_j=f$ in $\Omega$ and
$$
\int_\Omega|u_j|^2e^{-\tilde{\phi}_1}\leq (1+\frac{1}{j})\int_\Omega |f|^2_{A}e^{-\phi}.
$$

On $K_j$ we have that $\tilde{\phi}_1=\tilde{\phi}-\psi=\tilde{\phi}=\phi*\chi_\epsilon+\delta \|z\|^2.$
Hence for small enough $\epsilon$ we have that $|\tilde\phi_1-\phi|<\frac{1}{j}$ on $K_j.$ Therefore
\bea
\int_\Omega |u_j|^2e^{-\tilde{\phi}_1} & \geq & \int_{K_j}
|u_j|^2 e^{(\phi-\tilde{\phi}_1)-\phi}\\
& \geq & e^{-1/j}\int_{K_j}|u_j|^2e^{-\phi}\\
& \geq & (1-\frac{1}{j})\int_{K_j}|u_j|^2e^{-\phi}\\
\eea

\end{proof}

Next we prove the Theorem.
\begin{proof}
Since $\overline{\partial}u_j=f=\overline{\partial}u_k$ on $\Omega$ for all $j,k$,
it follows that all the functions $u_j-u_k$ are holomorphic on $\Omega.$ Also they are uniformly bounded in $L^2$ norm on compact subsets and therefore in $L^1$ norm as well. Hence by Corollary 1.2.6
there is a subsequence converging uniformly on compact sets to some $u$. Then $\overline\partial u=f$
and 
$$
\int_{\Omega}|u|^2 e^{-\phi}\leq \int_\Omega |g|^2_{A}e^{-\phi}.
$$

\end{proof}

Our next topic is the Ohsawa-Takegoshi theorem. We follow the proof by Bo-Yong Chen. \\
arXiv: 1105.2430v1 [math.CV] 12 May 2011.\\

We start with some preliminary formulas
valid in $\mathbb C.$\\

Let $0<r<e^{-1/2}.$ Also suppose the $0<s<1$ is small enough that $r^2+s^2<e^{-1}.$
Define $\rho_s(z):=\log(|z|^2+s^2)$ for $z\in \mathbb C, |z|<r.$ Then
$2\log s\leq \rho_s< -1,$ so $1<-\rho_s\leq 2\log \frac{1}{s}.$ 
Set $\eta_s:= -\rho_s +\log (-\rho_s).$ 

\begin{lemma}
$1<-\rho_s<\eta_s<-2\rho_s.$
\end{lemma}

\begin{proof}
This follows from the inequality $0<\log x<x-1<x,$ valid for $x>1.$
\end{proof}

Hence $1<\eta_s<4\log \frac{1}{s}.$ \\

Let $\psi_s:= -\log \eta_s.$
Then $-\log(4\log\frac{1}{s})<\psi_s<0.$
We have the following formulas: We skip the index $s$.

\begin{lemma}
\bea
(1)\;\;\rho_{\overline{z}} & = & \frac{z}{|z|^2+s^2}\\
(2)\;\; \rho_{z\overline{z}} & = & \frac{s^2}{(|z|^2+s^2)^2}\\
(3)\;\; \eta_{{z}} & = & -(1+(-\rho)^{-1})\rho_{{z}}\\
(4)\;\; \eta_{z\overline{z}} & = & -(1+(-\rho)^{-1})\rho_{z\overline{z}}-\frac{|\rho_z|^2}{\rho^2}\\
(5)\;\; \psi_{z\overline{z}} & = & -\frac{\eta_{z\overline{z}}}{\eta}+\frac{|\eta_z|^2}{\eta^2}\\
& = & (1+(-\rho)^{-1})\frac{\rho_{z\overline{z}}}{\eta}+\frac{|\rho_z|^2}{\eta \rho^2}+
\frac{|\eta_z|^2}{\eta^2}\\
(6)\;\; \psi_{z\overline{z}} & \geq & \left(\frac{1}{\eta^2}+\frac{1}{\eta (-\rho+1)^2}\right)|\eta_z|^2\\
(7)\;\; \psi_{z\overline{z}} & \geq & \frac{s^2}{\eta (|z|^2+s^2)^2}\\
(8)\;\; |\psi_z|^2 & = & \frac{(1+(-\rho)^{-1})^2}{\eta^2}|\rho_z|^2\\
(9)\;\; \psi_{z\overline{z}} & \geq & \frac{|\rho_z|^2}{\eta}\;\mbox{if} \; \; |z|^2\leq s^2\\
\eea
\end{lemma}

\begin{proof}
(1) clear\\
(2) clear\\
(3) clear\\
(4) clear\\
(5) clear\\
(6):
We use (5):
\bea
\psi_{z\overline{z}} & = & (1+(-\rho)^{-1})\frac{\rho_{z\overline{z}}}{\eta}+\frac{|\rho_z|^2}{\eta \rho^2}+
\frac{|\eta_z|^2}{\eta^2}\\
& \geq & \frac{|\rho_z|^2}{\eta \rho^2}+
\frac{|\eta_z|^2}{\eta^2}\\
&\mbox{apply (3)} & \\
& = & \frac{|\eta_z|^2}{(1+(-\rho)^{-1})^2\eta \rho^2}+
\frac{|\eta_z|^2}{\eta^2}\\
\eea
(7):
We use (5): 
\bea
\psi_{z\overline{z}} & = & (1+(-\rho)^{-1})\frac{\rho_{z\overline{z}}}{\eta}+\frac{|\rho_z|^2}{\eta \rho^2}+
\frac{|\eta_z|^2}{\eta^2}\\
& \geq & (1+(-\rho)^{-1})\frac{\rho_{z\overline{z}}}{\eta}\\
& \geq & \frac{\rho_{z\overline{z}}}{\eta}\\
& \mbox{we use (2)} & \\
& = & \frac{s^2}{\eta (|z|^2+s^2)^2}\\
\eea
(8): 
\bea
\psi_{z} & = & (-\log \eta)_z\\
& = & -\frac{\eta_z}{\eta}\\
& \mbox{we use (3) }\\
& = & -\frac{-(1+(-\rho)^{-1})\rho_z}{\eta}\\
|\psi_{z}|^2 & = & \frac{(1+(-\rho)^{-1})^2 |\rho_z|^2}{\eta^2}\\
\eea
(9): We use (7):
\bea
s^2\psi_{z\overline{z}} & \geq & |z|^2 \psi_{z\overline{z}}\\
& \geq & \frac{|z|^2s^2}{\eta (|z|^2+s^2)^2}\\
& \mbox{we use (1)} & \\
& = & s^2 \frac{|\rho_z|^2}{\eta}\\
\eea
 
\end{proof}

We pick a smooth decreasing function $\chi(t), t\in \mathbb R$ such that
$\chi(t)=1$ if $t\leq \frac{1}{2}$ and $\chi(t)=0$ if $t\geq 1.$

Let $C:= 2\int_{1/2<|w|^2<1} |\chi'(|w|^2)|^2 (|w|^2+1)^2 d\lambda(w).$ 

\begin{lemma}
2$\int_{s^2/2<|z|^2<s^2}|\chi'(\frac{|z|^2}{s^2})|^2 \cdot \frac{(|z|^2+s^2)^2}{s^6}d\lambda(z)=C.$
\end{lemma}

\begin{proof}
We use the substitution $z=sw.$ The result follows from the definition of $C.$
\end{proof}

\begin{lemma}
$$
\frac{\frac{\eta^2}{(-\rho+1)^2}}{1+\frac{\eta}{(-\rho+1)^2}}\geq \frac{1}{6}
$$
\end{lemma}

\begin{proof}
We use Lemma 14.3:
$$
\frac{\frac{\eta^2}{(-\rho+1)^2}}{1+\frac{\eta}{(-\rho+1)^2}}\geq 
\frac{\frac{(-\rho)^2}{(-\rho+1)^2}}{1+\frac{-2\rho}{(-\rho+1)^2}}
$$

Let $x=-\rho,$ then $x>1.$
\bea
\frac{\frac{x^2}{(x+1)^2}}{1+\frac{2x}{(x+1)^2}} & = &
\frac{x^2}{x^2+4x+1}\\
& \geq & \frac{x^2}{x^2+4x^2+x^2}\\
& = & \frac{1}{6}\\
\eea

\end{proof}

Let $\Omega$ be a bounded pseudoconvex domain in $\mathbb C^n$. Suppose that $$\sup_{z\in \Omega}|z_n|<e^{-1/2}.$$ Suppose there is a smooth function $\sigma$ on $\mathbb C^n$ such that $\Omega=\{z=(z_1,\dots,z_n)=(z',z_n); \sigma(z)<0\}.$
Let $\Omega'=\{z'; (z',0)\in \Omega\}.$
Assume that $d_{z'}(\sigma)\neq 0$ when $\sigma(z',0)=0.$ Let $f(z')$ be a holomorphic function
in an open set $V\supset \overline{\Omega}'$. For all $\epsilon>0$ small enough,
the function $G(z',z_n)=\chi(|z_n|^2/\epsilon^2)  f(z')$ is smooth on the subset of $\Omega$ where $\chi>0.$ We extend
$G$ smoothly to $\Omega$ by setting $G$ equal to zero when $|z_n|\geq \epsilon.$ In fact $G$ extends smoothly to a neighborhood of $\overline{\Omega}.$ 
Let $\nu:=\overline{\partial} G.$ Also let $\phi$ be a smooth strongly plurisubharmonic function in an open set containing $\overline{\Omega}.$ Let $0<\delta<\epsilon.$ Set $\phi':=\phi+\log(|z_n|^2+\delta^2)=\phi+\rho_\delta.$ By Lemma 14.4(2), $\rho_\delta$ is subharmonic. Let $A$ denote the Hermitian matrix $\{\frac{\partial^2\phi'}{\partial z_j\overline{\partial}z_k}\}$. Then $\int_\Omega |\nu|^2_Ae^{-\phi'}<\infty$.
By Theorem 14.1 there exists a function $u\in L^2(\Omega,\phi')$ so that
$\overline{\partial} u=\nu.$ Let $N_T$ denote the nullspace of $\overline{\partial}$ in 
$L^2(\Omega,\phi').$  We can subtract functions in $N_T$ and still have solutions to the equation.
Choosing the solution $u$ which is perpendicular to $N_T$ will minimize the norm of the solution.

Lemma 14.4 (5) shows that $\psi=\psi_\epsilon$ is subharmonic. Moreover $\psi$ is smooth and bounded.
We have that $\int u\overline{h}e^{-\phi'}=0$ for all holomorphic functions $h$ in $L^2(\Omega,\phi').$

Hence $\int(ue^\psi)\overline{h}e^{-(\phi'+\psi)}=0$ for all holomorphic functions $h$ in $L^2(\Omega,\phi')$.
Since $\psi$ is bounded, it follows that this is same set of holomorphic functions as those
in $L^2(\Omega,\phi'+\psi).$ Hence $ue^{\psi}\in L^2(\Omega,\phi'+\psi)$ and is perpendicular to
the nullspace of $\overline{\partial}$ in $L^2(\Omega,\phi'+\psi).$ 
Since $\psi$ is smooth, it follows also that $\overline{\partial}(ue^\psi)=
\nu e^\psi+ue^\psi \overline{\partial}\psi$ is in $L^2(\Omega, \phi'+\psi).$ This shows that
$ue^{\psi}$ is in the domain of $T$ using the weight $e^{-(\phi'+\psi)}.$ By smoothness of $\phi'$ and strong plurisubharmonicity
we also know that $\int |\overline{\partial}(ue^{\psi})|^2_{A_{(\phi'+\psi,0)}}e^{-(\phi'+\psi)}<\infty.$ 

It follows from Theorem 14.1 that
$$
\int |ue^\psi|^2e^{-\phi'-\psi}\leq \int |\nu e^\psi+ue^\psi\overline{\partial}\psi|^2_{A_{(\phi'+\psi,0)}}e^{-\phi'-\psi}.
$$

\section{Chen's proof of Ohsawa-Takegoshi, cont.}

 Hence
$$
\int |u|^2e^{-\phi'+\psi}\leq \int |\nu+u\overline{\partial}\psi|^2_{A_{(\phi'+\psi,0)}}e^{-\phi'+\psi}.
$$

Using the small constant, large constant lemma, we get for any $r>0:$
\bea
\int_\Omega |u|^2 e^{-\phi'+\psi} & \leq & \int_{supp(\nu)}|\nu+u\overline{\partial}\psi|^2_{A_{(\phi'+\psi,0)}}e^{-\phi'+\psi}\\
& + & \int_{\Omega\setminus supp(\nu)}|\nu+u\overline{\partial}\psi|^2_{A_{(\phi'+\psi,0)}}e^{-\phi'+\psi}\\
& \leq & (1+r)\int_{supp(\nu)}|u\overline{\partial}\psi|^2_{A_{(\phi'+\psi,0)}}e^{-\phi'+\psi}\\
& + & (1+\frac{1}{r})\int_{supp(\nu)}|\nu|^2_{A_{(\phi'+\psi,0)}}e^{-\phi'+\psi}\\
& + & \int_{\Omega\setminus supp(\nu)}|u\overline{\partial}\psi|^2_{A_{(\phi'+\psi,0)}}e^{-\phi'+\psi}\\
& \leq & (1+\frac{1}{r})\int_{\Omega}|\nu|^2_{A_{(\phi'+\psi,0)}}e^{-\phi'+\psi}\\
& + & 
\int_{\Omega}|u\overline{\partial}\psi|^2_{A_{(\phi'+\psi,0)}}e^{-\phi'+\psi}\\
& + & r\int_{supp(\nu)}|u\overline{\partial}\psi|^2_{A_{(\phi'+\psi,0)}}e^{-\phi'+\psi}\\
\eea

We estimate the three integrals.
The form $\nu$ is a multiple of  $d\overline{z}_n$.
Also the function $f$ only depends on $z'.$
By Lemma 14.3 (7), if $g=(0',\alpha),$ then
$|g|^2_A \leq |g|^2 \frac{\eta_\epsilon (|z_n|^2+\epsilon^2)^2}{\epsilon^2}.$
Hence 

\bea
\int_{\Omega}|\nu|^2_Ae^{-\phi'+\psi} & \leq & 
\int_\Omega |\nu|^2 \frac{\eta_\epsilon (|z_n|^2+\epsilon^2)^2}{\epsilon^2}
e^{-(\phi+\log (|z_n|^2+\delta^2))-\log \eta_\epsilon}\\
& \leq & \int (|\chi'(|z_n|^2/\epsilon^2)|)^2 \frac{|z_n|^2}{\epsilon^4}
\frac{(|z_n|^2+\epsilon^2)^2}{\epsilon^2}\frac{1}{|z_n|^2+\delta^2}|f(z')|^2 e^{-\phi}\\
& \leq & 2\left(\int_{\frac{\epsilon^2}{2} \leq |z_n|^2\leq \epsilon^2}
|\chi'|^2 \frac{(|z_n|^2+\epsilon^2)^2}{\epsilon^6}\frac{|z_n|^2}{|z_n|^2+\delta^2}\right)
\left(\int_{\Omega'}|f|^2e^{-\phi}\right)
\eea

The reason for the factor 2 is to take into account the error term in  $\phi(z',z_n)=\phi(z',0)+\mathcal O(z_n)$ and also that we need values of $f$ in a small neighborhood of $\overline{\Omega}'$. So this is valid when $\epsilon$ is small enough. Hence, using Lemma 14.4 we get
$$(1+r^{-1})\int_{\Omega}|\nu|^2_{A_{(\phi'+\psi,0)}}e^{-\phi'+\psi}\leq
2(1+r^{-1})C\int_{\Omega'}|f|^2e^{-\phi}.
$$

We next estimate the second integral.
\bea
\int_\Omega |u\overline{\partial}\psi|^2_{A_{(\phi'+\psi,0)}}e^{-\phi'+\psi} & = &
\int |u|^2 |\overline{\partial}\psi|^2_{A_{(\phi'+\psi,0)}} e^{-\phi'+\psi}\\
\eea

We have that $A_{\phi'+\psi,0}\geq A_{\psi,0}.$ Here $A_{\psi,0}$ has only one nonzero entry, at place $(n,n)$ where the entry is $\psi_{z_n,\overline{z}_n}.$
 By 
Lemma 14.3 (6), we have $a_{nn}\geq \left(\frac{1}{\eta^2}+\frac{1}{\eta(1+(-\rho)^{-1})^2}\right)
|\eta_z|^2.$ Moreover by 14.3 (8), we have
$|\psi_z|^2= \frac{(1+(-\rho)^{-1})^2}{\eta^2}|\rho_z|^2.$ By 14.3 (3) we
have that $\frac{|\rho_z|^2}{|\eta_z|^2}=(1+(-\rho)^{-1})^{-2}$.

Hence
\bea
\int_\Omega |u\overline{\partial}\psi|^2_Ae^{-\phi'+\psi} & \leq &
\int |u|^2 \frac{(1+(-\rho)^{-1})^2}{\eta^2}|\rho_z|^2\frac{1}{\left(\frac{1}{\eta^2}+\frac{1}{\eta(1+(-\rho)^{-1})^2}\right)
|\eta_z|^2}e^{-\phi'+\psi}\\
& = & \int_\Omega |u|^2 \frac{1}{1+\frac{\eta}{(1+(-\rho)^{-1})^2}}e^{-\phi'+\psi}\\
\eea

Finally, we do the third integral, over Supp$(\nu).$

We use 14.3 (8): $|\overline{\partial}\psi|^2\leq \frac{4}{\eta^2}|\overline{\partial}\rho|^2$

and 14.3 (9): $\partial\overline{\partial}\psi \geq \frac{|\overline{\partial}\rho|^2}{\eta}$, when
$|z_n|^2\leq \epsilon^2.$

\bea
r\int_{supp (\nu)}|u\overline{\partial}\psi|^2_Ae^{-\phi'+\psi}
& \leq & r\int_\Omega 
|u|^2 \frac{4}{\eta^2}|\overline\partial \rho|^2 \frac{\eta}{|\overline{\partial}\rho|^2}e^{-\phi'+\psi}\\
& = & r\int_\Omega |u|^2 \frac{4}{\eta}e^{-\phi'+\psi}\\
\eea

We combine the above calculations:
\bea
\int_\Omega |u|^2 e^{-\phi'+\psi}& \leq & (1+\frac{1}{r})\int_{\Omega}|\nu|^2_Ae^{-\phi'+\psi}\\
& + & 
\int_{\Omega}|u\overline{\partial}\psi|^2_Ae^{-\phi'+\psi}\\
& + & r\int_{supp(\nu)}|u\overline{\partial}\psi|^2_Ae^{-\phi'+\psi}\\
& \leq & 2(1+r^{-1})C\int_{\Omega'}|f|^2e^{-\phi}\\
& + & \int_\Omega |u|^2 \frac{1}{1+\frac{\eta}{(1+(-\rho)^{-1})^2}}e^{-\phi'+\psi}\\
& + &  r\int_\Omega |u|^2 \frac{4}{\eta}e^{-\phi'+\psi}\\
\eea

Hence we get 
\bea
\int_\Omega |u|^2 e^{-\phi'+\psi} & - & \int_\Omega |u|^2 \frac{1}{1+\frac{\eta}{(1+(-\rho)^{-1})^2}}e^{-\phi'+\psi}\\ & - & r\int_\Omega |u|^2 \frac{4}{\eta}e^{-\phi'+\psi}\\
& \leq & 2(1+r^{-1})C\int_{\Omega'}|f|^2e^{-\phi}\\
\eea

\bea
\int_\Omega |u|^2 e^{-\phi'+\psi}\left(1-\frac{1}{1+\frac{\eta}{(1+(-\rho)^{-1})^2}}
-\frac{4r}{\eta}\right) & \leq & 2(1+r^{-1})C\int_{\Omega'}|f|^2e^{-\phi}\\
\eea

Hence 
$$
\int_\Omega |u|^2 e^{-\phi'+\psi}\left(\frac{\frac{\eta}{(1+(-\rho)^{-1})^2}}{1+\frac{\eta}{(1+(-\rho)^{-1})^2}}-\frac{4r}{\eta}\right)\leq 2(1+r^{-1})C\int_{\Omega'}|f|^2e^{-\phi}
$$

From Lemma 14.5 we get

$$
\int_\Omega |u|^2 e^{-\phi'+\psi}\left(\frac{\frac{1}{6}-{4r}}{\eta}\right)\leq 2(1+r^{-1})C\int_{\Omega'}|f|^2e^{-\phi}
$$

Choose $0<r_0<\frac{1}{24}$ to minimize $\frac{1+\frac{1}{r}}{\frac{1}{6}-4r}$

and set $C'=\frac{1+\frac{1}{r_0}}{\frac{1}{6}-4r_0}$.

Then 
$$
\int_\Omega \frac{1}{\eta}|u|^2 e^{-\phi'+\psi}\leq 2C'C\int_{\Omega'}|f|^2e^{-\phi}
$$
Since $\psi=-\log \eta$, we get
$$
\int_\Omega \frac{1}{\eta^2}|u|^2 e^{-\phi'}\leq 2C'C\int_{\Omega'}|f|^2e^{-\phi}.
$$
Recall that $\eta=-\rho+\log (-\rho)<-2\rho=-2 \log(|z|^2+\epsilon^2)$, so
$$
\int_\Omega \frac{1}{(\log (|z_n|^2+\epsilon^2))^2}|u|^2 e^{-\phi'}\leq 8C'C\int_{\Omega'}|f|^2e^{-\phi}.
$$

Using that $\phi'=\phi+\log(|z_n|^2+\delta^2)$, we obtain

$$
\int_\Omega \frac{1}{(|z_n|^2+\delta^2)(\log (|z_n|^2+\epsilon^2))^2}|u|^2 e^{-\phi}\leq 8C'C\int_{\Omega'}|f|^2e^{-\phi}.
$$

Hence we have found functions $u=u_{\delta,\epsilon}$ solving the problem
$\overline{\partial}u_{\delta,\epsilon}=\nu=\nu_\epsilon$ om $\Omega$ satisfying the above estimate.
In particular, the family $u_{\delta,\epsilon}$ are uniformly in $L^2_{loc}$ and hence also in $L^1_{loc}$
for $\epsilon$ fixed. Applying Corollary H.2.2.5 for fixed $\epsilon$ to the family $u_{\delta,\epsilon}-f\chi(|z_n|^2/\epsilon^2)$ we can find a subsequence $u_{\delta_j,\epsilon}\rightarrow u_\epsilon$ such that
$\overline{\partial} u_\epsilon=\nu_\epsilon$ and

$$
\int_\Omega \frac{1}{|z_n|^2(\log (|z_n|^2+\epsilon^2))^2}|u_\epsilon|^2 e^{-\phi}\leq 8C'C\int_{\Omega'}|f|^2e^{-\phi}.
$$

Notice that this forces $u_\epsilon(z',0)=0$ since the function $\int_{z\in \mathbb C;|z|<a}\frac{1}{|z|^2}d\lambda=\infty$ for all $a>0.$

Hence, the function $f\chi(|z_n|^2/\epsilon^2)-u_\epsilon$ is a holomorphic function on $\Omega$
which extends $f$. Since the functions are uniformly in $L^2_{loc}$ there is a limit holomorphic function $F$ for a subsequence $\epsilon_j\searrow 0$ which is in $L^2_{loc}$ in $\Omega$ and which extends $f$.
If we fix any compact subset $K\subset \Omega$ which does not intersect the hyperplane $z_n=0$, we have for small enough $\epsilon$ that $f\chi(|z_n|^2/\epsilon^2)=0$ on $K.$
Hence we must have that

$$
\int_K \frac{1}{|z_n|^2(\log (|z_n|^2)^2}|F|^2 e^{-\phi}\leq 8C'C\int_{\Omega'}|f|^2e^{-\phi}.
$$

But then also

$$
\int_\Omega\frac{1}{|z_n|^2(\log (|z_n|^2)^2}|F|^2 e^{-\phi}\leq 8C'C\int_{\Omega'}|f|^2e^{-\phi}.
$$

In conclusion, we have shown:

\begin{lemma}
Let $\Omega$ be a bounded pseudoconvex domain in $\mathbb C^n$. Suppose that $\sup_{z\in \Omega}|z_n|<e^{-1/2}.$ Suppose there is a smooth function $\sigma$ on $\mathbb C^n$ such that $\Omega=\{z=(z_1,\dots,z_n)=(z',z_n); \sigma(z)<0\}.$
Let $\Omega'=\{z'; (z',0)\in \Omega\}.$
Assume that $d_{z'}(\sigma)\neq 0$ when $\sigma(z',0)=0.$ Also let $\phi$ be a smooth strongly plurisubharmonic function in an open set containing $\overline{\Omega}.$ Let $f(z')$ be a holomorphic function
in an open set $V\supset \overline{\Omega}'$. 
Then there exists a holomorphic function $F:\Omega\rightarrow \mathbb C$
such $F(z',0)=f(z')$ for all $z'\in \Omega'$. Also
$$
\int_\Omega \frac{|F(z)|^2}{|z_n|^2(\log |z_n|^2)^2}e^{-\phi}\leq 8C'C\int_{\Omega'}|f|^2 e^{-\phi}.
$$
\end{lemma}

We now prove the Ohsawa-Takegoshi Theorem. This version of the Ohsawa-Takegoshi Theorem
was first proved by Demailly in ICTP lecture notes, vol 6, 1-148, Trieste, 2000).\\

\begin{theorem}
Let $\Omega$ be a pseudoconvex domain in $\mathbb C^n$. Suppose that $\sup_{z\in \Omega}|z_n|<e^{-1/2}.$ 
Let $\Omega'=\{z'; (z',0)\in \Omega\}.$
 Also let $\phi$ be a plurisubharmonic function on  $\Omega.$ Let $f(z')$ be a holomorphic function
on $\Omega'$. 
Then there exists a holomorphic function $F:\Omega\rightarrow \mathbb C$
such $F(z',0)=f(z')$ for all $z'\in \Omega'$. Also
$$
\int_\Omega \frac{|F(z)|^2}{|z_n|^2(\log |z_n|^2)^2}e^{-\phi}\leq 8C'C\int_{\Omega'}|f|^2 e^{-\phi}.
$$
\end{theorem}

\begin{proof}
Let $f(z')$ be a holomorphic function on $\Omega'$.  We proved in Theorem 4.2.$8_a$ that there exists an extension $F$ which is holomorphic on $\Omega.$ We only need to show that in case the integral on the right is finite, we can find $F$ satisfying the inequality. 
By Theorem H.2.6.11, there exists a smooth plurisubharmonic function $\psi$ on $\Omega$ so that all sublevel sets $\psi\leq c$ are compact. By Sards Lemma, the gradient of $\psi$ is nonzero on almost all level sets. The same applies to the restriction $\psi(z',0)$ to $\Omega'$. Hence, if we let
$\Omega_1=\{\psi<c\}, \Omega'_1=\{(z',0)\in \Omega_1\}$ then there is a function $\sigma$ as in Lemma 15.1. We can do this construction for arbitrarily large $c$.
Let $\phi_k$ be a sequence of smooth strongly plurisubharmonic functions defined on neighborhoods of 
$\overline{\Omega}_1$ such that $\phi_k\searrow \phi$ pointwise.

For each $k$ there exists by Lemma 15.1 a holomorphic extension of $f$ to $F_k$ on $\Omega_1$ such that
$$
\int_{\Omega_1} \frac{|F_k(z)|^2}{|z_n|^2(\log |z_n|^2)^2}e^{-\phi_k}\leq 8C'C\int_{\Omega'}|f|^2 e^{-\phi_k}.
$$

Next note that $e^{-\phi_k}\nearrow e^{-\phi}.$ Hence we get for $m\geq k$
that
\bea
\int_{\Omega_1}|F_m|^2e^{-\phi_k} & \leq & \int_{\Omega_1}\frac{|F_m|^2}{|z_n|^2(\log(|z_n|^2)^2}e^{-\phi_k} \\
& \leq & \int_{\Omega_1}\frac{|F_m|^2}{|z_n|^2(\log(|z_n|^2)^2}e^{-\phi_m} \\
& \leq & 8C'C\int_{\Omega'}|f|^2 e^{-\phi_m}\\
& \leq & 8C'C\int_{\Omega'}|f|^2 e^{-\phi}\\
\eea

Hence we see that the family $\{F_m\}$ is locally uniformly bounded in $L^2$, hence there is a subsequence which converges uniformly to a holomorphic function $F_c$ on $\Omega_1.$
This function is an extension of $f$. Moreover,
$$
\int_{\Omega_1}\frac{|F_c|^2}{|z_n|^2(\log(|z_n|^2)^2}e^{-\phi_k} \leq 8C'C\int_{\Omega'}|f|^2 e^{-\phi}
$$
for every $k.$
Hence
$$
\int_{\Omega_1}\frac{|F_c|^2}{|z_n|^2(\log(|z_n|^2)^2}e^{-\phi} \leq 8C'C\int_{\Omega'}|f|^2 e^{-\phi}.
$$

Next we can take a limit $F$ of a subsequence $F_{c_j}, c_j\rightarrow \infty$. Then $F$ extends $f$
and
$$
\int_{\Omega}\frac{|F|^2}{|z_n|^2(\log(|z_n|^2)^2}e^{-\phi} \leq 8C'C\int_{\Omega'}|f|^2 e^{-\phi}.
$$
\end{proof}

\section{Appendix 1, $\overline{\partial}$ on polydiscs}

2.3 The inhomogeneous Cauchy-Riemann equations in a polydisc.\\

We let $f=\sum_{|I|=p,|J|=q}f_{I,J}dz^I\wedge d\overline{z}^J$ be a $(p,q)$ form.
Note that if $p>n$ or $q>n$ the form must be zero because some $dz_i$ or $d\overline{z}_j$ must be repeated and when we switch them the form changes sign.
We will use the notation $f=\sum'_{I,J}f_{I,J}dz^I\wedge d\overline{z}^J$ if all the multiindeces are in increasing order, so if $I=(i_1,\dots, i_p)$ then $i_1<i_2\cdots < i_p$ and similar for $J$.
In this case the coefficients are unique.\\

We say that $\sum'_{I,J}dz^I\wedge d\overline{z}^J$ is a $\mathcal C^\infty$ form if all coefficients
are in $\mathcal C^\infty(\Omega).$

{\bf Theorem 2.3.3} Let $D$ be a polydisc and let $f\in \mathcal C^\infty_{(p,q+1)}(D).$
Assume that $\overline\partial f=0.$ Suppose $D'\subset \subset D$ is a concentric smaller polydisc.
Then there exists $u\in \mathcal C^\infty_{(p,q)}(D')$ such that $\overline\partial u=f.$

{\bf Lemma 2.3.3a}
Let $D'\subset \subset D''\subset \subset D$ be concentric polydiscs with polyradii $r',r'',r$ in $\mathbb C^n.$
Let $1\leq k\leq n.$ Suppose that $g\in \mathcal C^{\infty}(D)$ and suppose that
$\frac{\partial g}{\partial\overline{z}_j}=0$ for all $j>k.$ Then there exists a $\mathcal C^\infty$
function $G$ on $D$ such that $\frac{\partial G}{\partial \overline{z}_k}=g$ on $D''$ and
$\frac{\partial G}{\partial \overline{z}_j}=0$ on $D$ for all $j>k.$

\begin{proof}
Let $\psi(z_k)\in \mathcal C^\infty_0(\Delta^1(0,r_k))$which is $1$ on $\Delta^1(0,r''_k).$
We define
\bea
G(z) & = & \frac{1}{2\pi i}\int \frac{\psi(\tau)g(z_1,\dots,z_{k-1},\tau,z_{k+1},\dots,z_n)}{\tau-z_k}d\tau\wedge d\overline{\tau}\\
& & \mbox{Set}\; \sigma=z_k-\tau\\
& = & \frac{1}{2\pi i}\int \frac{\psi(z_k-\sigma)g(z_1,\dots,z_{k-1},z_k-\sigma,z_{k+1},\dots,z_n)}{-\sigma}d\sigma\wedge d\overline{\sigma}\\
\eea

From the last expression we see that $G\in \mathcal C^\infty(D).$ Moreover we see that
$\frac{\partial G}{\partial \overline{z}_j}=0$ on $D$ for all $j>k.$
Finally,
\bea
\frac{\partial G}{\partial \overline{z}_k} & = & 
\frac{1}{2\pi i}\int \frac{\frac{\partial (\psi(z_k-\sigma)g(z_1,\dots,z_{k-1},z_k-\sigma,z_{k+1},\dots,z_n))}{\partial \overline{z}_k}}{-\sigma}d\sigma\wedge d\overline{\sigma}\\
& = & \frac{1}{2\pi i}\int \frac{\frac{\partial (\psi(\tau)g(z_1,\dots,z_{k-1},\tau,z_{k+1},\dots,z_n))}{\partial \overline{\tau}}}{\tau-z_k}d\tau\wedge d\overline{\tau}\\
& = & \mbox{using Theorem 1.2.1}\\
& & \psi(z_k)g(z).\\
\eea

Since $\psi=1$ on $D'$, we are done.
\end{proof}

Let $q\leq k\leq n.$\\

{\bf Lemma 2.3.$3_k$}
Let $D$ be a polydisc and let $f\in \mathcal C^\infty_{(p,q+1)}(D).$
Assume that $\overline\partial f=0.$ Assume furthermore that $f$ contains no term
with $d\overline{z}_{k+1}, \dots, d\overline{z}_n.$ Suppose $D'\subset \subset D$ is a concentric smaller polydisc.
Then there exists $u\in \mathcal C^\infty_{(p,q)}(D')$ such that $\overline\partial u=f.$\\

Notice that if $k=q$, then all terms in $f$ must be zero because only $d\overline{z}_1,\dots,d\overline{z}_p$ can appear and you need $p+1$ differentials. Hence Lemma 2.3.$3_q $
is true. We next show that if $q\leq k-1<n$, then Lemma 2.3.$3_{k-1}$ implies Lemma 2.3.$3_k.$

Let the assumptions be as in Lemma 2.3.$3_k.$ We can then write
$$
f=d\overline{z}_k\wedge g+h
$$
where $g$ is a $(p,q)$ form with no terms with $d\overline{z}_{k+1}, \dots, d\overline{z}_n$
and $h$ is a $(p,q+1)$ form without any of the  $d\overline{z}_{k+1}, \dots, d\overline{z}_n$.

We can write $g=\sum'_{I,J}g_{I,J}dz^I\wedge d\overline{z}^J$ where $I,J$ do not contain any of the indices $k,\cdots,n.$ The only terms in $\overline\partial f$ containing both $I,J,d\overline{z}_k,d\overline{z}_j$ for some $j>k$, come from 
$\frac{\partial g_{I,J}}{\partial \overline{z}_j}.$ Hence, since $\overline\partial f=0$, it follows
that $\frac{\partial g_{I,J}}{\partial \overline{z}_j}=0$ for all $j>k.$ Hence by Lemma 2.3.3a there exist
$G_{I,J}\in \mathcal C^\infty(D)$ such that $\frac{\partial G_{I,J}}{\partial \overline{z}_k}=g$ on $D''$ and
$\frac{\partial G_{I,J}}{\partial \overline{z}_j}=0$ on $D$ for all $j>k.$
Let $G=\sum'G_{I,J}dz^I\wedge d\overline{z}^J.$ Then $\overline\partial G=d\overline{z}_k\wedge
g+h_1$ on some $D''$, $D'\subset \subset D''\subset \subset D$, where $h_1$ does not contain any of the differentials $d\overline{z}_k,\dots,d\overline{z}_n.$
The form $f-\overline\partial{G}=h-h_1$ is $\overline\partial-$ closed and contains none of
$d\overline{z}_k,\dots,d\overline{z}_n.$ Hence, by the inductive hypothesis,
there is a $\mathcal C^\infty $ form $v$ on $D$ such that $\overline{\partial v}=f-\overline{\partial}= G.$ Hence $v+G$ solves the problem $\overline{\partial} (v+G)=f.$\\

Finally, to prove Theorem 2.3.3, we observe that this is equivalent to Lemma 2.3.$3_n.$\\

{\bf Corollary 2.3.3b}
 Let $D$ be a polydisc and let $f\in \mathcal C^\infty_{(p,q+1)}(D).$
Assume that $\overline\partial f=0.$
Then there exists $u\in \mathcal C^\infty_{(p,q)}(D)$ such that $\overline\partial u=f.$\\

\begin{proof}
Let $D_n$ denote an increasing sequence of concentric polydiscs such that $D_j\subset \subset D_{j+1}$ and $D=\cup D_j.$
Let us first consider the case $q=0.$ 
We will make an inductive construction. Suppose we have a form $v_j=\sum{I,0}v_{I,0,j}dz^I\in \mathcal C^\infty_{(p,q)}(D)$
such that $\overline{\partial} v_j=f$ on $D_j.$ Use Theorem 2.3.3 to find a $v_{j+1}'=\sum{I,0}g_{I,0,j+1}dz^I\in  \mathcal C^\infty_{(p,q)}(D)$ such that $\overline{\partial}v_{j+1}'= f$ on $D_{j+1}.$ 
Then all the coefficients of the $(p,0)$ form  $v_{j+1}'-v_j$ are holomorphic functions on $D_j.$
Hence they have a normally convergent power series in $D_j$ Now we can approximate
each $v_{I,0,j}-v'_{I,0,j+1}$ closer than $\frac{1}{2^j}$ by holomorphic polynomials $P_{I,0,j}$ on $D_{j-1}.$
We now define a new solution to $\overline{\partial}u=f$ on $D_{j+1}$
by setting $v_{j+1}=\sum_{I} (v'_{I,0,j+1)}+P_{I,0,j})dz^I$,
This sequence of solutions will converge normally to a solution on $D.$\\

Next, let us consider the case $q>0.$
Again we make an inductive construction. So assume we have a  solution $v_j$ defined on $D$ and
such that $\overline{\partial}v_{j}=f$ on $D_j.$ Next, let $v_{j+1}'$ be smooth $(p,q)$ form on $D$
such that $\overline\partial v'_{j+1}=f$ on $D_{j+1}.$ Then the form $v'_{j+1}-v_j$ is a smooth $(p,q)$
form on $D$ such that $\overline\partial (v'_{j+1}-v_j)=0$ on $D_j$. Since $q>0$, theorem 2.3.3 applies to find a smooth $(p,q-1)$ form $w_j$ on $D$ so that $\overline{\partial}w_j=v'_{j+1}-v_j$
on $D_{j-1}.$ Define $v_{j+1}$ on $D$ by $v_{j+1}=v'_{j+1}-\overline\partial w_j$. Then still
$\overline\partial v_{j+1}=\overline{\partial} v'_{j+1}=f$ on $D_{j+1}$ and
$v_{j+1}=v_j$ on $D_{j-1}.$ Hence we get a solution on $D$ by letting $v=\lim v_j.$

\end{proof}

\section{Appendix 2, Solution of $T^*f=v$}

This is the analogue of Lemma 4.1.1 in Hormander for $T^*$ instead of $T.$ It was not needed
in our presentation.\\

{\bf Lemma 4.1.2}
Let $G_1,G_2$ be reflexive Banach spaces and let $T:G_1\rightarrow G_2$ be a densely defined
closed linear operator.  Let $F\subset G_2$ be a closed subspace containing the range of $T, R_T.$
Assume that there exists a constant $C>0$ such that
$$
(*) \;\;\;\|y_{|F}\|_{G_2'} \leq C\|T^*y\|_{G'_1}\; \forall \; y\in  D_{T^*}.
$$
Then for every $v\in G_1'$ which is in $N_T^{\perp}$ there is an $f\in D_{T^*}$ such that $T^*(f)=v$
and
$$
\|f\|_{G_2'}\leq C \|v\|_{G_1'}
$$

We prove first two lemmas.

{\bf Lemma 4.1.$2_a$}
Let $G_1,G_2$ be two reflexive Banach spaces, $T:G_1\rightarrow G_2$ a closed, densely defined linear operator. Then $$N_T=\{x\in G_1 \; \mbox{such that} \; T^*(y)(x)=0\; \forall y \in D_{T^*}\}.$$
Also $\overline{R}_{T^*}= N_T^\perp.$\\

\begin{proof}
Suppose that $x\in N_T$ and $y\in D_{T^*}.$ Then $(T^*y)(x)=y(Tx)=y(0)=0.$ Hence
$N_T\subset \{x\in G_1 \; \mbox{such that} \; T^*(y)(x)=0\; \forall y \in D_{T^*}\}$.
Also $\overline{R}_{T^*}= N_T^\perp.$
Suppose that $x\in G_1, T^*y(x)=0\;\forall y\in D_{T^*}.$ Then for the isometric $\phi=\phi(x)\in G_1''$
we have $\phi(x)(y)(T^*y)=0$ for all $y\in D_{T^*}.$ Hence $(\phi,0)\in G_{T^*}^\perp.$ This implies that $(\phi,0)\in G_{T^{**}}$ so $(x,0)\in G_T.$ Therefore $Tx=0$ so $x\in N_{T}.$ This proves the first part of the lemma.

Suppose next that $x\in N_T$ and  $z\in R_{T^*}$, $z=T^*(y).$ Then $z(x)=y(Tx)=y(0)=0.$
Hence $z\in N_T^\perp.$ Hence $R_{T^*}\subset N_T^\perp.$ Since $N_T^\perp$ is closed, we see that
$\overline{R}_{T^*}\subset  N_T^\perp.$ Suppose finally that there exists some $z\in N_T^{perp}$ which is not in $\overline{R}_{T^*}$. Then, by the Hahn-Banach theorem there exists a $\phi\in G_1''$
so that $\phi(T^*(y))=0$ for all $y\in D_{T^*}$ but $\phi(z)\neq 0.$ Hence $(\phi,0)\in G_{T^*}^\perp.$
Hence there exists $x\in G_1$ so that $(x,0)\in G_T$. Hence $Tx=0$ so $x\in N_T$ and
$z(x)=\phi(z)\neq 0$. This is a contradiction since $z\in N_T^\perp.$
\end{proof}

{\bf Lemma 4.1.$2_b$}
Assume the hypothesis of Lemma 4.1.2. Then $T^*$ has closed range.\\

\begin{proof}
Pick a $y\in D_{T^*}.$ Then $|y(x)|\leq C \|T^*(y)\|\|x\|\; \forall \; x\in F.$
By the Hahn-Banach theorem we can extend $y_{|F}$ to $\tilde{y}$ on $G_2$ with the same norm.
Hence $\|\tilde{y}\|_{G_2'}\leq C\|T^*(y)\|_{G_1'}.$ Also, $\tilde{y}-y$ vanishes on $F.$
Therefore $\tilde{y}-y \in D_{T^*}$ and $T^*(\tilde{y}-y)=0.$
Since $y\in D_{T^*}$ it follows that $\tilde{y}\in D_{T^*}$ and $T^*(\tilde{y})=T^*(y).$
Next, let $\{y_n\}$ be a sequence in $D_{T^*}$ such that $\{T^*(y_n)$ converges to some $z\in G_1'.$
We can assume that $\|T^*(y_{n+1}-y_n)\|<\frac{1}{2^n}.$ Set $u_1=y_1$ and
$u_{n+1}=y_{n+1}-y_n$ for $n\geq 1.$ So $y_{n}=u_1+\cdots + u_{n}$ for $n\geq 2.$
Replace $u_n$ by $\tilde{u}_n$ as above, and define $\tilde{y}_n=\tilde{u}_1+\cdots +\tilde{u}_n.$
Then it follows that the sequence $T^*(\tilde{y}_n)\rightarrow z$ and  hence
also $\tilde{y}_n$ is a Cauchy sequence converging to some $y\in G_2'$. Since the graph of $T^*$ is closed, we see that $y\in D_{T^*}$ and $z=T^*(y).$ Hence the range of $T^*$ is closed.
\end{proof}

Next we prove Lemma 4.1.2.:

\begin{proof}
Suppose that $v\in G_1'$ and that $v\in N_T^{\perp}.$
Then by Lemma 1, $v\in \overline{R}_{T^*}$ and by Lemma 2, we have that $v=T^*(f)$ for some
$f\in D_{T^*}.$ By our hypothesis, it follows that $\|f_{|F}\{_{G_2'}\leq C \|v\|_{G_1'}.$ Again we use Hahn-Banach and replace $f$ by $\tilde{f}$ on $G_2$ with the same norm on $G_2$ as the norm of $f$ restricted to $F.$ Again as above, $\tilde{f}\in D_{T^*}$ and $T^*(\tilde{f})=T^*(f)=v.$ Also,
$\|\tilde{f}\|_{G_2'}\leq C \|v\|_{G_1'}.$
\end{proof}

\section{Appendix 3, Ohsawa-Takegoshi in $L^p$}

Ohsawa-Takegoshi in $L^p$ spaces.\\

We present an elegant proof for the validity of Ohsawa-Takegoshi for $L^p$ spaces, $0<p<2.$
This proof is due to Berndtsson-Paun: Bergman kernels and subadjunction. arXiv: 1002.4145v1 
[math.AG] 22Feb 2010

\begin{theorem}
Let $\Omega$ be a pseudoconvex domain in $\mathbb C^n$. Suppose that $\sup_{z\in \Omega}|z_n|<e^{-1/2}.$ 
Let $\Omega'=\{z'; (z',0)\in \Omega\}.$
 Also let $\phi$ be a plurisubharmonic function on  $\Omega.$ Let $f(z')$ be a holomorphic function
on $\Omega$. 
Then there exists a holomorphic function $F:\Omega\rightarrow \mathbb C$
such $F(z',0)=f(z')$ for all $z'\in \Omega'$. Also
$$
\int_\Omega \frac{|F(z)|^2}{|z_n|^2(\log |z_n|^2)^2}e^{-\phi}\leq C_0\int_{\Omega'}|f|^2 e^{-\phi}.
$$
\end{theorem}

\begin{theorem}
Let $\Omega$ be a pseudoconvex domain in $\mathbb C^n$. Suppose that $\sup_{z\in \Omega}|z_n|<e^{-1/2}.$ 
Let $\Omega'=\{z'; (z',0)\in \Omega\}.$
 Also let $\phi$ be a plurisubharmonic function on  $\Omega.$ Let $f(z')$ be a holomorphic function
on $\Omega$. Let $0<p\leq 2.$
Then there exists a holomorphic function $F:\Omega\rightarrow \mathbb C$
such $F(z',0)=f(z')$ for all $z'\in \Omega'$. Also
$$
\int_\Omega \frac{|F(z)|^p}{|z_n|^2(\log |z_n|^2)^2}e^{-\phi}\leq C_0\int_{\Omega'}|f|^p e^{-\phi}.
$$
\end{theorem}

The constant $C_0$ is the same as for $L^2.$
We introduce the notation $\sigma=\frac{e^{-\phi}}{|z_n|^2(\log|z_n|^2)^2}.$

\begin{proof}
We can exhaust $\Omega$ by smoothly bounded strongly pseudoconvex domains. It suffices to prove the estimate for those if the function $f$ extends to a neighborhood of the closure and there exists some extension defined in a neighborhood of the closure. It also suffices to assume that $\phi$ is smooth on $\Omega.$

We can assume that $\int_{\Omega'}|f|^pe^{-\phi}=1$. Pick some holomorphic extension $F_1$
and let $\int_\Omega |F_1|^p \sigma=:A<\infty.$\\ 
We then apply the Ohsawa-Takegoshi theorem to $\phi_1=\phi+(1-\frac{p}{2})\log |F_1|^2.$
We then get an extension $F_2$
such that
\bea
\int_\Omega \frac{|F_2(z)|^2}{|z_n|^2(\log |z_n|^2)^2}e^{-\phi_1} & \leq & C_0\int_{\Omega'}|f|^2 e^{-\phi_1}\\
\int_\Omega \frac{|F_2(z)|^2\sigma}{e^{(1-\frac{p}{2})\log |F_1|^2}} & \leq  & C_0 \int_{\Omega'}|f|^2\frac{1}{e^{(1-\frac{p}{2})\log |F_1|^2}}e^{-\phi}\\
\int_\Omega \frac{|F_2(z)|^2\sigma}{ |F_1|^{2-p}}& \leq  & C_0 \int_{\Omega'}\frac{|f|^2}{|F_1|^{2-p}}e^{-\phi}\\
& = & C_0 \int_{\Omega'}\frac{|f|^2}{|f|^{2-p}}e^{-\phi}\\
& = & C_0 \int_{\Omega'}|f|^pe^{-\phi}\\
& = & C_0\\
\eea

\bea
\int_\Omega |F_2|^p\sigma & = & \int_\Omega |F_2|^p \sigma^{p/2}\sigma^{1-\frac{p}{2}}\\
& = & \int_\Omega \left(\frac{|F_2|^p \sigma^{p/2}}{|F_1|^{(1-\frac{p}{2})p}}\right)
\left(|F_1|^{(1-\frac{p}{2})p}\sigma^{1-\frac{p}{2}}\right)\\
\eea
\bea
& \leq & (\int_\Omega \left(\frac{|F_2|^p \sigma^{p/2}}{|F_1|^{(1-\frac{p}{2})p}}\right)^{\frac{1}{p/2}})^{p/2}\\
\eea
\bea
& * & (\int_\Omega\left(|F_1|^{(1-\frac{p}{2})p}\sigma^{1-\frac{p}{2}}\right)^{\frac{1}{1-\frac{p}{2}}})^{1-\frac{p}{2}}\\
\eea
\bea
& = & (\int_\Omega \left(\frac{|F_2|^2 \sigma}{F_1|^{2-p}}\right))^{p/2}\\
\eea
\bea
& * & (\int_\Omega |F_1|^{p}\sigma))^{1-\frac{p}{2}}\\
&\leq & ( C_0) ^{p/2} A^{1-\frac{p}{2}}\\
& = & A (C_0/A)^{p/2}\\
\eea

We can repeat the construction and find a sequence of holomorphic functions $F_n$ on $\Omega$
which extend $f$ and which satisfy
$$
\int_\Omega |F_n|^p\sigma \leq A_n,\; \mbox{ where}\;
A_n=A_{n-1}(C_0/A_{n-1})^{p/2}.
$$
We see that if $A_n>C$, then $A_{n+1}<A_n.$ If all $A_n>C$, then $A_n\rightarrow C.$

\end{proof}

Break Down of Ohsawa-Takegoshi for $p>2.$\\

We show here that for extensions from $L^q$ to $L^p$,$0<p,q\leq \infty,$ the only
case  possible is $0<p=q\leq 2.$ \\

First we consider unweighted extension on pseudoconvex domains contained in the unit ball.

\begin{example}
For small $\delta>0,$ let $\Omega=\{|z|<\delta,|w|<1\}.$ We extend the function $f=1$ on the $z$ axis. The $L^q$ norm is about $\delta^{1/q}.$ The extension $F=1$ which is optimal, has $L^p$ norm
about $\delta^{1/p}.$ For O-T to hold we need $p\leq q.$\\
\end{example}

We give another example of an unweighted extension of a domain in the unit ball which will show that $p\leq 2.$

{\bf Extensions of $z^n.$}\\

Pick $a_j$, $n+1$ distinct complex numbers. Consider the subdomain of the unit ball in $\mathbb C^2$
given by $|\Pi(z-a_j w)|<\delta.$
We investigate extension of $z^n$ from the intersection with the $z$ axis from $L^q$ to $L^p.$
The $L^q$ norm of $z^n$ is about 
\bea
(\int_{|z|<\delta^{1/(n+1)}}|z|^{nq})^{1/q} & = & (\int_0^{\delta^{1/(n+1)}}r^{nq+1})^{1/q}\\
& = & (\frac{(\delta^{1/(n+1)})^{nq+2}}{nq+2})^{1/q}\\
& = & c_n \delta^{\frac{nq+2}{(n+1)q}}\\
\eea

The extension to at least one of the lines must be at least on the order of $z^n$.
This gives the $L^p$ estimate at least $\delta^{2/p}.$

\begin{lemma}
For O-T to hold from $L^q$ to $L^p$ we need
$$
\frac{2}{p} \geq \frac{nq+2}{(n+1)q}
$$
 \end{lemma}

For this to hold for arbitrarily large $n,$ we must have $\frac{2}{p}\geq 1,$ i.e. $p\leq 2.$

These two examples show that for unweighted O-T to hold on pseudoconvex domains in the unit ball,
then we need $0<p\leq 2, p\leq q\leq \infty.$ We observe that in this case $L^q\subset L^p$, so
the theorem holds already as proved in the previous section. There is nothing new.

Finally, let us assume we allow as is usual, the domains to be arbitarily large in the $z$ direction:\\
Take a domain which is a ball of radius $R$ in the $z_1,...,z_{n-1}$ direction and
a disc of fixed small radius in the $z_n$ direction. Extend the function $1.$
Then the $L^p$ norm is about $R^{(2n-2)/p}.$

Hence O-T requires that
\bea
R^{(2n-2)/p} & \leq & R^{(2n-2)/q}\\
\frac{1}{p} \log R & \leq & \frac{1}{q}\log R\\
\eea
for all $R$. In particular, for large $R$ this shows that $q\leq p.$ So only the case $p=q$, $0<p\leq 2$ can work even for unweighted spaces.\\

Finally, we recall that the restrictions obtained work in weighted spaces as well.
(Use $\int|f|^p e^{-p\phi}$ instead of $\int |f|^pe^{-\phi}$ to avoid scaling problems as one sees from using a family of weights, $\phi_c=\phi+c, c\in \mathbb R.$)

\section{Appendix 4, The strong openness conjecture in $L^p.$}

We give an example where $L^2$ results extend to all $L^p,0<p<\infty.$
Let $\phi_j<0$ be a sequence of plurisubharmonic functions defined in a neighborhood of the origin 
in $\mathbb C^n.$ Suppose that $\phi_j\leq \phi_{j+1}$ and let $\phi=\lim \phi_j.$
Let $0<p<\infty.$ The strong openness conjecture in $L^p$ says that
if $f$ is a holomorphic function in a neighborhood $U$ of $0$ and $\int_U|f|^pe^{-\phi}<\infty$,
then there exists a $j$ and a neighborhood $V\subset U$ of $0$
so that $\int_V|f|^pe^{-\phi_j}<\infty.$
This conjecture has been proved for $p=2$ by Qi'an Guan and Xiangyu Zhou, Strong openness conjecture
arXiv:1311.3781v1 [math.CV] 15 Nov 2013.\\

We see here how the result for $p=2$ immediately proves the conjecture for all $p.$

\begin{theorem}
The strong openness conjecture holds in $L^p$ for $p<2.$
\end{theorem}

\begin{proof}
Let $\{\phi_j\}_{j=1}^\infty$ be a sequence of plurisubharmonic functions converging to a negative function $\phi$ on a neighborhood of $0$ in $\mathbb C^n,$
$\phi_j\nearrow \phi$.
Suppose that $\int_U|f|^pe^{-\phi}<\infty$ on some neighborhood $U$ of the origin, where $f$ is a holomorphic function. 
Let $v_j=\phi_j+(2-p)\log|f|, v=\phi+(2-p)\log|f|.$ The $v_j$ are plurisubharmonic functions 
and $v_j\nearrow v.$ 
Moreover,
$$
\int_U|f|^2 e^{-v}=\int_U |f|^2 e^{-\phi-(2-p)\log|f|}=\int |f|^2 |f|^{p-2}e^{-\phi}<\infty.
$$
Hence by the strong openness conjecture for $p=2$, it follows that for some large enough $j,$
and some smaller neighborhood $V$ of the origin,
$$
\int_V |f|^pe^{-\phi_j}= \int_V |f|^2e^{-(2-p)\log|f|-\phi_j}=\int_V |f|^2e^{-v_j}<\infty.
$$
\end{proof}

\begin{corollary}
The strong openness conjecture holds in $L^p$ for all $0<p<\infty.$
\end{corollary}

\begin{proof}
Choose $p$ and let $m$ be a positive integer so that $p/m<2.$
Suppose $\int_U |f|^pe^{-\phi}<\infty.$
Then $\int_U |f^m|^{p/m}e^{-\phi}<\infty.$  But since $p/m<2$ it follows that for large $j$
that $\int |f^m|^{p/m}e^{-\phi_j}<\infty,$ i.e. $\int |f|^{p}e^{-\phi_j}<\infty.$ 
\end{proof}

We can also vary $p$ and $\phi$ at the same time.

\begin{theorem}
Let $0<p<\infty$ and suppose that $\phi<0$ in a neighborhood $U$ of $0$ in $\mathbb C^n.$
Suppose that $p_j\nearrow p$ and $\phi_j\nearrow \phi$ is a sequence of plurisubharmonic functions on $U.$ If $f$ is a holomorphic function on $U$ such that $\int_U |f|^pe^{-\phi}<\infty$, then
there exists $j$ and a smaller neighborhood $V$ of $0$ so that
$$
\int_V |f|^{p_j} e^{-\phi_j}<\infty.
$$
\end{theorem}

\begin{proof}
We can assume that $|f|<1.$
Consider the weights $v_j=(p-p_j)\log |f|+\phi_j.$ Then $v_j\nearrow \phi$ except for possibly the zero set of $f.$  Hence by the strong openness conjecture for $L^p$, $\int_V |f|^pe^{-(p-p_j)\log|f|-\phi_j}<\infty$ for large $j$.
Hence $\int_V |f|^{p_j}e^{-\phi_j}<\infty.$\\
\end{proof}

We have a similar result for decreasing sequence in $p_j.$

\begin{theorem}
Let $0<p<\infty$ and suppose that $\phi<0$ in a neighborhood $U$ of $0$ in $\mathbb C^n.$
Suppose that $p_j\searrow p$ and $\phi_j\nearrow \phi$ is a sequence of plurisubharmonic functions on $U.$ If $f$ is a holomorphic function on $U$ such that $\int_U |f|^pe^{-p\phi}<\infty$, then
there exists $j$ and a smaller neighborhood $V$ of $0$ so that
$$
\int_V |f|^{p_j} e^{-p_j\phi_j}<\infty.
$$
\end{theorem}

\begin{proof}
We can assume that $|f|<1.$
We have that $p_j\phi_j\nearrow p\phi.$
Suppose that $\int_U|f|^pe^{-p\phi}<\infty.$
Hence by the strong openness conjecture in $L^p$, we see there is a $j$ and a $V$ so that
$\int_V|f|^p e^{-p_j\phi_j}<\infty.$
Since $|f|^{p_j}\leq |f|^p$, the result follows.
\end{proof}


\begin{thebibliography}{}

\bibitem[H]{H}
Hormander, Lars,: An introduction to complex analysis in several variables.
D. Van Nostrand. (1966).

\end{thebibliography}
\end{document}